\documentclass[a4paper,10pt]{amsart}
\usepackage[english]{babel}
\usepackage[english]{translator}
\usepackage[T1]{fontenc}
\usepackage[utf8]{inputenc}
\usepackage{amsmath}
\usepackage{amssymb}
\usepackage{amsfonts}
\usepackage{amstext}
\usepackage{amsthm}
\usepackage{bbm}

\usepackage{graphicx}
\usepackage{color}
\usepackage{psfrag}
\usepackage{subfigure}
\usepackage{float}
\usepackage{marvosym}
\usepackage{enumitem}
\usepackage{hyperref}
\usepackage[
  hmarginratio={1:1},     % equal left and right margins
  vmarginratio={1:1},     % equal top and bottom margins
  textwidth=420pt,        % new text width
  heightrounded,          % always useful
]{geometry}

\hypersetup{
  pdffitwindow=false,
  pdfhighlight=/O,
  pdfnewwindow,
  colorlinks=true,
  citecolor=red,            % Farbe der Zitate(Buchverweise)
  linkcolor=blue,             % Farbe der Hyperref-Links
  menucolor=blue,            % color of Acrobat Reader menu buttons
  urlcolor=blue,             % Farbe der Internetadressverweise
  pdfpagemode=UseOutlines,
  bookmarksnumbered=true,
  linktocpage,
  pdftitle={Spatial Decay of Rotating Waves in Reaction-Diffusion Systems},
  pdfsubject={Spatial Decay of Rotating Waves in Reaction-Diffusion Systems},
  pdfauthor={Wolf-J"urgen Beyn, Denny Otten},
  pdfkeywords={Rotating waves, spatial exponential decay, Ornstein-Uhlenbeck operator, exponentially weighted resolvent estimates, reaction-diffusion equations.},
  pdfcreator={latex,bibtex,dvips,ps2pdf},
  pdfproducer={LaTeX mit hyperref}
}

\newcommand{\D}{\mathcal{D}}              % Definitionsbereich
\newcommand{\K}{\mathbb{K}}               % Koerper
\newcommand{\C}{\mathbb{C}}               % komplexe Zahlen
\newcommand{\R}{\mathbb{R}}               % reelle Zahlen
\newcommand{\Z}{\mathbb{Z}}                % ganze Zahlen
\newcommand{\N}{\mathbb{N}}                % natuerliche Zahlen
\renewcommand{\S}{\mathcal{S}}              % Schwartz Raum
\renewcommand{\Re}{\mathrm{Re}\,}          % Realteil
\renewcommand{\Im}{\mathrm{Im}\,}          % Imaginaerteil
\renewcommand{\L}{\mathcal{L}}             % L Operator
\newcommand{\A}{\mathcal{A}}             % A Generator of a strongly continuous semigroup
\newcommand{\B}{\mathcal{B}}             % A Generator of a strongly continuous semigroup
\newcommand{\Q}{\mathcal{Q}}             % A Generator of a strongly continuous semigroup
\newcommand{\diag}{\mathrm{diag}}          % Diagonalmatrix
\newcommand{\one}{\mathbbm{1}}             % Indikatorfunktion
\newcommand{\cond}{\mathrm{cond}}          % matrix condition number

\DeclareMathOperator{\esssup}{ess\,sup\,}

\newcommand{\amin}{a_{\mathrm{min}}} 
\newcommand{\amax}{a_{\mathrm{max}}} 
\newcommand{\azero}{a_0}
\newcommand{\aone}{a_1}
 
\newcommand{\bzero}{b_0}

\newcommand{\begriff}[1]{\textbf{#1}}

\makeatletter
\renewcommand{\@secnumfont}{\bfseries}
\makeatother
\makeatletter
  \def\section{\@startsection{section}{1}%
    \z@{.7\linespacing\@plus\linespacing}{.5\linespacing}%
    {\normalfont\LARGE\bfseries}}
\makeatother
\makeatletter
\def\@seccntformat#1{%
  \protect\textup{%
    \protect\@secnumfont
    \expandafter\protect\csname format#1\endcsname % <--- added
    \csname the#1\endcsname
    \protect\@secnumpunct
  }%
}

\newcommand{\sect}
{
  \setcounter{equation}{0}
  \setcounter{figure}{0}
  \section
}
\newcommand{\enum}[1]{\textnormal{(\textbf{#1})}}

\theoremstyle{plain}
\newtheorem{definition}{Definition}[section]
\newtheorem{theorem}[definition]{Theorem}
\newtheorem{lemma}[definition]{Lemma}
\newtheorem{corollary}[definition]{Corollary}
\newtheorem{assumption}[definition]{Assumption}
\newtheorem{remark}[definition]{Remark}

\theoremstyle{definition}

\allowdisplaybreaks

\begin{document}
%  -----------
% |   Title   |
%  -----------
\title[Spatial Decay of Rotating Waves in\\Reaction Diffusion Systems]{Spatial Decay of Rotating Waves in\\Reaction Diffusion Systems}
%\maketitle
\setlength{\parindent}{0pt}
\vspace*{0.75cm}
\begin{center}
%\normalfont\LARGE\bfseries{\shorttitle}
\normalfont\huge\bfseries{\shorttitle}\\
\vspace*{0.25cm}
%\Large\bfseries\MakeUppercase{\shorttitle}
\end{center}

%  -------------
% |   Authors   |
%  -------------
\vspace*{0.5cm}
\noindent
\hspace*{4.2cm}
% First Author 
\textbf{Wolf-J{\"u}rgen Beyn}\footnote[1]{e-mail: \textcolor{blue}{beyn@math.uni-bielefeld.de}, phone: \textcolor{blue}{+49 (0)521 106 4798}, \\
                                          fax: \textcolor{blue}{+49 (0)521 106 6498}, homepage: \url{http://www.math.uni-bielefeld.de/~beyn/AG\_Numerik/}.} \\
\hspace*{4.2cm}
Department of Mathematics \\
\hspace*{4.2cm}
Bielefeld University \\
\hspace*{4.2cm}
33501 Bielefeld \\
\hspace*{4.2cm}
Germany

\vspace*{0.5cm}
\noindent
\hspace*{4.2cm}
% Second Author
\textbf{Denny Otten}\footnote[2]{e-mail: \textcolor{blue}{dotten@math.uni-bielefeld.de}, phone: \textcolor{blue}{+49 (0)521 106 4784}, \\
                                 fax: \textcolor{blue}{+49 (0)521 106 6498}, homepage: \url{http://www.math.uni-bielefeld.de/~dotten/}. \\
                                 supported by CRC 701 'Spectral Structures and Topological Methods in Mathematics'} \\
\hspace*{4.2cm}
Department of Mathematics \\
\hspace*{4.2cm}
Bielefeld University \\
\hspace*{4.2cm}
33501 Bielefeld \\
\hspace*{4.2cm}
Germany

%  ----------
% |   Date   |
%  ----------
\vspace*{0.8cm}
\noindent
\hspace*{4.2cm}
Date: \today
\normalparindent=12pt

%  --------------
% |   Abstract   |
%  --------------
\vspace{0.4cm}
%\begin{abstract}
\noindent
\begin{center}
\begin{minipage}{0.9\textwidth}
  {\small
  \textbf{Abstract.} 
  In this paper we study nonlinear problems for Ornstein-Uhlenbeck operators
  \begin{align*}
    A\triangle v(x) + \left\langle Sx,\nabla v(x)\right\rangle + f(v(x)) = 0,\,x\in\mathbb{R}^d,\,d\geqslant 2,
  \end{align*}
  where the matrix $A\in\R^{N,N}$ is diagonalizable and has eigenvalues with
  positive real part, the map $f:\R^N \rightarrow \R^N$ is
  sufficiently smooth and the matrix $S\in\mathbb{R}^{d,d}$ in the unbounded drift term is skew-symmetric.
  Nonlinear problems of this form appear as stationary equations for
  rotating waves in time-dependent reaction 
  diffusion systems. 
  We prove under appropriate conditions that every bounded classical solution $v_{\star}$ 
  of the nonlinear problem, which falls below a certain threshold at infinity, already decays exponentially in space, in the sense that $v_{\star}$ belongs to  
  an exponentially weighted Sobolev space $ W^{1,p}_{\theta}(\R^d,\R^N)$.
  % The main idea of proof is to decompose the nonlinearity into
  % \begin{align*}
  %   f(v(x)) = \left(Df(0) + Q_{\mathrm{s}}(x) + Q_{\mathrm{c}}(x)\right)v(x)
  % \end{align*}
  % where $Df(0)$ is stable, $Q_{\mathrm{s}}$ small w.r.t. $\left\|\cdot\right\|_{L^{\infty}}$ and $Q_{\mathrm{c}}(x)$ compactly supported on $\R^d$. 
  Several extensions of this basic result are presented: to complex-valued systems, to exponential decay in higher order Sobolev spaces and to pointwise estimates. 
  We also prove that every bounded classical solution $v$ of the eigenvalue problem
  \begin{align*}
    A\triangle v(x) + \left\langle Sx,\nabla v(x)\right\rangle + Df(v_{\star}(x))v(x) = \lambda v(x),\,x\in\mathbb{R}^d,\,d\geqslant 2,
  \end{align*}
  decays exponentially in space, provided $\Re \lambda$ lies to the right
  of the essential spectrum. As an application we analyze spinning soliton solutions which occur in the 
  Ginzburg-Landau equation.
  Our results form the basis for investigating nonlinear stability of rotating waves in higher space dimensions and truncations to bounded domains. 
  %In particular, we provide conditions 
  %that justifies the decay assumption postulated in an earlier work of Beyn und Lorenz for investigating nonlinear stability in two space dimensions. 
  }
\end{minipage}
\end{center}
%\end{abstract}

%  ---------------
% |   Key Words   |
%  ---------------

%  --------------------------------
% |   AMS Subject Classification   |
%  --------------------------------
% Reaction diffusion systems: 35K57 (or: 35J47)
% Asymptotic behavior:        35B40
% Perturbation theory:        47A55
% Eigenvalue problem:         35Pxx (or: 47A10, 35J47)
% Ginzburg Landau equation:   35Q56
% Application num. analysis:  47N40

%  -----------------------
% |   Table of contents   |
%  -----------------------
%\tableofcontents

%---------------------------------------------------------------------------------------------------------------------------------------------------
%
%  SECTION 1: (Introduction)
%
%---------------------------------------------------------------------------------------------------------------------------------------------------
\sect{Introduction}
\label{sec:1}
%---------------------------------------------------------------------------------------------------------------------------------------------------

In the present paper we study systems of reaction-diffusion equations
\begin{equation}
  \begin{aligned}
  \label{equ:NonlinearParabolicProblem2}
  \begin{split}
    u_t(x,t) & = A\triangle u(x,t) + f(u(x,t))   ,\,t>0,\,x\in\R^d,\,d\geqslant 2,\\
      u(x,0) & = u_0(x) \qquad\qquad\qquad\quad\;,\,t=0,\,x\in\R^d,
  \end{split}
  \end{aligned}
\end{equation}
where $A\in\R^{N,N}$ is a diffusion matrix, $f:\R^N\rightarrow\R^N$ is a sufficiently smooth nonlinearity, $u_0:\R^d\rightarrow\R^N$ are the initial data and 
$u:\R^d\times[0,\infty)\rightarrow\R^N$ denotes a vector-valued solution.

We are mainly interested in rotating wave solutions of \eqref{equ:NonlinearParabolicProblem2} which are of the form
\begin{align}
  \label{equ:RotatingWaveIntro}
  u_{\star}(x,t)=v_{\star}(e^{-tS}x),\,t\geqslant 0,\,x\in\R^d,\,d\geqslant 2
\end{align}
with space-dependent profile $v_{\star}:\R^d\rightarrow\R^N$ and skew-symmetric matrix $S\in\R^{d,d}$.
The skew-symmetry of $S$ implies that $e^{-tS}$ describes a  rotation 
in $\R^d$, and hence $u_{\star}$ is a solution rotating at constant velocity while maintaining its shape determined by $v_{\star}$. The profile $v_{\star}$ is called (exponentially) localized, if it tends (exponentially) to some constant vector $v_{\infty}\in\R^N$ 
as $|x|\to\infty$. 
%Note that rotating waves always come in families: If $u_{\star}$ from \eqref{equ:RotatingWaveIntro} solves \eqref{equ:NonlinearParabolicProblem2}, 
%then so does the function $v_{\star}(e^{-tS}(R^{-1}(x-\tau)))$ for every $(R,\tau)\in\SE(d)$, where $\SE(d)$ denotes the special Euclidean group.

Transforming \eqref{equ:NonlinearParabolicProblem2} via $u(x,t)=v(e^{-tS}x,t)$ into a co-rotating frame yields the evolution equation
\begin{equation}
  \begin{aligned}
  \label{equ:RotatingFrame2}
  \begin{split}
    v_t(x,t) =& A\triangle v(x,t) + \left\langle Sx,\nabla v(x,t)\right\rangle + f(v(x,t)) ,\,t>0,\,x\in\R^d,\,d\geqslant 2,\\
      v(x,0) =& u_0(x)                       \qquad\qquad\qquad\qquad\qquad\qquad\qquad\;\,,\,t=0,\,x\in\R^d.
  \end{split}
  \end{aligned}
\end{equation}
The diffusion and drift term are given by
\begin{align} \label{equ:diffdrift}
  A\triangle v(x):=A\sum_{i=1}^{d}\frac{\partial^2}{\partial x_i^2}v(x)\quad\text{and}\quad
  \left\langle Sx,\nabla v(x)\right\rangle:=\sum_{i=1}^{d}\sum_{j=1}^{d}S_{ij}x_j D_i v(x).
\end{align}
The pattern $v_{\star}$ itself appears as a stationary solution of \eqref{equ:RotatingFrame2}, i.e. $v_{\star}$ solves the steady state problem
\begin{align}
  \label{equ:NonlinearSteadyStateProblem2}
  A\triangle v_{\star}(x)+\left\langle Sx,\nabla v_{\star}(x)\right\rangle+f(v_{\star}(x))=0,\,x\in\R^d,\,d\geqslant 2.
\end{align}
We may write \eqref{equ:NonlinearSteadyStateProblem2} as $[\L_0 v_{\star}](x)+f(v_{\star}(x))=0$ by introducing the Ornstein-Uhlenbeck operator
\begin{align} \label{equ:OUintro}
  \left[\L_0 v\right](x) := A\triangle v(x)+\left\langle Sx,\nabla v(x)\right\rangle,\,x\in\R^d.
\end{align}
 
By the skew-symmetry of $S$ we can write the drift term in terms of
angular derivatives as follows
\begin{align} \label{equ:angderiv}
  \left\langle Sx,\nabla v(x)\right\rangle = \sum_{i=1}^{d-1}\sum_{j=i+1}^{d}S_{ij}\left(x_j\frac{\partial}{\partial x_i}-x_i\frac{\partial}{\partial x_j}\right)v(x).
\end{align}

The aim of this paper is to derive suitable conditions guaranteeing that every localized rotating wave of \eqref{equ:NonlinearParabolicProblem2} 
is already exponentially localized. More precisely, the main theorem  states the following: if the difference $v_{\star}-v_{\infty}$ of a  rotating wave to its far field value falls below a certain threshold at 
infinity, then it decays exponentially in space. The decay is specified
by showing  that $v_{\star}-v_{\infty}$ belongs to some exponentially weighted Sobolev space $W^{1,p}_{\theta}(\R^d,\R^N)$, $1<p<\infty$. Our key
assumption requires all eigenvalues of the Jacobian $Df(v_{\infty})$ to have
negative real part.

We extend this result to complex-valued systems and then apply it to prove exponential decay of localized spinning solitons 
arising in the cubic-quintic complex Ginzburg-Landau equation (QCGL), \cite{CrasovanMalomedMihalache2001}. Figure \ref{fig:QCGLRotatingWavesIntroduction}(a) 
shows the real part of a spinning soliton $v_{\star}$ in two space dimensions,  
while Figure \ref{fig:QCGLRotatingWavesIntroduction}(b) shows the isosurfaces of the real part of a spinning soliton in three space dimensions. 
Both of these rotating waves are exponentially localized, as our results will show. Two nonlocalized rotating waves are illustrated in Figure \ref{fig:QCGLRotatingWavesIntroduction}(c)-(d).
Figure \ref{fig:QCGLRotatingWavesIntroduction}(c) shows the real part of a spiral wave in two space dimensions and Figure \ref{fig:QCGLRotatingWavesIntroduction}(d) 
the isosurfaces of the real part of an untwisted scroll wave. In Section \ref{sec:6} below we will discuss this example in more detail.

\begin{figure}[H]
  \centering
  \subfigure[]{\includegraphics[page=1,height=3.3cm] {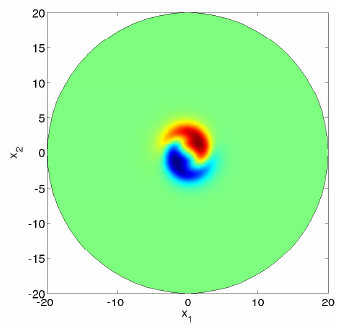}       \label{fig:QCGLSpinningSoliton2DRealPartIntroduction}}
  \subfigure[]{\includegraphics[page=2,height=3.3cm] {Images.pdf}       \label{fig:QCGLSpinningSoliton3DRealPartIntroduction}}
  \subfigure[]{\includegraphics[page=3,height=3.3cm] {Images.pdf} \label{fig:QCGLRotatingSpiralSoliton2DRealPart}}
  \subfigure[]{\includegraphics[page=4,height=3.3cm] {Images.pdf}     \label{fig:LambdaOmegaScrollWave3DRealPart}}
  \caption{Rotating waves of QCGL \eqref{equ:ComplexQuinticCubicGinzburgLandauEquation}. (a) Spinning solitons for $d=2$ with colorbar reaching from $-1.6$ (blue) to $1.6$ (red), 
           (b) spinning soliton for $d=3$ with isosurfaces at values $-0.5$ (blue) and $0.5$ (red), (c) spiral wave for $d=2$ with colorbar reaching from $-1.7$ (blue) to $1.7$ (red), 
           and (d) scroll wave for $d=3$ with isosurfaces at values $-0.5$ (blue) and $0.5$ (red)}
  \label{fig:QCGLRotatingWavesIntroduction}
\end{figure}  

An important issue is to investigate \textit{nonlinear stability of rotating waves} (more precisely, \textit{stability with asymptotic phase}) in 
reaction diffusion systems, see \cite{BeynLorenz2008}. A well known task is to derive nonlinear stability from linear stability of the linearized operator
\begin{align} \label{equ:linop}
  \left[\L v\right](x) := \left[\L_0 v\right](x) + Df\left(v_{\star}(x)\right)v(x),\,x\in\R^d.
\end{align}
By \textit{linear stability} (also called \textit{strong spectral stability}) we mean that the essential spectrum and the isolated eigenvalues of $\L$ lie strictly 
to the left of the imaginary axis, except for those on the imaginary axis
caused by Euclidean equivariance, \cite[Chp.9]{Otten2014}. This
requires to study isolated eigenvalues $\lambda\in\C$ of the problem
\begin{align}
  \label{equ:EigenvalueProblemIntro}
  \left[\left(\lambda I-\L\right)v\right](x) = 0,\,x\in\R^d.
\end{align}

A further aim of this paper is to prove that every bounded eigenfunction $v$ of the linearized operator $\L$ decays exponentially in space, provided the
real parts of the associated (isolated) eigenvalues $\lambda$ lie to the right of the
essential spectrum. To be more precise, we  show that for such values of
$\lambda$, every bounded classical solution $v$ of 
the eigenvalue problem \eqref{equ:EigenvalueProblemIntro} belongs to some exponentially 
weighted Sobolev space $W^{1,p}_{\theta}(\R^d,\R^N)$ for some $1<p<\infty$. In particular, we prove that the eigenfunction $v(x)=\left\langle Sx,\nabla v_{\star}(x)\right\rangle$ 
associated to the eigenvalue $\lambda=0$ decays exponentially in space.

A nonlinear stability result for two dimensional localized rotating patterns was proved by Beyn and Lorenz in \cite{BeynLorenz2008}. Their proof requires three 
essential assumptions: The matrix $Df(v_{\infty})$ is stable, meaning that all its eigenvalues have a negative real part. Moreover, strong spectral stability in 
the sense above is assumed. And finally, the profile $v_{\star}$ of the rotating wave and its derivatives up to order $2$ decay to zero at infinity. 
%The profile $v_{\star}$ of the rotating wave and their partial derivatives up to order $2$ are localized in the above sense. Furthermore, 
%the matrix $Df(v_{\infty})$ is stable, meaning that all its eigenvalues have a negative real part. And finally, strong spectral stability in 
%the sense above is assumed. 
Their analysis shows that the decay of the rotating wave itself and the spectrum of the linearization are both crucial for investigating
nonlinear stability. A corresponding result on nonlinear stability of nonlocalized rotating waves, such as spiral 
waves and scroll waves, is still an open problem. The difficulty is related to the fact that the essential spectrum  touches the imaginary
axis at infinitely many points. The spectrum of the linearization at 
(nonlocalized) spiral waves is well-known and has been extensively studied 
by Sandstede, Scheel and Fiedler in \cite{FiedlerScheel2003,SandstedeScheel2000,SandstedeScheel2001}.

For numerical computations it is essential to truncate the equations \eqref{equ:NonlinearParabolicProblem2}, \eqref{equ:RotatingFrame2} and \eqref{equ:EigenvalueProblemIntro} 
to a bounded domain, so that standard approximations, e.g. with finite elements, apply. 
The truncation error arising in this  process, depends on the boundary conditions. Assuming that a rotating wave is (exponentially) localized, 
we can expect the truncation error to be (exponentially) small as well. For this reason, the \textit{exponential decay of rotating waves} plays a fundamental 
role when estimating errors caused by \textit{approximations of rotating waves on bounded domains}. 

We consider our results on the decay  of rotating waves for \eqref{equ:NonlinearParabolicProblem2} on the whole $\R^d$ as a first step
in studying such truncation errors. Despite numerous numerical simulations
of spiral behavior on bounded domains, a rigorous analysis of the 
errors caused by spatial truncation seems not to be available.

We emphasize that the results from Section \ref{sec:3}-\ref{sec:6} are extensions of the results from the PhD thesis \cite{Otten2014}. One major
improvement refers to the fact that our main result 
Theorem \ref{thm:NonlinearOrnsteinUhlenbeckSteadyState} avoids the additional assumption $v_{\star}\in L^p(\R^d,\R^N)$ from \cite[Theorem 1.8]{Otten2014}
by using ideas from the work \cite{BeynLorenz2014}.

% Literatur ergänzen: (Insbesondere \cite{ZelikMielke2009} wegen der gewichtsfunktion erwähnen).
%Dissertation von Sahitya Konda: \cite{Konda2015}

%---------------------------------------------------------------------------------------------------------------------------------------------------
%
%  SECTION 2: (Assumptions and main result)
%
%---------------------------------------------------------------------------------------------------------------------------------------------------
\sect{Assumptions and main result}
\label{sec:2}
%---------------------------------------------------------------------------------------------------------------------------------------------------

%----------------------------------------------------------------------------------
% SUBSECTION 2.1: (Assumptions and main result).
%----------------------------------------------------------------------------------
\subsection{Assumptions and main result}
\label{subsec:2.1}
%----------------------------------------------------------------------------------

Consider the steady state problem
\begin{align}
  \label{equ:NonlinearSteadyStateProblem}
  A\triangle v(x)+\left\langle Sx,\nabla v(x)\right\rangle+f(v(x))=0,\,x\in\R^d,\,d\geqslant 2,
\end{align}
with \begriff{diffusion matrix} $A\in\K^{N,N}$ and a function $f:\K^N\rightarrow\K^N$ for $\K\in\{\R,\C\}$.
Recall the Ornstein-Uhlenbeck operator from \eqref{equ:OUintro} with drift and diffusion term specified in \eqref{equ:diffdrift}.

We define a rotating wave  $u_{\star}$  as follows:
\begin{definition}\label{def:RotatingWave}
  A function $u_{\star}:\R^d\times[0,\infty)\rightarrow\K^N$ is called a \begriff{rotating wave} (or \begriff{rotating pattern}) 
  if it has the form
  \begin{align}
    \label{equ:RotatingWave}
    u_{\star}(x,t)=v_{\star}(e^{-tS}(x-x_{\star})),\,x\in\R^d,\,t\in[0,\infty),
  \end{align}
  with \begriff{profile} (or \begriff{pattern}) $v_{\star}:\R^d\rightarrow\K^N$, a skew-symmetric matrix $0\neq S\in\R^{d,d}$ and 
  $x_{\star}\in\R^d$. A rotating wave $u_{\star}$ is called \begriff{localized} (\begriff{exponentially localized with decay rate $\eta$})
  if it satisfies 
  \begin{align}\label{eq:localized}
    \lim_{|x|\to\infty}e^{\eta|x|}\left|v_{\star}(x)-v_{\infty}\right|=0\text{ for some $v_{\infty}\in\K^N$}
  \end{align}
  and for $\eta=0$ ($\eta>0$). It is called \begriff{nonlocalized}, if it is not localized in the sense above.
\end{definition}
The vector $x_{\star}\in\R^d$ can be considered as the center of rotation for $d=2$ and as the support vector of the axis of rotation for $d=3$. In case $d\in\{2,3\}$, $S$ can 
be considered as the angular velocity tensor associated to the angular velocity vector $\omega\in\R^{\frac{d(d-1)}{2}}$ containing $S_{ij}$, $i=1,\dots,d-1$, $j=i+1,\ldots,d$. 
Some examples of rotating patterns are illustrated in Figure \ref{fig:QCGLRotatingWavesIntroduction} and will be treated in Section \ref{sec:6} below.

%A transformation into a \begriff{co-rotating frame} shows that if $u(x,t)$ solves \eqref{equ:NonlinearParabolicProblem2} 
%then $v(x,t)=u(e^{tS}x+x_{\star},t)$ solves
%\begin{align}
%  \label{equ:RotatingFrame}
%  \begin{split}
%    v_t(x,t) =& A\triangle v(x,t) + \left\langle Sx,\nabla v(x,t)\right\rangle + f(v(x,t)) ,\,t>0,\,x\in\R^d,\,d\geqslant 2,\\
%      v(x,0) =& u_0(x)                       \qquad\qquad\qquad\qquad\qquad\qquad\qquad\;\,,\,t=0,\,x\in\R^d,
%  \end{split}
%\end{align}
%where the drift term is given by \eqref{equ:diffdrift}. Conversely, if $v(x,t)$ solves \eqref{equ:RotatingFrame} then $u(x,t)=v(e^{-tS}(x-x_{\star}),t)$ solves 
%\eqref{equ:NonlinearParabolicProblem2}.

%Note that $v_{\star}$ is a stationary solution of \eqref{equ:RotatingFrame}, meaning that $v_{\star}$ solves the nonlinear problem \eqref{equ:NonlinearSteadyStateProblem}. 

In the following we will impose various restrictions on the matrix $A$:
\begin{assumption}
  \label{ass:Assumption1}
  For $A\in\K^{N,N}$ with $\K\in\{\R,\C\}$ and $1 < p <\infty$  consider
the conditions
  \begin{flalign}
    &\text{$A$ is diagonalizable (over $\C$)},         \tag{A1}\label{cond:A1} &\\
    &\Re\sigma(A)>0, \tag{A2}\label{cond:A2} &\\
    &\Re\left\langle w,Aw\right\rangle\geqslant\beta_A\;\forall\,w\in\K^N,\,|w|=1\text{ for some $\beta_A>0$,} \tag{A3}\label{cond:A3} &\\
    &\text{There exists $\gamma_A>0$ such that} \tag{A4}\label{cond:A4DC} &\\
    &\quad|z|^2\Re\left\langle w,Aw\right\rangle + (p-2)\Re\left\langle w,z\right\rangle\Re\left\langle z,Aw\right\rangle\geqslant\gamma_A |z|^2|w|^2\;\forall\,z,w\in\K^N \nonumber &\\
        &\text{Case ($N=1$, $\K=\R$): $A=a>0$,} \tag{A5}\label{cond:A4} \\
    &\text{Cases ($N\geqslant 2$, $\K=\R$) and ($N\geqslant 1$, $\K=\C$):} \nonumber 
    \quad A\text{ invertible and } \mu_1(A)>\frac{|p-2|}{p}. & 
  \end{flalign}
\end{assumption}

Assumption \eqref{cond:A1} is a \textbf{system condition} and ensures that all results for scalar equations can be extended to system cases. This condition 
is independent of \eqref{cond:A2}-\eqref{cond:A4} and is used in \cite{Otten2014,Otten2014a} to derive an explicit formula for the heat kernel 
of $\L_{0}$. A typical case where \eqref{cond:A1} holds, is a scalar complex-valued equation when transformed into a  real-valued system of dimension $2$ . 
The \textbf{positivity condition} \eqref{cond:A2} guarantees that the diffusion part $A\triangle$ 
is an elliptic operator. All eigenvalues $\lambda\in \sigma(A)$ of $A$ lie in the open right half-plane $\{\lambda\in\C\mid\Re\lambda>0\}$. 
 Condition \eqref{cond:A2} guarantees that $A^{-1}$ exists and that $-A$ is a stable matrix. 
The \textbf{strict accretivity condition} \eqref{cond:A3} is more restrictive than \eqref{cond:A2}. In \eqref{cond:A3} we use $\left\langle u,v\right\rangle:=\overline{u}^T v$ to denote
the standard inner product on $\K^N$. Recall that condition \eqref{cond:A2} is satisfied iff there exists an inner product
  $\left[\cdot,\cdot\right]$  and some $\beta_A>0$ such that  $\Re\left[w,Aw\right]\geqslant\beta_A$ forall $w\in\K^N$ with $\left[w,w\right]=1$.
 Condition \eqref{cond:A3} ensures that the differential operator $\L_{0}$ 
is closed on its (local) domain $\D^p_{\mathrm{loc}}(\L_0)$, see Theorem \ref{thm:LpMaximalDomainPart1} below. The \textbf{$L^p$-dissipativity condition} \eqref{cond:A4DC} is more restrictive than \eqref{cond:A3} 
and imposes additional requirements on the spectrum of $A$. This condition, which comes originally from \cite{CialdeaMazya2005,CialdeaMazya2009}, is used 
in \cite{Otten2014,Otten2015a} to prove $L^p$-resolvent estimates for $\L_{0}$. 
A geometrical meaning of \eqref{cond:A4DC} can be given in terms of the antieigenvalues of the diffusion matrix $A$. In \cite{Otten2014,Otten2015b}, it is proved 
that condition \eqref{cond:A4DC} is equivalent to the \textbf{$L^p$-antieigenvalue condition} \eqref{cond:A4}. Condition \eqref{cond:A4} requires that the 
\textbf{first antieigenvalue of $A$} (see \cite{Gustafson1968,Gustafson2012}), defined by
\begin{align*}
  \mu_1(A):=\inf_{\substack{w\in\K^N\\w\neq 0\\Aw\neq 0}}\frac{\Re\left\langle w,Aw\right\rangle}{|w||Aw|}
           =\inf_{\substack{w\in\K^N\\|w|=1\\Aw\neq 0}}\frac{\Re\left\langle w,Aw\right\rangle}{|Aw|},
\end{align*}
is bounded from below by a non-negative $p$-dependent constant. Condition \eqref{cond:A4} is also equivalent to the following $p$-dependent upper bound for the 
(\begriff{real}) \begriff{angle of $A$} (cf. \cite{Gustafson1968}),
\begin{align*}
  \Phi_{\R}(A):=\cos^{-1}\left(\mu_1(A)\right)<\cos^{-1}\left(\frac{|p-2|}{p}\right)\in\big(0,\frac{\pi}{2}\big],\quad 1<p<\infty.
\end{align*}
Therefore, the first antieigenvalue $\mu_1(A)$ can be considered as the cosine of the maximal (real) turning angle of vectors mapped by the matrix $A$. Some special cases 
in which the first antieigenvalue can be given explicitly are treated in \cite{Otten2015b}. We summarize the relationship of \eqref{cond:A2}--\eqref{cond:A4}:
\begin{align}
  \label{equ:RelationAssumptions}
  A\text{ invertible}\Longleftarrow\text{\eqref{cond:A2}}\Longleftarrow\text{\eqref{cond:A3}}\Longleftarrow\text{\eqref{cond:A4DC}}\Longleftrightarrow\text{\eqref{cond:A4}}.
\end{align}

We continue with the \textbf{rotational condition} \eqref{cond:A5} 
and a \textbf{smoothness condition} \eqref{cond:A6},
\begin{assumption}
  \label{ass:Assumption2}
  The matrix $S\in\R^{d,d}$ satisfies
  \begin{flalign}
    &\text{$S$ is skew-symmetric, i.e. $S=-S^T$.}        \tag{A6}\label{cond:A5} &
  \end{flalign}
\end{assumption}
\begin{assumption}
  \label{ass:Assumption3}
  The function $f:\R^N\rightarrow\R^N$ satisfies
  \begin{flalign}
    &f\in C^2(\R^N,\R^N).                             \tag{A7}\label{cond:A6} &
  \end{flalign}
\end{assumption}
Later on we apply our results to complex-valued nonlinearities of the form
\begin{equation} \label{equ:complexversion}
  f:\C^N\rightarrow\C^N,\quad f(u)=g\left(|u|^2\right)u,
\end{equation}
where $g:\R\rightarrow\C^{N,N}$ is a sufficiently smooth function. Such nonlinearities arise for example in Ginzburg-Landau equations, Schr\"odinger equations, 
$\lambda-\omega$ systems and many other equations from physical sciences, see Section \ref{sec:6}. 
Note, that in this case, the function $f$ is not holomorphic in $\C$, but its real-valued version in $\R^2$ satisfies \eqref{cond:A6} if $g$ is in $C^2$. For 
differentiable functions $f:\R^N\rightarrow\R^N$ we denote by $Df$ 
the Jacobian matrix in the real sense.
\begin{assumption}
  \label{ass:Assumption4}
  For $v_{\infty}\in\R^N$ consider the following conditions:
  \begin{flalign}
    &f(v_{\infty})=0,                                                                     \tag{A8}\label{cond:A7} &\\
    &\text{$A,Df(v_{\infty})\in\R^{N,N}$ are simultaneously diagonalizable (over $\C$),}  \tag{A9}\label{cond:A8} &\\
    &\Re\sigma\left(Df(v_{\infty})\right)<0,                                              \tag{A10}\label{cond:A9} &\\
    &\Re\left\langle w,-Df(v_{\infty})w\right\rangle\geqslant\beta_{-Df(v_{\infty})}\;\forall\,w\in\K^N,\,|w|=1\text{ for some $\beta_{\infty}:=\beta_{-Df(v_{\infty})}>0$.} \tag{A11}\label{cond:A10} &
  \end{flalign}
\end{assumption}
The \textbf{constant asymptotic state condition} \eqref{cond:A7} requires $v_{\infty}$ to be a steady state of the nonlinear equation. The \textbf{system condition} 
\eqref{cond:A8} is an extension of Assumption \eqref{cond:A1}, and
the \textbf{coercivity condition}  \eqref{cond:A10} is again  more restrictive
than the \textbf{spectral condition} \eqref{cond:A9}.

\begin{definition}\label{def:ClassicalSolution}
A function $v_{\star}:\R^d\rightarrow\K^N$ is called a \begriff{classical solution of \eqref{equ:NonlinearSteadyStateProblem}} if 
\begin{align}
  \label{equ:SmoothnessWeakSolution}
  v_{\star}\in C^2(\R^d,\K^N)
\end{align}
and $v_{\star}$ solves \eqref{equ:NonlinearSteadyStateProblem} pointwise.
\end{definition}

Later on, we will consider classical solutions $v_{\star}$ which are even
bounded, i.e. $v_{\star} \in C_{\mathrm{b}}(\R^d,\K^N)$. 
For matrices $C\in\K^{N,N}$ with spectrum $\sigma(C)$ we denote by $\rho(C):=\max_{\lambda\in\sigma(C)}\left|\lambda\right|$ its \begriff{spectral radius} 
and by $s(C):=\max_{\lambda\in\sigma(C)}\Re\lambda$ its \begriff{spectral abscissa} (or \begriff{spectral bound}). With this notation, we define the following 
constants which appear in the linear theory from \cite{Otten2014,Otten2014a,Otten2015a}:
\begin{equation}
  \begin{aligned}
  \amin :=& \left(\rho\left(A^{-1}\right)\right)^{-1}, &&\amax  := \rho(A),              &&\azero := -s(-A),\\
  \aone :=& \left(\frac{\amax^2}{\amin\azero}\right)^{\frac{d}{2}},                      &&\bzero := -s(Df(v_{\infty})).               && %&&\atwo:=\frac{4\amax^2}{\azero}, &&\bzero := -s(-B).
  \end{aligned}
  \label{equ:aminamaxazerobzero}
\end{equation}
Recall  the relations $0< a_0 \le \beta_{A}$ and $0<b_0 \le \beta_{\infty}$
to the coercivity constants from \eqref{cond:A3},\eqref{cond:A10}. 
Our main tool for investigating exponential decay in space are exponentially weighted function spaces. For the 
 choice of weight function we  follow \cite[Def. 3.1]{ZelikMielke2009}:
\begin{definition}\label{def:WeightFunctionOfExponentialGrowthRate} 
  \enum{1} A function $\theta\in C(\R^d,\R)$ is called a \begriff{weight function of exponential growth rate $\eta\geqslant 0$} provided that 
  \begin{flalign}
    &\theta(x)>0\;\forall\,x\in\R^d,                                                                        \tag{W1}\label{equ:WeightFunctionProp1} &\\
    &\exists\,C_{\theta}>0:\;\theta(x+y)\leqslant C_{\theta}\theta(x)e^{\eta|y|}\;\forall\,x,y\in\R^d.      \tag{W2}\label{equ:WeightFunctionProp2} &
  \end{flalign}
  \enum{2} A weight function $\theta\in C(\R^d,\R)$ of exponential growth rate $\eta\geqslant 0$ is called \begriff{radial} if
  \begin{flalign}
    &\exists\,\phi:[0,\infty)\rightarrow\R:\;\theta(x)=\phi\left(\left|x\right|\right)\;\forall\,x\in\R^d.   \tag{W3}\label{equ:WeightFunctionProp3} &
  \end{flalign}
  \enum{3} A radial weight function $\theta\in C(\R^d,\R)$ of exponential growth rate $\eta\geqslant 0$ is called \begriff{nondecreasing} 
  (or \begriff{monotonically increasing}) provided that
  \begin{flalign}
    &\theta(x)\leqslant\theta(y)\;\forall\,x,y\in\R^d\text{ with }|x|\leqslant|y|.                           \tag{W4}\label{equ:WeightFunctionProp4} &
  \end{flalign}
  %\enum{4} 
  %\begin{align}
  %  \frac{1}{\theta(x)}\leqslant C_{\theta}e^{-\eta|x|}\;\forall\,x\in\R^d.                                 \tag{W5}\label{equ:WeightFunctionProp5}
  %\end{align}
\end{definition}

Standard examples of radial weight functions are
\begin{align*}
  \theta_1(x)=\exp\left(\mu|x|\right) \quad\text{and}\quad \theta_2(x)=\cosh\left(\mu|x|\right), \quad x\in \R^d, \mu \in \R,
\end{align*}
as well as their smooth analogs 
\begin{align*}
  \theta_3(x)=\exp\left(\mu\sqrt{\left|x\right|^2+1}\right)\quad\text{and}\quad\theta_4(x)=\cosh\left(\mu\sqrt{\left|x\right|^2+1}\right), 
\quad x\in \R^d, \mu \in \R.
\end{align*}
Obviously, all these functions are radial weight functions of exponential growth rate $\eta=|\mu|$ with $C_{\theta}=1$. 
Moreover, $\theta_2$, $\theta_4$ are nondecreasing for any $\mu\in\R$ and $\theta_1$, $\theta_3$ if $\mu\geqslant 0$. 
%Moreover, $\theta_1$, $\theta_3$ are non-decreasing if $\mu \geqslant 0$ and $\theta_2$, $\theta_4$ are non-decreasing if $\mu\leqslant 0$. 
% Note that the constant weight function $\theta(x)=1$ is included as well as  (radial) tableau functions, e.g.
% \begin{align*}
%   \theta_5(x)=\begin{cases}1&,\,|x|\leqslant R,\\\exp(-\mu(|x|-R))&,\,|x|\geqslant R,\end{cases}
% \end{align*}
% for some $R>0$, where the constant $C_{\theta}$ depends on the size of the support, but not on the growth rate $\eta$.

With every weight function of exponential growth rate we associate \begriff{exponentially weighted Lebesgue} and 
\begriff{Sobolev spaces}
\begin{align*}
  L_{\theta}^{p}(\R^d,\K^N)   :=& \{u\in L^1_{\mathrm{loc}}(\R^d,\K^N)\mid \left\|\theta u\right\|_{L^p}<\infty\}, \\
  W_{\theta}^{k,p}(\R^d,\K^N) :=& \{u\in L^p_{\theta}(\R^d,\K^N)\mid D^{\beta}u\in L^p_{\theta}(\R^d,\K^N)\;\forall\,\left|\beta\right|\leqslant k\},
\end{align*}
for every $1\leqslant p\leqslant\infty$ and $k\in\N_0$.

With these preparations we can formulate the main result of our paper.

\begin{theorem}[Exponential decay of $v_{\star}$]\label{thm:NonlinearOrnsteinUhlenbeckSteadyState}
  Let the assumptions \eqref{cond:A4DC}, \eqref{cond:A5}--\eqref{cond:A8} and \eqref{cond:A10} be satisfied for $\K=\R$ and for some $1<p<\infty$. 
  Moreover, let $\amax=\rho(A)$ denote the spectral radius of $A$, $-\azero=s(-A)$ the spectral bound of $-A$ and $-\bzero=s(Df(v_{\infty}))$ the 
  spectral bound of $Df(v_{\infty})$. Further, let $\theta(x)=\exp\left(\mu\sqrt{|x|^2+1}\right)$ denote a weight function for $\mu\in\R$. 
  Then, for every $0<\varepsilon<1$ there is a constant $K_1=K_1(A,f,v_{\infty},d,p,\varepsilon)>0$ with the following property: Every classical solution 
  $v_{\star}$ of
  %Let the assumptions \eqref{cond:A4DC}, \eqref{cond:A5}--\eqref{cond:A8} and \eqref{cond:A10} be satisfied for some $1<p<\infty$ and $\K=\R$. 
  %Moreover, let $\amax=\rho(A)$ denote the spectral radius of $A$, $-\azero=s(-A)$ the spectral bound of $-A$ and 
  %$-\bzero=s(Df(v_{\infty}))$ the spectral bound of $Df(v_{\infty})$. Then, for every $0<\varepsilon<1$ there is a constant $K_1=K_1(A,f,v_{\infty},d,p,\varepsilon)>0$ 
  %with the following property: Every classical solution $v_{\star}$ of
  \begin{align}
    \label{equ:NonlinearProblemRealFormulation}
    A\triangle v(x)+\left\langle Sx,\nabla v(x)\right\rangle+f(v(x))=0,\,x\in\R^d,
  \end{align}
  such that
  \begin{align}
    \label{equ:BoundednessConditionForVStar}
    \sup_{|x|\geqslant R_0}\left|v_{\star}(x)-v_{\infty}\right|\leqslant K_1\text{ for some $R_0>0$},
  \end{align}
  satisfies
  \begin{align*}
    v_{\star}-v_{\infty}\in W^{1,p}_{\theta}(\R^d,\R^N)
  \end{align*}
  for every exponential decay rate
  \begin{align} 
    \label{eq:growthest}
    0\leqslant\mu\leqslant\varepsilon\frac{\sqrt{\azero\bzero}}{\amax p}.
  \end{align}
  %If additionally, $v_{\star}\in C^3(\R^d,\R^N)$, then it even holds $v_{\star}-v_{\infty}\in W^{2,p}_{\theta}(\R^d,\R^N)$. And if additionally, $f\in C^{k-1}(\R^N,\R^N)$ 
  %and $v_{\star}\in C^{k+1}(\R^d,\R^N)$ for some $k\in\N$ with $k\geqslant 3$, then $v_{\star}$ even satisfies $v_{\star}-v_{\infty}\in W^{k,p}_{\theta}(\R^d,\R^N)$.
\end{theorem}

%\begin{remark}
%  A detailed analysis in \eqref{equ:thresholdconstant} shows that $K_1$ can be taken of the order $\mathcal{O}((1-\varepsilon)^{\frac{d+1}{2p}})$ as $\varepsilon\to 0$.
%\end{remark}

Roughly speaking, Theorem \ref{thm:NonlinearOrnsteinUhlenbeckSteadyState} states that every bounded classical solution $v_{\star}$ which is sufficiently close to the 
steady state $v_{\infty}$ at infinity, see \eqref{equ:BoundednessConditionForVStar}, must decay exponentially in space. The exponential decay 
is expressed by the fact, that $v_{\star}-v_{\infty}$ belongs to an exponentially weighted Sobolev space. Moreover, the theorem gives an explicit bound for 
the exponential growth rate, that depends only on $p$, the spectral radius of $A$, and the spectral abscissas of $-A$ and $Df(v_{\infty})$. 
The role of $\varepsilon$ becomes clear upon noting that $K_1 \rightarrow 0$
as $\varepsilon \rightarrow 1$ whereas $K_1 \rightarrow K_1^0>0$
as $\varepsilon \rightarrow 0$.  The stronger the exponential rate, the
closer the solution $v_{\star}$ has to approach $v_{\infty}$ at infinity.

\subsection{Outline of proof: Decomposition of linear differential operators}
\label{subsec:2.2}
%----------------------------------------------------------------------------------

In the following we explain the decomposition of differential operators that leads to the proof of Theorem \ref{thm:NonlinearOrnsteinUhlenbeckSteadyState}.

\vspace*{\topsep}
\noindent
\textbf{Far-Field Linearization.} Consider the nonlinear problem
\begin{align}
  \label{equ:FarField1}
  A\triangle v_{\star}(x)+\left\langle Sx,\nabla v_{\star}(x)\right\rangle+f(v_{\star}(x))=0,\,x\in\R^d,\,d\geqslant 2.
\end{align}
Let $v_{\infty}\in\R^N$ be the constant asymptotic state satisfying \eqref{cond:A7} and let $f\in C^{1}(\R^N,\R^N)$.
By the Mean Value Theorem we can write
\begin{align*}
  f(v_{\star}(x))=\underbrace{f(v_{\infty})}_{=0}+\underbrace{\int_{0}^{1}Df(v_{\infty}+t(v_{\star}(x)-v_{\infty}))dt}_{=:a(x)}(v_{\star}(x)-v_{\infty}),\,x\in\R^d.
\end{align*}
From $v_{\star}\in C_{\mathrm{b}}(\R^d,\R^N)$ we deduce $a\in C_{\mathrm{b}}(\R^d,\R^{N,N})$. Moreover, since the classical solution $v_{\star}$ solves \eqref{equ:FarField1} 
pointwise and $v_{\infty}\in\R^N$ is constant, the difference $w_{\star}:=v_{\star}-v_{\infty}$ belongs to $C^2(\R^d,\R^N)\cap C_{\mathrm{b}}(\R^d,\R^N)$ and satisfies 
the linearized equation
\begin{align}
  \label{equ:FarField2}
 \L w_{\star}= A\triangle w_{\star}(x)+\left\langle Sx,\nabla w_{\star}(x)\right\rangle+a(x)w_{\star}(x)=0,\,x\in\R^d.
\end{align}
In order to study the behavior of solutions to \eqref{equ:FarField1} as $|x|\to\infty$, we decompose the variable coefficient $a(x)$ in \eqref{equ:FarField2}.

\vspace*{\topsep}
\noindent
\textbf{Decomposition of $a$.} Let $a(x)=Df(v_{\infty})+Q(x)$ with $Q$ defined by
\begin{align*}
  Q(x):=\int_{0}^{1}Df\left(v_{\infty}+t w_{\star}(x)\right)-Df\left(v_{\infty}\right)dt,\,x\in\R^d.
\end{align*}
This yields $Q\in C_{\mathrm{b}}(\R^d,\R^{N,N})$ and \eqref{equ:FarField2} reads as
\begin{align}
  \label{equ:FarField3}
  A\triangle w_{\star}(x)+\left\langle Sx,\nabla w_{\star}(x)\right\rangle+\left(Df(v_{\infty})+Q(x)\right)w_{\star}(x)=0,\,x\in\R^d,\,d\geqslant 2.
\end{align}

\vspace*{\topsep}
\noindent
\textbf{Decomposition of $Q$.} Let $Q(x)=Q_{\mathrm{s}}(x)+Q_{\mathrm{c}}(x)$, where $Q_{\mathrm{s}}\in C_{\mathrm{b}}(\R^d,\R^{N,N})$ is 
small w.r.t. $\left\|\cdot\right\|_{\infty}$ and $Q_{\mathrm{c}}\in C_{\mathrm{b}}(\R^d,\R^{N,N})$ is compactly supported on $\R^d$, 
see Figure \ref{fig:qdecomposition}. Then, we arrive at
\begin{align}
  \label{equ:FarFieldLinearization}
  A\triangle w_{\star}(x)+\left\langle Sx,\nabla w_{\star}(x)\right\rangle+\left(Df(v_{\infty})+Q_{\mathrm{s}}(x)+Q_{\mathrm{c}}(x)\right)w_{\star}(x)=0,\,x\in\R^d.
\end{align}
If we omit the term $Q_{\mathrm{s}}+Q_{\mathrm{c}}$ in \eqref{equ:FarFieldLinearization}, the equation \eqref{equ:FarFieldLinearization} is called the 
\begriff{far-field linearization}.

\begin{figure}[ht]
  \centering
  \begin{psfrags}
    \psfrag{qabs}[b][b]{\footnotesize$\left|Q(x)\right|$}
    \psfrag{qeps}[b][b]{\footnotesize\textcolor{red}{$\left|Q_{\mathrm{s}}(x)\right|$}}
    \psfrag{qcom}[b][b]{\footnotesize\textcolor{blue}{$\left|Q_{\mathrm{c}}(x)\right|$}}
    \psfrag{K2}[b][b]{\footnotesize$K_1$}
    \psfrag{R0}[b][b]{\footnotesize$R_0$}
    \psfrag{R}[b][b]{\footnotesize$\left|x\right|=R$}
    \includegraphics[page=5,width=0.75\textwidth]{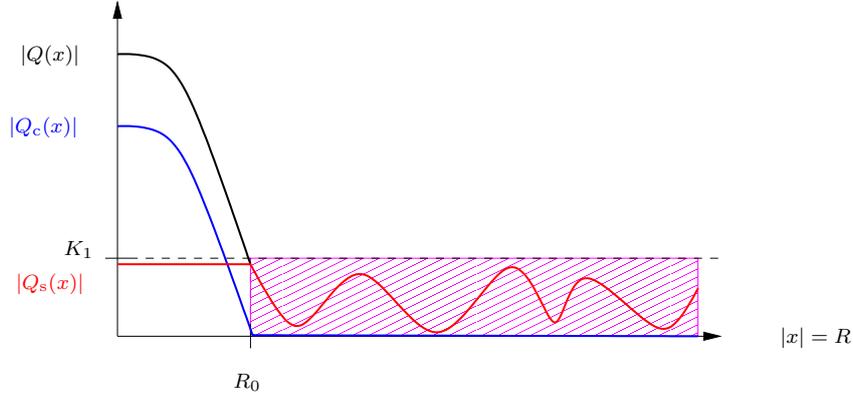}  
    %\label{fig:qdecomposition}
    \caption{Decomposition of $Q$ with data $R_0$ and $K_1$ from Theorem \ref{thm:NonlinearOrnsteinUhlenbeckSteadyState}\label{fig:qdecomposition}}
  \end{psfrags}
\end{figure}

%\vspace*{\topsep}
\noindent
\textbf{Perturbations of Ornstein-Uhlenbeck operator.} In order to show exponential decay for the solution $v_{\star}$ of the nonlinear steady state 
problem \eqref{equ:FarField1}, it is sufficient to analyze solutions of the linear system \eqref{equ:FarFieldLinearization}. Abbreviating $B_{\infty}:=-Df(v_{\infty})$, 
we will study the following linear differential operators:
\begin{equation} \label{equ:operators}
\begin{aligned}
  \left[\L_{\mathrm{c}}v\right](x)  &= A\triangle v(x)+\left\langle Sx,\nabla v(x)\right\rangle-B_{\infty}v(x)+Q_{\mathrm{s}}(x)v(x)+Q_{\mathrm{c}}(x)v(x), \\
  \left[\L_{\mathrm{s}}v\right](x)  &= A\triangle v(x)+\left\langle Sx,\nabla v(x)\right\rangle-B_{\infty}v(x)+Q_{\mathrm{s}}(x)v(x), \\
  \left[\L_{\infty}v\right](x)      &= A\triangle v(x)+\left\langle Sx,\nabla v(x)\right\rangle-B_{\infty}v(x), \\
  \left[\L_{0}v\right](x)           &= A\triangle v(x)+\left\langle Sx,\nabla v(x)\right\rangle.
\end{aligned}
\end{equation}
Recall that the drift term $\left\langle Sx,\nabla v(x)\right\rangle, x\in \R^d$, in the Ornstein-Uhlenbeck operator $\L_{0}$ has unbounded coefficients 
and cannot be considered as a lower order term.
Later on, it will be convenient to allow complex coefficients for the operators $\L_0$, $\L_{\infty}$, $\L_{\mathrm{s}}$ and $\L_{\mathrm{c}}$. 
Therefore, we rewrite the assumptions \eqref{cond:A8}--\eqref{cond:A10} as follows:
\begin{assumption}
  \label{ass:Assumption5}
  For  $B_{\infty}\in\K^{N,N}$ consider the conditions
  \begin{flalign}
    &\text{$A,B_{\infty}\in\K^{N,N}$ are simultaneously diagonalizable (over $\C$), i.e.}          \tag{A9${}_{B_{\infty}}$}\label{cond:A8B} &\\
    &\quad \exists\,Y\in\C^{N,N}\text{ invertible}:\;Y^{-1}AY=\Lambda_A\text{ and }Y^{-1}B_{\infty}Y=\Lambda_{B_{\infty}}, \nonumber &\\
    &\text{where $\Lambda_A=\diag\left(\lambda_1^A,\ldots,\lambda_N^A\right),\Lambda_{B_{\infty}}=\diag\left(\lambda_1^{B_{\infty}},\ldots,\lambda_N^{B_{\infty}}\right)\in\C^{N,N}$}\nonumber &\\
    &\Re\sigma(B_{\infty})>0,                     \tag{A10${}_{B_{\infty}}$}\label{cond:A9B} & \\
    &\Re\left\langle w,B_{\infty}w\right\rangle\geqslant\beta_{B_{\infty}}\;\forall\,w\in\K^N,\,|w|=1\text{ for some $\beta_{\infty}:=\beta_{B_{\infty}}>0$.} \tag{A11${}_{B_{\infty}}$}\label{cond:A10B} &
  \end{flalign}
\end{assumption}
Similar comments as those following \eqref{cond:A8}--\eqref{cond:A10} apply.
% Assumption \eqref{cond:A8B} is a \begriff{system condition}, corresponds to \eqref{cond:A8} and is independent of \eqref{cond:A9B} and \eqref{cond:A10B}. The condition 
% \eqref{cond:A8B} allows us to extend results for the scalar case to system cases and is used to derive a solution representation for the resolvent equation of $\L_{\infty}$. 
% Assumption \eqref{cond:A9B} is a \begriff{spectral condition} and states that the matrix $-B_{\infty}$ is a stable. This condition corresponds to \eqref{cond:A9}. 
% Similarly, assumption \eqref{cond:A10B} corresponds \eqref{cond:A10} and is more restrictive than \eqref{cond:A9B}, compare the relation between \eqref{cond:A2} and \eqref{cond:A3}.
In addition to  \eqref{equ:aminamaxazerobzero}, we need the constants
\begin{equation} \label{equ:constant2}
\bzero:=-s(-B_{\infty}), \quad \kappa:=\cond(Y) \quad
\text{(the condition number of $Y$ from \eqref{cond:A8B})}.
\end{equation}

%----------------------------------------------------------------------------------
% SUBSECTION 2.3: (Constant coefficient perturbations of complex Ornstein-Uhlenbeck operators).
%----------------------------------------------------------------------------------
\subsection{Constant coefficient perturbations of complex Ornstein-Uhlenbeck operators}
\label{subsec:OrnsteinUhlenbeckOperator}
%----------------------------------------------------------------------------------

In the first step we review and collect results from \cite{Otten2014,Otten2014a,Otten2015a,Otten2015b} for the 
complex-valued Ornstein-Uhlenbeck operator $\L_0$ in $L^p(\R^d,\C^N)$ and its constant coefficient perturbation $\L_{\infty}$.

Assuming \eqref{cond:A2}, \eqref{cond:A5} and \eqref{cond:A8B} for $\K=\C$ it is shown in \cite[Theorem 4.2-4.4]{Otten2014}, \cite[Theorem 3.1]{Otten2014a} 
that the function $H_{\infty}:\R^d\times\R^d\times(0,\infty)\rightarrow\C^{N,N}$ defined by
\begin{align}
  \label{equ:HeatKernel}
  H_{\infty}(x,\xi,t)=(4\pi t A)^{-\frac{d}{2}}\exp\left(-B_{\infty}t-(4tA)^{-1}\left|e^{tS}x-\xi\right|^2\right),%\,x,\xi\in\R^d,\,t\in(0,\infty)
\end{align}
is a heat kernel of the perturbed Ornstein-Uhlenbeck operator $\L_{\infty}$ from \eqref{equ:operators}.
% \begin{align}
%   \label{equ:Linfty2}
%   \left[\L_{\infty}v\right](x):=A\triangle v(x)+\left\langle Sx,\nabla v(x)\right\rangle-B_{\infty}v(x).%\,x\in\R^d.
% \end{align}
Under the same assumptions it is proved in \cite[Theorem 5.3]{Otten2014a} that the family of mappings 
\begin{align}
  \left[T_{\infty}(t)v\right](x):= \begin{cases}
                              \int_{\R^d}H_{\infty}(x,\xi,t)v(\xi)d\xi &\text{, }t>0 \\
                              v(x) &\text{, }t=0
                            \end{cases}\quad ,\,x\in\R^d,
  \label{equ:OrnsteinUhlenbeckSemigroupLp}
\end{align}
generates a strongly continuous semigroup 
$T_{\infty}(t):L^p(\R^d,\C^N)\rightarrow L^p(\R^d,\C^N)$, $t\geqslant 0$
 for each $1\leqslant p<\infty$, which satisfies the 
following estimate (see  \eqref{equ:aminamaxazerobzero},\eqref{equ:constant2} for the constants)
\begin{align}
  \label{equ:LpSemigroupBound}
  \left\|T_{\infty}(t)v\right\|_{L^p}\leqslant\kappa\aone e^{-\bzero t}\left\|v\right\|_{L^p}\;\forall\,t\geqslant 0.
\end{align}
The semigroup $\left(T_{\infty}(t)\right)_{t\geqslant 0}$ is called the Ornstein-Uhlenbeck semigroup if $B_{\infty}=0$. Otherwise, $\left(T_{\infty}(t)\right)_{t\geqslant 0}$ 
is a perturbed Ornstein-Uhlenbeck semigroup. The strong continuity of the semigroup justifies to introduce its infinitesimal generator 
$\A_p:L^p(\R^d,\C^N)\supseteq\D(\A_p)\rightarrow L^p(\R^d,\C^N)$ via
% \begin{align*}
%    \lim_{t\downarrow 0}\frac{T_{\infty}(t)v-v}{t},\; 1\leqslant p<\infty
% \end{align*}
% for every $v\in\D(\A_p)$, where the domain (or maximal domain) of $\A_p$ is given by
\begin{align*}
  \D(\A_p):=&\left\{v\in L^p(\R^d,\C^N)\mid \A_p v :=\lim_{t\downarrow 0}\frac{T_{\infty}(t)v-v}{t}\text{ exists in $L^p(\R^d,\C^N)$}\right\} .
\end{align*}
An application of abstract semigroup theory yields the unique solvability of the resolvent equation 
\begin{align}
  \label{equ:ResolventEquationAp}
  \left(\lambda I-\A_p\right)v = g,\quad\text{for all }g\in L^p(\R^d,\C^N),\;\lambda\in\C,\;\Re\lambda>-\bzero:=s(-B_{\infty})
\end{align}
in $L^p(\R^d,\C^N)$ for $1\leqslant p<\infty$, \cite[Corollary 6.7]{Otten2014}, \cite[Corollary 5.5]{Otten2014a}. Combining \eqref{equ:OrnsteinUhlenbeckSemigroupLp} 
with the representation $\left(\lambda I-\A_p\right)^{-1}g:=\int_{0}^{\infty}e^{-\lambda s}T_{\infty}(s)g ds$, the solution $v\in\D(\A_p)$ of \eqref{equ:ResolventEquationAp} 
satisfies
\begin{align}
  \label{equ:IntegralRepresentationLinfty}
  v = \left(\lambda I-\A_p\right)^{-1}g = \int_{0}^{\infty}\int_{\R^d}e^{-\lambda s}H_{\infty}(\cdot,\xi,s)g(\xi)d\xi ds.
  %v =& R_{\infty}(\lambda)(g+\Q_p v) = \int_{0}^{\infty}e^{-\lambda s}T_{\infty}(s)(g+\Q_p v)ds \\
  %=& \int_{0}^{\infty}\int_{\R^d}e^{-\lambda s}H_{\infty}(\cdot,\xi,s)(g(\xi)+Q(\xi)v(\xi))d\xi ds.
\end{align}
The following a-priori estimate in exponentially weighted $L^p$-spaces is based on the integral expression \eqref{equ:IntegralRepresentationLinfty} and is taken from \cite[Theorem 5.7]{Otten2014a}.

\begin{theorem}[Existence and uniqueness in weighted $W^{1,p}$-spaces]\label{thm:APrioriEstimatesInLpConstantCoefficients}
  Let the assumptions \eqref{cond:A2}, \eqref{cond:A5} and \eqref{cond:A8B} be satisfied for $1\leqslant p<\infty$ and $\K=\C$, and let $0<\varepsilon<1$ 
  and $\lambda\in\C$ with $\Re\lambda>-\bzero$ be given. Moreover, let $\theta\in C(\R^d,\R)$ be a radially nondecreasing weight function of exponential growth rate 
  $\eta\geqslant 0$ with 
  \begin{align*}
    0\leqslant\eta^2\leqslant\varepsilon\frac{\azero(\Re\lambda+\bzero)}{\amax^2 p^2}.
  \end{align*} 
     Then, there exists a unique solution $v\in\D(\A_p)$ of the resolvent equation $(\lambda I-\A_p)v = g$ for every $g\in L^p_{\theta}(\R^d,\C^N)$. The solution satisfies 
  $v\in W^{1,p}_{\theta}(\R^d,\C^N)$ and the following estimates
  \begin{align}
        \left\|v\right\|_{L^p_{\theta}} \leqslant& \frac{C_{0,\varepsilon}}{\Re\lambda+\bzero}\left\|g\right\|_{L^p_{\theta}},                                             \label{equ:ExpDecStatVstar}\\
    \left\|D_i v\right\|_{L^p_{\theta}} \leqslant& \frac{C_{1,\varepsilon}}{\left(\Re\lambda+\bzero\right)^{\frac{1}{2}}}\left\|g\right\|_{L^p_{\theta}},\,i=1,\ldots,d,   \label{equ:ExpDecStatDiVstar}
  \end{align}
  where the $\lambda$-independent constants $C_{0,\varepsilon}$, $C_{1,\varepsilon}$ are given by
  \begin{align*}
    C_{0,\varepsilon} =& C_{\theta}\kappa \aone\left(\frac{\Gamma\left(\frac{d+1}{2}\right)}{\Gamma\left(\frac{d}{2}\right)}(\pi\varepsilon)^{\frac{1}{2}}(1-\varepsilon)^{-\frac{d+1}{2}}+{}_2F_1\left(\frac{d}{2},1;\frac{1}{2};\varepsilon\right)\right)^{\frac{1}{p}}, \\
    C_{1,\varepsilon} =& C_{\theta}\kappa\frac{\aone^{\frac{d+1}{d}}\pi^{\frac{1}{2}}}{\amin^{\frac{1}{2}}}
%\frac{\Gamma\left(\frac{d+2}{2}\right)}{\Gamma\left(\frac{d}{2}\right)}
\left(\frac{\Gamma\left(\frac{d+1}{2}\right)}{\Gamma\left(\frac{d}{2}\right)}(1-\varepsilon)^{-\frac{d+1}{2}}+\frac{d 
\varepsilon^{\frac{1}{2}}}{\pi^{\frac{1}{2}}}{}_2F_1\left(\frac{d+1}{2},1;\frac{3}{2};\varepsilon\right)\right)^{\frac{1}{p}},
    %C_{7,\varepsilon} =& C_{\theta}\kappa \aone\left(\frac{1}{1-\varepsilon}\right)^{\frac{1}{p}}
    %       \bigg({}_2F_1\left(-\frac{d-1}{2},1;\frac{1}{2};-\frac{\varepsilon}{1-\varepsilon}\right) \\
    %     & +\pi^{\frac{1}{2}}\frac{\Gamma\left(\frac{d+1}{2}\right)}{\Gamma\left(\frac{d}{2}\right)}\left(\frac{\varepsilon}{1-\varepsilon}\right)^{\frac{1}{2}}
    %       {}_2F_1\left(-\frac{d-2}{2},\frac{3}{2};\frac{3}{2};-\frac{\varepsilon}{1-\varepsilon}\right)\bigg)^{\frac{1}{p}}, \\
    %C_{8,\varepsilon} =& C_{\theta}\kappa \aone^{\frac{d+1}{d}}\frac{\Gamma\left(\frac{1}{2}\right)}{\amin^{\frac{1}{2}}}
    %       \left(\frac{1}{1-\varepsilon}\right)^{\frac{1}{2p}}
    %       \bigg(\frac{\Gamma\left(\frac{d+1}{2}\right)}{\Gamma\left(\frac{d}{2}\right)}{}_2F_1\left(-\frac{d}{2},\frac{1}{2};\frac{1}{2};-\frac{\varepsilon}{1-\varepsilon}\right) \\
    %     &  +2\frac{\Gamma\left(\frac{d+2}{2}\right)}{\Gamma\left(\frac{1}{2}\right)\Gamma\left(\frac{d}{2}\right)}\left(\frac{\varepsilon}{1-\varepsilon}\right)^{\frac{1}{2}}
    %       {}_2F_1\left(-\frac{d-1}{2},1;\frac{3}{2};-\frac{\varepsilon}{1-\varepsilon}\right)\bigg)^{\frac{1}{p}}
  \end{align*}
  with constants $\azero,\aone,\amin,\amax$ from \eqref{equ:aminamaxazerobzero}, $\bzero,\kappa$ from \eqref{equ:constant2} and $C_{\theta}$ from \eqref{equ:WeightFunctionProp2}.
  %with the constants from \eqref{equ:aminamaxazerobzero}, \eqref{equ:constant2}
  %and $C_{\theta}$ from \eqref{equ:WeightFunctionProp2}.
\end{theorem}

\begin{remark}
 Above
we used the hypergeometric function ${}_2F_1$, see \cite[15.4]{OlverLozierBoisvertClark2010}. Moreover, we  modified the original constants from \cite[Theorem 5.7]{Otten2014a} by using 
  ${}_2F_1\left(a,b;b;z\right)=(1-z)^{-a}$ from \cite[(15.4.6)]{OlverLozierBoisvertClark2010} and the Pfaff transformation ${}_2F_1\left(a,b;c,z\right)=(1-z)^{-b}{}_2F_1\left(c-a,b;c;\frac{z}{z-1}\right)$ for $z\in \C \, \setminus\, [1,\infty)$. Note that both quantities
${}_2F_1\left(\frac{d}{2},1;\frac{1}{2};\varepsilon\right)$ and
${}_2F_1\left(\frac{d+1}{2},1;\frac{3}{2};\varepsilon\right)$ behave like $(1-\varepsilon)^{-\frac{d+1}{2}}$ as $\varepsilon \rightarrow 1$
(\cite[(15.4.23)]{OlverLozierBoisvertClark2010}), which
then also determines the behavior of  the constants
 $C_{0,\varepsilon}$ and $C_{1,\varepsilon}$.
\end{remark}

So far, we neither have an explicit representation for the maximal domain $\D(\A_p)$ in terms of Sobolev spaces, nor do we have the relation between the generator $\A_p$ and the differential 
operator $\L_{\infty}$. For this purpose, one has to solve the identification problem, which has been done in \cite{Otten2015a}. 
Assuming \eqref{cond:A2}, \eqref{cond:A5} and \eqref{cond:A8B} for $\K=\C$, it is proved in \cite[Theorem 3.2]{Otten2015a} that the Schwartz space 
$\S(\R^d,\C^N)$ is a core of the infinitesimal generator $\left(\A_p,\D(\A_p)\right)$ for any $1\leqslant p<\infty$. Next, one considers the operator 
$\L_{\infty}:L^p(\R^d,\C^N)\supseteq\D^p_{\mathrm{loc}}(\L_0)\rightarrow L^p(\R^d,\C^N)$ on its domain 
\begin{align} \label{eq:domainL0}
  \D^p_{\mathrm{loc}}(\L_0):=\left\{v\in W^{2,p}_{\mathrm{loc}}(\R^d,\C^N)\cap L^p(\R^d,\C^N)\mid A\triangle v+\left\langle S\cdot,\nabla v\right\rangle\in L^p(\R^d,\C^N)\right\}.
\end{align} 
Under the assumption \eqref{cond:A3} for $\K=\C$, it is shown in \cite[Lemma 4.1]{Otten2015a} that $\left(\L_{\infty},\D^p_{\mathrm{loc}}(\L_0)\right)$ 
is a closed operator in $L^p(\R^d,\C^N)$ for any $1<p<\infty$.
Then the $L^p$-dissipativity 
condition \eqref{cond:A4DC} is the key assumption which leads to
an energy estimate for the resolvent with respect to the $L^p$-norm, see \cite[Theorem 4.4]{Otten2015a}.
The same argument reappears in Theorem \ref{thm:WeightedResolventEstimates} below which is an extension of \cite[Theorem 4.4]{Otten2015a}. 
As a direct consequence, the operator $\L_{\infty}$ is dissipative in $L^p(\R^d,\C^N)$, provided 
$\beta_{\infty}$ from Assumption \eqref{cond:A10B} satisfies $\beta_{\infty}\leqslant 0$, \cite[Corollary 4.6]{Otten2015a}. 
Combining these results one can solve 
the identification problem for $\L_{\infty}$ as follows (see \cite[Theorem 5.1]{Otten2015a}).

\begin{theorem}[Maximal domain, local version]\label{thm:LpMaximalDomainPart1}
  Let the assumptions \eqref{cond:A4DC}, \eqref{cond:A5} and \eqref{cond:A8B} be satisfied for $1<p<\infty$ and $\K=\C$, then 
  \begin{align*}
    \D(\A_p)=\D^p_{\mathrm{loc}}(\L_0)
  \end{align*} 
  is the maximal domain of $\A_p$, where $\D^p_{\mathrm{loc}}(\L_0)$ is defined by
\eqref{eq:domainL0}.  
% \begin{align}
%     \label{equ:localDomain}
%     \D^p_{\mathrm{loc}}(\L_0):=\left\{v\in W^{2,p}_{\mathrm{loc}}(\R^d,\C^N)\cap L^p(\R^d,\C^N)\mid A\triangle v+\left\langle S\cdot,\nabla v\right\rangle\in L^p(\R^d,\C^N)\right\}.
%   \end{align} 
  In particular, $\A_p$ is the maximal realization of $\L_{\infty}$ in $L^p(\R^d,\C^N)$, i.e. 
  \begin{align*}
    \A_p v = \L_{\infty} v\quad\forall\,v\in\D(\A_p).
  \end{align*}
\end{theorem}

Theorem \ref{thm:LpMaximalDomainPart1} shows that, if we restrict $1<p<\infty$ and replace \eqref{cond:A2} by the stronger assumption \eqref{cond:A4DC} in Theorem 
\ref{thm:APrioriEstimatesInLpConstantCoefficients}, we can write $\L_{\infty}$ and $\D^p_{\mathrm{loc}}(\L_0)$ instead of $\A_p$ and $\D(\A_p)$. This will be crucial 
in the proof of Theorem \ref{thm:APrioriEstimatesInLpSmallPerturbation}
below.
Moreover, we stress again that it is this theorem into which
the $L^p$-dissipativity condition \eqref{cond:A4DC} enters,
see the comments following Assumption \ref{ass:Assumption1}.
% Theorem \ref{thm:LpMaximalDomainPart1} also shows that the  is essential to solve the identification problem 
% for perturbed complex-valued Ornstein-Uhlenbeck operators. For the application of Theorem \ref{thm:LpMaximalDomainPart1} it is instructive to understand  classes 
% of matrices $A$ which satisfy the algebraic condition \eqref{cond:A4DC}. In \cite{Otten2015b}, the $L^p$-dissipativity condition \eqref{cond:A4DC} is completly 
% characterized in terms of the first antieigenvalue $\mu_1(A)$ of the diffusion matrix $A$. More precisely, it is proved in \cite[Theorem 3.1]{Otten2015b} that 
% \eqref{cond:A4DC} is equivalent to the $L^p$-antieigenvalue condition \eqref{cond:A4}. Some special cases of \eqref{cond:A4} and explicit 
% representations of $\mu_1(A)$ are also discussed in \cite[Section 4]{Otten2015b}.

%----------------------------------------------------------------------------------
% SUBSECTION 2.4: (Bootstrapping and regularity).
%----------------------------------------------------------------------------------
\subsection{Bootstrapping and regularity.}
\label{subsec:2.4}
%----------------------------------------------------------------------------------

In Section \ref{sec:3} we study the variable coefficient operator
\begin{equation} \label{equ:varcoeffop}
  \left[\L_Q v\right](x) = A\triangle v(x) + \left\langle Sx,\nabla v(x)\right\rangle - B_{\infty}v(x) + Q(x)v(x),\,x\in\R^d,
\end{equation}
and its resolvent equation
\begin{align}
  \label{equ:REQ}
  \left(\lambda I-\L_Q\right) v = g, \quad
\text{in} \; L^p(\R^d,\C^N) 
\end{align}
for $1<p<\infty$ and for different choices of $Q\in L^{\infty}(\R^d,\C^{N,N})$. 
In Section \ref{subsec:3.1}, we first derive an existence and uniqueness result for the resolvent equation \eqref{equ:REQ} in $L^p(\R^d,\C^N)$ for general $Q$ 
(Theorem \ref{thm:LpSolvabilityUniquenessBoundedPerturbation}). The proof uses the standard bounded perturbation theorem from 
abstract semigroup theory as well as Theorem \ref{thm:LpMaximalDomainPart1}. 
In Section \ref{subsec:3.2}, we then analyze the resolvent equation \eqref{equ:REQ} for perturbations $Q=Q_{\mathrm{s}}$ which are small
w.r.t. $\left\|\cdot\right\|_{L^{\infty}}$. We prove that the unique solution of \eqref{equ:REQ} in $L^p(\R^d,\C^N)$ decays exponentially 
if the inhomogeneity $g$ does (Theorem \ref{thm:APrioriEstimatesInLpSmallPerturbation}). The proof is based on a fixed point argument and uses the results 
from Theorem \ref{thm:LpSolvabilityUniquenessBoundedPerturbation} and Theorem \ref{thm:APrioriEstimatesInLpConstantCoefficients}. 
In Section \ref{subsec:3.3}, we study differential operators of the form
\begin{align*}
  \left[\L_B v\right](x) = A\triangle v(x) + \left\langle Sx,\nabla v(x)\right\rangle - B(x)v(x),\,x\in\R^d,
\end{align*}
where the matrix-valued function $B\in L^{\infty}(\R^d,\C^{N,N})$ satisfies
\begin{align*}
  \Re\left\langle w,B(x)w\right\rangle \geqslant c_B |w|^2\;\forall\,x\in\R^d\;\forall\,w\in\C^N
\end{align*}
for some constant $c_B\in\R$.  We consider two different weight functions
$\theta_1,\theta_2$ satisfying $\theta_1 \le C \theta_2$, so that
$L^p_{\theta_2}(\R^d,\C^N)\subseteq L^p_{\theta_1}(\R^d,\C^N)$, e.g.
$\theta_2$ may grow while $\theta_1$ decays. Then
we prove uniqueness of solutions $v$ of $(\lambda I-\L_B)v=g$ in the large space $W^{2,p}_{\mathrm{loc}}(\R^d,\C^N)\cap L^p_{\theta_1}(\R^d,\C^N)$ 
if $g$ is in the small space $ L^p_{\theta_2}(\R^d,\C^N)$, and we derive resolvent estimates (Theorem \ref{thm:WeightedResolventEstimates}). 
The proof generalizes the approach from \cite[Theorem 5.13]{Otten2014} to variable coefficient 
perturbations and weighted spaces. In Section \ref{subsec:3.4} we study the resolvent equation \eqref{equ:REQ} for asymptotically small variable coefficient 
matrices $Q$. We prove that if $|Q(x)|$ falls below a certain threshold at infinity, then every solution $v\in W^{2,p}_{\mathrm{loc}}(\R^d,\C^N)\cap L^p_{\theta_1}(\R^d,\C^N)$ of
\eqref{equ:REQ} in $L^p_{\mathrm{loc}}(\R^d,\C^N)$ already belongs to the small
space $W^{1,p}_{\theta_2}(\R^d,\C^N)$ if $g\in L^p_{\theta_2}(\R^d,\C^N)$, and $\Re\lambda>-\beta_{\infty}$ 
(Theorem \ref{thm:APrioriEstimatesInLpRelativelyCompactPerturbation}). The idea of the proof is to decompose $Q$ into $Q=Q_{\mathrm{s}}+Q_{\mathrm{c}}$, where 
$Q_{\mathrm{s}}\in L^{\infty}(\R^d,\C^N)$ is small w.r.t. $\left\|\cdot\right\|_{L^{\infty}}$ and $Q_{\mathrm{c}}$ is compactly supported on $\R^d$. Then, Theorem 
\ref{thm:APrioriEstimatesInLpSmallPerturbation} implies the existence of a solution in the smaller space $W^{2,p}_{\mathrm{loc}}(\R^d,\C^N)\cap L^p_{\theta_2}(\R^d,\C^N)$ and 
Theorem \ref{thm:WeightedResolventEstimates} yields the uniqueness in the larger space $W^{2,p}_{\mathrm{loc}}(\R^d,\C^N)\cap L^p_{\theta_1}(\R^d,\C^N)$. Note that 
Theorem \ref{thm:APrioriEstimatesInLpRelativelyCompactPerturbation} is the core theorem which allows us to analyze exponential decay for both, solutions of the nonlinear problem 
and solutions of the eigenvalue problem for $\L$.

In Section \ref{sec:4} we prove spatial exponential decay for bounded solutions of the nonlinear problem \eqref{equ:FarField1}
by employing a bootstrapping argument to the linear equation
\eqref{equ:FarField2}. Shifting the term with the compactly supported coefficient
to the right-hand side, we obtain an inhomogeneity which lies in any
weighted $L^p$-space. Applying the previous linear theory then
provides exponential decay in space
provided the difference $|v_{\star}(x)-v_{\infty}|$ falls below a certain threshold at 
infinity. In a second step, assuming additional regularity of the nonlinearity $f$ and the solution $v_{\star}$, we 
show that the higher order derivatives also decay exponentially in space (Corollary \ref{cor:NonlinearOrnsteinUhlenbeckSteadyStateMoreRegularity}, 
Remark \ref{rem:HigherRegularity})
  \begin{equation} \label{equ:decaysobolev}
v_{\star}-v_{\infty}\in W^{k,p}_{\theta}(\R^d,\R^N) \quad \text{if} 
    \quad f\in C^{\max\{2,k-1\}}(\R^N,\R^N),\;v_{\star}\in C^{k+1}(\R^d,\R^N).
 \end{equation}
This holds for $k\in\N$ and $p\geqslant\frac{d}{2}$ in case $k\geqslant 3$, where $p$ is from \eqref{cond:A4DC}.
%and if $p$ from \eqref{cond:A4DC} satisfies  $p\geqslant\frac{d}{2}$ in case $k\geqslant 3$.
% \begin{align*}
%   &f\in C^2(\R^N,\R^N),\;v_{\star}\in C^3(\R^d,\R^N)         &&\Longrightarrow\quad v_{\star}-v_{\infty}\in W^{2,p}_{\theta}(\R^d,\R^N), \\
%   &f\in C^{k-1}(\R^N,\R^N),\;v_{\star}\in C^{k+1}(\R^d,\R^N) &&\Longrightarrow\quad v_{\star}-v_{\infty}\in W^{k,p}_{\theta}(\R^d,\R^N)\;\forall\,k\in\N,\,k\geqslant 3.
% \end{align*}
In Section \ref{subsec:4.3} we combine this result with Sobolev embeddings to deduce that $v_{\star}-v_{\infty}$ satisfies 
exponentially weighted pointwise estimates (Corollary \ref{cor:pointwise}) 
\begin{align}\label{eq:derivpointwise}
  \left|D^{\alpha}\left(v_{\star}(x)-v_{\infty}\right)\right| \leqslant C\exp\left(-\mu\sqrt{|x|^2+1}\right)
  \quad\forall\,x\in\R^d,\;0\leqslant\mu\leqslant\varepsilon\frac{\sqrt{\azero\bzero}}{\amax p}
\end{align}
and for every multi-index $\alpha\in\N_0^d$ with $d<(k-|\alpha|)p$. In Section \ref{subsec:4.4} we extend our main result from Theorem 
\ref{thm:NonlinearOrnsteinUhlenbeckSteadyState}, Corollary \ref{cor:NonlinearOrnsteinUhlenbeckSteadyStateMoreRegularity} and Corollary 
\ref{cor:pointwise} to complex-valued systems with $f$ as in \eqref{equ:complexversion} (Corollary \ref{cor:NonlinearOrnsteinUhlenbeckSteadyStateComplexVersion}). 

In Section \ref{sec:5} we study spatial exponential decay for solutions of the eigenvalue problem
\begin{align}
  \label{equ:EP}
  A\triangle v(x) + \left\langle Sx,\nabla v(x)\right\rangle + Df(v_{\star}(x))v(x) = \lambda v(x),\,x\in\R^d,\,d\geqslant 2.
\end{align}
In Section \ref{subsec:5.1} we show that every bounded classical solution $v$ of \eqref{equ:EP} decays exponentially in space, in the sense that 
$v$ belongs to $W^{1,p}_{\theta}(\R^d,\R^N)$, provided that its associated eigenvalue $\lambda\in\C$ satisfies $\Re\lambda>-\beta_{\infty}$. 
In Section \ref{subsec:5.2} we apply our result from Section \ref{subsec:5.1} to those eigenfunctions which belong to eigenvalues  on the imaginary axis.
These eigenfunctions are due to equivariance with respect to the action
of the Euclidean group and can be calculated explicitly in terms of
the profile $v_{\star}$, see Theorem \ref{thm:EigenfunctionsOfTheLinearizedOrnsteinUhlenbeckInLp}.
 In particular, this yields exponential decay of the eigenfunction $v(x)=\left\langle Sx,\nabla v_{\star}(x)\right\rangle, x \in \R^d$ 
associated with the eigenvalue $\lambda=0$.
As in the nonlinear case we proceed with proving exponential decay
of derivatives of eigenfunctions, first in Sobolev spaces and then in a pointwise sense as in \eqref{eq:derivpointwise},
see Theorem \ref{thm:LinearizedOrnsteinUhlenbeckExponentialDecayInLp}.

In Section \ref{sec:6} we apply the theory to so called spinning solitons of the cubic-quintic complex Ginzburg-Laundau equation (QCGL)
\begin{align*}
  u_t = \alpha\triangle u + u\left(\delta+\beta|u|^2+\gamma|u|^4\right)u,
\end{align*}
where $u:\R^d\times[0,\infty)\rightarrow\C$, $d\in\{2,3\}$ and $\alpha,\beta,\gamma,\delta\in\C$ with $\Re\alpha>0$. We derive suitable 
conditions on  the parameters $\alpha,\beta,\gamma,\delta\in\C$ such
that  Theorem 
\ref{thm:APrioriEstimatesInLpConstantCoefficients} and Corollary \ref{cor:NonlinearOrnsteinUhlenbeckSteadyStateComplexVersion} apply.
 In Section \ref{subsec:6.1} we compute the
profile and (angular) speed of the spinning 
solitons. 
In Section \ref{subsec:6.2} we compute spectra and eigenfunctions of the associated eigenvalue problem. 
In Section \ref{subsec:6.3} we compare in a final step
the theoretical decay rates with numerical rates obtained from numerical
data on a large ball. It turns out that the theoretical bounds are
surprisingly close to the values found from numerical computations.

%  ----------------
% | Acknowledgment |
%  ----------------
%\noindent
\textbf{Acknowledgment.} The authors are grateful to Jens Lorenz for useful
discussions during the preparatory stages of this work.

%---------------------------------------------------------------------------------------------------------------------------------------------------
%
%  SECTION 3: (Variable coefficient perturbations of complex Ornstein-Uhlenbeck operators)
%
%---------------------------------------------------------------------------------------------------------------------------------------------------
\sect{Variable coefficient complex Ornstein-Uhlenbeck operators}
\label{sec:3}
%--------------------------------------------------------------------------------------------------------------------------------------------------- 

In this section we analyze the resolvent equation of the differential operator
\begin{equation} \label{equ:operatorLQ}
  \left[\L_Q v\right](x) = A\triangle v(x) + \left\langle Sx,\nabla v(x)\right\rangle - B_{\infty}v(x) + Q(x)v(x),\,x\in\R^d,
\end{equation}
in $L^p(\R^d,\C^N)$ for $1<p<\infty$ and for different choices of $Q\in L^{\infty}(\R^d,\C^{N,N})$. 

%----------------------------------------------------------------------------------
% SUBSECTION 3.1: (Solvability and uniqueness of the resolvent equation).
%----------------------------------------------------------------------------------
\subsection{Solvability and uniqueness of the resolvent equation}
\label{subsec:3.1}
%----------------------------------------------------------------------------------

Let us assume \eqref{cond:A2}, \eqref{cond:A5} and \eqref{cond:A8B} for $\K=\C$ and let $(\A_p,\D(\A_p))$ denote the generator of the 
strongly continuous semigroup $(T_{\infty}(t))_{t\geqslant 0}$ from Section \ref{subsec:OrnsteinUhlenbeckOperator} on 
$L^p(\R^d,\C^N)$ for some $1\leqslant p<\infty$.
% Recall that the semigroup $(T_{\infty}(t))_{t\geqslant 0}$ satisfies the exponential estimate \eqref{equ:LpSemigroupBound}.
% \begin{align*}
%   \left\|T_{\infty}(t)v\right\|_{L^p}\leqslant\kappa\aone e^{-\bzero t}\left\|v\right\|_{L^p}\;\forall\,t\geqslant 0.
% \end{align*}
% with $\aone$, $\bzero$ from \eqref{equ:aminamaxazerobzero} and condition number $\kappa:=\mathrm{cond}(Y)$ with $Y$ from \eqref{cond:A8B}. 
Let us introduce the bounded operator
\begin{align*}
  \Q_p:L^p(\R^d,\C^N)\rightarrow L^p(\R^d,\C^N)\quad\text{with}\quad \left[\Q_p v\right](x) := Q(x)v(x),\;x\in\R^d.
\end{align*}
Then the bounded perturbation theorem \cite[III.1.3]{EngelNagel2000} implies that
\begin{align*}
  \B_p := \A_p + \Q_p\quad\text{with}\quad \D(\B_p):=\D(\A_p)
\end{align*}
generates a strongly continuous semigroup $(T_Q(t))_{t\geqslant 0}$ in $L^p(\R^d,\C^N)$  satisfying
\begin{align*}
            \left\|T_Q(t)v\right\|_{L^p}
  \leqslant \kappa\aone e^{(-\bzero+\kappa\aone\left\|\Q_p\right\|)t}\left\|v\right\|_{L^p}
  \leqslant \kappa\aone e^{(-\bzero+\kappa\aone\left\|Q\right\|_{L^{\infty}})t}\left\|v\right\|_{L^p}\;\forall\,t\geqslant 0,
\end{align*}
where we used \eqref{equ:LpSemigroupBound} and the estimate $\left\|\Q_p\right\|\leqslant\left\|Q\right\|_{L^{\infty}}$
of the $L^p$-operator norm. 
Then an application of \cite[II.1.10]{EngelNagel2000} yields that the resolvent equation
\begin{align}
  \label{equ:ResolventEquationBp}
  \left(\lambda I-\B_p\right)v=g,\,\text{in }L^p(\R^d,\C^N)
\end{align}
for $\lambda\in\C$ with $\Re\lambda>-\bzero+\kappa\aone\left\|Q\right\|_{L^{\infty}}$ and $g\in L^p(\R^d,\C^N)$ admits a unique solution $v\in\D(\A_p)$ 
which satisfies the resolvent estimate
\begin{align*}
  \left\|v\right\|_{L^p}\leqslant \frac{\kappa\aone}{\Re\lambda-\left(-\bzero+\kappa\aone\left\|Q\right\|_{L^{\infty}}\right)}\left\|g\right\|_{L^p}.
\end{align*}
%Let us write \eqref{equ:ResolventEquationBp} as $\left(\lambda I-\A_p\right)v=g+\Q_p v$, which by the representation of the resolvent operator  
% $R_{\infty}(\lambda)g:=\int_{0}^{\infty}e^{-\lambda s}T_{\infty}(s)g ds$ for $g\in L^p(\R^d,\C^N)$, implies that the solution $v\in\D(\A_p)$ 
%of \eqref{equ:ResolventEquationBp} satisfies the integral equation
%\begin{align*}
%  v =& R_{\infty}(\lambda)(g+\Q_p v) = \int_{0}^{\infty}e^{-\lambda s}T_{\infty}(s)(g+\Q_p v)ds \\
%  =& \int_{0}^{\infty}\int_{\R^d}e^{-\lambda s}H_{\infty}(\cdot,\xi,s)(g(\xi)+Q(\xi)v(\xi))d\xi ds.
%\end{align*}
If we restrict $1<p<\infty$ and assume the stronger assumption \eqref{cond:A4DC} (or equivalently \eqref{cond:A4}) instead of \eqref{cond:A2}, an application 
of Theorem \ref{thm:LpMaximalDomainPart1} yields that
  $\D(\B_p) := \D(\A_p) = \D^p_{\mathrm{loc}}(\L_0)$
and
 $ \B_p v := \A_p v + \Q_p v = \L_{\infty}v + Qv = \L_Q$ for all $v\in\D(\B_p)$.
Therefore we can write in the following $\L_Q$ and $\D^p_{\mathrm{loc}}(\L_0)$ instead of $\B_p$ and $\D(\B_p)$. Summarizing,
we obtain the following result.
\begin{theorem}[Existence and uniqueness in weighted $L^p$-spaces]\label{thm:LpSolvabilityUniquenessBoundedPerturbation}
  Let the assumptions \eqref{cond:A4DC}, \eqref{cond:A5}, \eqref{cond:A8B} and $Q\in L^{\infty}(\R^d,\C^{N,N})$ be satisfied 
  for $1<p<\infty$ and $\K=\C$. Moreover, with constants $\aone$ from \eqref{equ:aminamaxazerobzero}, $\bzero,\kappa$ from \eqref{equ:constant2}, 
  %with the constants from \eqref{equ:aminamaxazerobzero}, \eqref{equ:constant2}
  let $\omega := -\bzero+\kappa\aone\left\|Q\right\|_{L^{\infty}}$
  and $\lambda\in\C$ with $\Re\lambda>\omega$ be given. Then, for every $g\in L^p(\R^d,\C^N)$ the resolvent equation
  \begin{align*}
    \left(\lambda I-\L_Q\right)v = g
  \end{align*}
  admits a unique solution $v\in\D^p_{\mathrm{loc}}(\L_0)$.
  %which satisfies the integral equation
  %\begin{align*}
  %  v(x) = \int_{0}^{\infty}\int_{\R^d}e^{-\lambda s}H_{\infty}(x,\xi,s)\left(g(\xi)+Q(\xi)v(\xi)\right) d\xi ds,\,x\in\R^d,
  %\end{align*}
  %where the heat kernel $H_{\infty}$ is given by \eqref{equ:HeatKernel}. 
  Moreover, the following resolvent estimate holds:
  \begin{align*}
    \left\|v\right\|_{L^p}\leqslant \frac{\kappa\aone}{\Re\lambda-\omega}\left\|g\right\|_{L^p}.
  \end{align*}
\end{theorem}

%----------------------------------------------------------------------------------
% SUBSECTION 3.2: (Small perturbations).
%----------------------------------------------------------------------------------
\subsection{Exponential decay for small perturbations}
\label{subsec:3.2}
%----------------------------------------------------------------------------------
In the following we use the constants $\azero,\aone,\amax,\amin$ from \eqref{equ:aminamaxazerobzero}, $\bzero,\kappa$ from \eqref{equ:constant2}
and $C_{0,\varepsilon},C_{1,\varepsilon}$ from Theorem \ref{thm:APrioriEstimatesInLpConstantCoefficients} without further reference.

\begin{theorem}[Existence and uniqueness in weighted $W^{1,p}$-spaces]\label{thm:APrioriEstimatesInLpSmallPerturbation}
  Let the assumptions \eqref{cond:A4DC}, \eqref{cond:A5}, \eqref{cond:A8B} and \eqref{cond:A9B} be satisfied for $1<p<\infty$ and $\K=\C$. 
  Moreover, let $0<\varepsilon<1$, $\theta\in C(\R^d,\R)$ be a radially nondecreasing weight function of exponential growth rate 
  \begin{align} \label{equ:etarestrict}
    0\leqslant\eta\leqslant\varepsilon\frac{\sqrt{\azero\bzero}}{\amax p},
  \end{align}
  and let $Q_{\mathrm{s}}\in L^{\infty}(\R^d,\C^{N,N})$ satisfy 
  \begin{align} \label{equ:Qsmall}
    \left\|Q_{\mathrm{s}}\right\|_{L^{\infty}}\leqslant
    \frac{\varepsilon b_0}{2}\min\left\{\frac{1}{\kappa\aone},\frac{1}{C_{0,\varepsilon}}\right\}.
  \end{align} 
  Further,  let $\lambda\in\C$ with $\Re\lambda\geqslant -(1-\varepsilon)\bzero$ and $g\in L^p_{\theta}(\R^d,\C^N)$. \\
  Then there exists a unique solution $v\in\D^p_{\mathrm{loc}}(\L_0)$ of the resolvent equation $(\lambda I-\L_{\mathrm{s}})v = g$ in $L^p(\R^d,\C^N)$ 
  which satisfies $v\in W^{1,p}_{\theta}(\R^d,\C^N)$. Moreover, the following estimates hold:
  \begin{align}
    \left\|v\right\|_{L^p_{\theta}} \leqslant& \frac{2C_{0,\varepsilon}}{\Re\lambda+\bzero}\left\|g\right\|_{L^p_{\theta}},                                                     
    \label{equ:ExpDecStatVstarBoundedPerturbation}\\
    \left\|D_i v \right\|_{L^p_{\theta}} \leqslant& \frac{2 C_{1,\varepsilon}}{\left(\Re\lambda+\bzero\right)^{\frac{1}{2}}}\left\|g\right\|_{L^p_{\theta}},\,i=1,\ldots,d.   
    \label{equ:ExpDecStatDiVstarBoundedPerturbation}
  \end{align}
 \end{theorem}

\begin{proof}
Our proof proceeds in three steps.
  % The proof is structured as follows: First we prove the existence and uniqueness in $L^p(\R^d,\C^N)$ by an application of Theorem \ref{thm:LpSolvabilityUniquenessBoundedPerturbation} 
  % (step 1), then we show the existence in $L^p_{\theta}(\R^d,\C^N)$ by a fixed point argument and Theorem \ref{thm:APrioriEstimatesInLpConstantCoefficients} (step 2) and finally, 
  % we derive $L^p_{\theta}$-estimates by the contraction mapping principle and and $W^{1,p}_{\theta}$-estimates by bootstrapping (step 3).
  \begin{itemize}[leftmargin=0.43cm]\setlength{\itemsep}{0.1cm}
  % 1.
  \item[1.] Existence and uniqueness in $L^p(\R^d,\C^N)$ (by Theorem \ref{thm:LpSolvabilityUniquenessBoundedPerturbation}): Since $\theta$ is nondecreasing we have 
  $g\in L^p_{\theta}(\R^d,\C^N)\subseteq L^p(\R^d,\C^N)$, and due to $\Re\lambda\geqslant-(1-\varepsilon)\bzero$ and \eqref{equ:Qsmall} 
% $\left\|Q_{\mathrm{s}}\right\|_{L^{\infty}}\leqslant
%   \frac{\varepsilon\bzero}{2\kappa\aone}$, $\varepsilon>0$ and $\bzero>0$, cf. 
%  \eqref{cond:A9B},
 we have
  \begin{align*}
    \Re\lambda \geqslant -(1-\varepsilon)\bzero \geqslant -\bzero + \frac{\varepsilon}{2}\bzero + \kappa\aone\left\|Q_{\mathrm{s}}\right\|_{L^{\infty}} 
     > -\bzero + \kappa\aone\left\|Q_{\mathrm{s}}\right\|_{L^{\infty}}.
  \end{align*}
  Thus, an application of Theorem \ref{thm:LpSolvabilityUniquenessBoundedPerturbation} implies that there exists a unique solution $v_1\in\D^p_{\mathrm{loc}}(\L_0)$ of 
  $(\lambda I-\L_{\mathrm{s}})v=g$ in $L^p(\R^d,\C^N)$. In order to verify that $v_1$ belongs to $W^{1,p}_{\theta}(\R^d,\C^N)$ and satisfies the inequalities 
  \eqref{equ:ExpDecStatVstarBoundedPerturbation} and \eqref{equ:ExpDecStatDiVstarBoundedPerturbation} we must analyze $(\lambda I-\L_{\mathrm{s}})v=g$ in $L^p_{\theta}(\R^d,\C^N)$.
  % 2.
  \item[2.] Existence in $L^p_{\theta}(\R^d,\C^N)$ (by a fixed point argument): Our aim is to show that the equation
  \begin{align} \label{equ:fixedpointlinear}
    v = \left(\lambda I-\L_{\infty}\right)^{-1}g + \left(\lambda I-\L_{\infty}\right)^{-1}Q_{\mathrm{s}}v =: Fv
  \end{align}
  in $L^p_{\theta}(\R^d,\C^N)$ has a unique fixed point $v_2\in L^p_{\theta}(\R^d,\C^N)$ which even belongs to $\D^p_{\mathrm{loc}}(\L_0)$
and agrees with $v_1$. 
  For this purpose, consider  in $L^p(\R^d,\C^N)$ the equation
  \begin{align}
    \label{equ:uEquation}
    \left(\lambda I-\L_{\infty}\right)u = g+Q_{\mathrm{s}}v,
\quad \text{given} \quad v\in L^p_{\theta}(\R^d,\C^N).
  \end{align}
   First note, that the assumptions of Theorem \ref{thm:LpMaximalDomainPart1} are satisfied. This allows us 
  to write $\L_{\infty}$ and $\D^p_{\mathrm{loc}}(\L_0)$ instead of $\A_p$ and $\D(\A_p)$ in Theorem \ref{thm:APrioriEstimatesInLpConstantCoefficients}. 
  Further, $\Re\lambda\geqslant -(1-\varepsilon)\bzero$ 
and equation \eqref{equ:etarestrict} imply
  \begin{align*}
    0 \leqslant \eta^2 \leqslant \varepsilon^2 \frac{\azero\bzero}{\amax^2 p^2} 
     \leqslant \varepsilon\frac{\azero(\Re\lambda+\bzero)}{\amax^2 p^2}.
  \end{align*}
  Then Theorem \ref{thm:APrioriEstimatesInLpConstantCoefficients} yields a unique solution $u\in\D^p_{\mathrm{loc}}(\L_0)$ of \eqref{equ:uEquation} 
  which satisfies $u\in L^p_{\theta}(\R^d,\C^N)$. This shows that $F$ maps $L^p_{\theta}(\R^d,\C^N)$ into itself and satisfies $Fv\in\D^p_{\mathrm{loc}}(\L_0)$ 
  for every $v\in L^p_{\theta}(\R^d,\C^N)$. Applying $\left(\lambda I-\L_{\infty}\right)^{-1}$ to both sides in \eqref{equ:uEquation} shows $u=Fv$ with $F$ 
  defined in \eqref{equ:fixedpointlinear}. 
  % \begin{align*}
  %   u = \left(\lambda I-\L_{\infty}\right)^{-1}g + \left(\lambda I-\L_{\infty}\right)^{-1}Q_{\mathrm{s}}v =: Fv
  % \end{align*}
  % in $L^p(\R^d,\C^N)$. Hence, we obtain $u=Fv$ in $L^p(\R^d,\C^N)$, 
  Moreover, $Fv\in\D^p_{\mathrm{loc}}(\L_0)\cap L^p_{\theta}(\R^d,\C^N)$. The linear part of $F$ 
  is a contraction due to \eqref{equ:ExpDecStatVstar} and \eqref{equ:Qsmall}
  \begin{align*}
    \left\|\left(\lambda I-\L_{\infty}\right)^{-1}Q_{\mathrm{s}}v\right\|_{L^p_{\theta}}\leqslant q\left\|v\right\|_{L^p_{\theta}}\quad\forall\,v\in L^p_{\theta}(\R^d,\C^N)
  \end{align*}
  with Lipschitz constant
  \begin{align} \label{equ:Lipschitzq}
    0\leqslant q:=\frac{C_{0,\varepsilon}}{\Re\lambda+\bzero}\left\|Q_{\mathrm{s}}\right\|_{L^{\infty}}
    \leqslant \frac{C_{0,\varepsilon}}{\varepsilon\bzero}\left\|Q_{\mathrm{s}}\right\|_{L^{\infty}} 
    \leqslant \frac{1}{2} <1.
  \end{align}
    Consequently, $F$ is a contraction in $L^p_{\theta}(\R^d,\C^N)$.
  % \begin{align*}
  %   \left\|Fv-Fw\right\|_{L^p_{\theta}} = \left\|\left(\lambda I-\L_{\infty}\right)^{-1}Q_{\mathrm{s}}(v-w)\right\|_{L^p_{\theta}}\leqslant q\left\|v-w\right\|_{L^p_{\theta}}
  %   \quad\forall\,v,w\in L^p_{\theta}(\R^d,\C^N).
  % \end{align*}
  Thus, $F$ has a unique fixed point $v_2\in L^p_{\theta}(\R^d,\C^N)$ satisfying 
$v_2=Fv_2\in\D^p_{\mathrm{loc}}(\L_0)$.
 Since $L^p_{\theta}(\R^d,\C^N)\subseteq L^p(\R^d,\C^N)$, the equality $Fv_2=v_2$ holds in $L^p(\R^d,\C^N)$ as well, 
  and applying $\left(\lambda I-\L_{\infty}\right)$ to both sides yields $(\lambda I-\L_{\mathrm{s}})v_2=g$ in $L^p(\R^d,\C^N)$. 
By the unique solvability of this equation we conclude $v:=v_1=v_2\in L^p_{\theta}(\R^d,\C^N)$.
  % 3.
  \item[3.] $L^p_{\theta}$- and $W^{1,p}_{\theta}$-estimates (by contraction mapping principle and bootstrapping): The $L^p_{\theta}$-estimate follows from the contraction mapping 
  principle and the estimates \eqref{equ:ExpDecStatVstar}, \eqref{equ:Lipschitzq}  
  \begin{align*}
     \left\| v \right\|_{L^p_{\theta}} 
    \leqslant \frac{1}{1-q}\left\|F0\right\|
&
\leqslant\frac{2 C_{0,\varepsilon}}{\Re\lambda+\bzero}\left\|g\right\|_{L^p_{\theta}}.
  \end{align*}
  Finally, the $W^{1,p}_{\theta}$-estimate is proved by bootstrapping using the $L^p_{\theta}$-estimate \eqref{equ:ExpDecStatVstarBoundedPerturbation},
the smallness condition  
\eqref{equ:Qsmall} and 
  \eqref{equ:ExpDecStatDiVstar} for every $i=1,\ldots,d$
  \begin{align*}
               \left\|D_i v\right\|_{L^p_{\theta}}
    \leqslant&   
     \frac{C_{1,\varepsilon}}{\left(\Re\lambda+\bzero\right)^{\frac{1}{2}}}\left(\left\|g\right\|_{L^p_{\theta}} + \left\|Q_{\mathrm{s}}\right\|_{L^{\infty}} \left\|v\right\|_{L^p_{\theta}}\right) \\
    \leqslant& \frac{C_{1,\varepsilon}}{\left(\Re\lambda+\bzero\right)^{\frac{1}{2}}}\left(1 + 2 q \right)\left\|g\right\|_{L^p_{\theta}} 
            % =& \frac{C_{1,\varepsilon}}{\left(\Re\lambda+(1-\varepsilon^2)\bzero\right)^{\frac{1}{2}}}\left(\frac{\Re\lambda+\bzero}{\Re\lambda+(1-\varepsilon^2)\bzero}\right)^{\frac{1}{2}}\left\|g\right\|_{L^p_{\theta}} \\
    \leqslant \frac{2 
C_{1,\varepsilon}}{\left(\Re\lambda+\bzero\right)^{\frac{1}{2}}}\left\|g\right\|_{L^p_{\theta}}.
  \end{align*}
  % where we used again $\left\|Q_{\mathrm{s}}\right\|_{L^{\infty}}\leqslant\frac{\varepsilon^2\bzero}{C_{0,\varepsilon}}$ and $\Re\lambda\geqslant -(1-\varepsilon)\bzero$. 
 % This shows that $v\in W^{1,p}_{\theta}(\R^d,\C^N)$.
  \end{itemize}
\end{proof}

%----------------------------------------------------------------------------------
% SUBSECTION 3.3: (Exponentially weighted resolvent estimates for small perturbations).
%----------------------------------------------------------------------------------
\subsection{Exponentially weighted resolvent estimates for variable coefficient operators}
\label{subsec:3.3}
%----------------------------------------------------------------------------------

Consider the differential operator
\begin{align*}
  \left[\L_{B}v\right](x) := A\triangle v(x) +\left\langle Sx,\nabla v(x)\right\rangle - B(x)v(x),\,x\in\R^d.
\end{align*}
The following Lemma \ref{lem:LemmaForUniquenessInDpmax} is crucial to derive 
energy estimates for $\L_{B}$ in exponentially weighted $L^p$-spaces, see Theorem 
\ref{thm:WeightedResolventEstimates} below. The result is proved in \cite[Lemma 4.2]{Otten2015a}, \cite[Lemma 5.12]{Otten2014}, it is a vector-valued and 
complex-valued version of 
\cite[Lemma 2.1]{MetafunePallaraVespri2005}.

\begin{lemma}\label{lem:LemmaForUniquenessInDpmax}
  Let the assumption \eqref{cond:A3} be satisfied for $\K=\C$. Moreover, let $\Omega\subset\R^d$ be a bounded domain with a $C^2$-boundary 
  or $\Omega=\R^d$, $1<p<\infty$, $v\in W^{2,p}(\Omega,\C^N)\cap W^{1,p}_0(\Omega,\C^N)$ and $\eta\in C^1_b(\Omega,\R)$ be nonnegative, then
  \begin{equation*}
    \begin{aligned}
              -\Re\int_{\Omega}\eta\overline{v}^T|v|^{p-2}A\triangle v 
    \geqslant & \Re\int_{\Omega}\eta|v|^{p-2}\sum_{j=1}^{d}\overline{D_j v}^T A D_j v\one_{\{v\neq 0\}}
               +\Re\int_{\Omega}\overline{v}^T\left|v\right|^{p-2}\sum_{j=1}^{d}D_j\eta A D_jv \\ 
             & +(p-2)\Re\int_{\Omega}\eta|v|^{p-4}\sum_{j=1}^{d}\Re\left(\overline{D_j v}^Tv\right)\overline{v}^T A D_j v\one_{\{v\neq 0\}}.
    %\geqslant& (p-1)\Re\int_{\Omega}\eta|v|^{p-2}\sum_{j=1}^{d}\overline{D_j v}^T A D_j v\one_{\{v\neq 0\}}
    %           +\Re\int_{\Omega}\overline{v}^T\left|v\right|^{p-2}\sum_{j=1}^{d}D_j\eta A D_jv \\ 
    %         & +(p-2)\Re\int_{\Omega}\eta|v|^{p-4}\sum_{j=1}^{d}\Bigg[\Re\left(\overline{D_j v}^Tv\right)\overline{v}^T-|v|^2\overline{D_j v}^T \Bigg]A D_j v\one_{\{v\neq 0\}}.
    \end{aligned}
  \end{equation*}
\end{lemma}
Some care has to be taken when using this estimate. By a slight abuse of notation,  the term $|v|^q \one_{\{v\neq0\}}$ in the integrands should be read for powers $q<0$ as follows
\begin{equation*}
\left[|v|^{q} \one_{\{v \neq 0\}}\right](x) =\begin{cases} |v(x)|^{q} , & |v(x)|>0, \\
                                              0 , & v(x) = 0.
                              \end{cases}
\end{equation*}
The proof of Lemma \ref{lem:LemmaForUniquenessInDpmax} shows by using Lebesgue's dominated convergence and Fatou's lemma that the integrals 
involving $ \one_{\{v\neq 0\}}$ exist for $1<p<\infty$, which is nontrivial in case $1<p<2$.

In the following theorem we prove resolvent estimates for $\L_{B}$ in exponentially weighted $L^p$-spaces. The theorem extends  \cite[Theorem 5.13]{Otten2014} 
to variable coefficient perturbations of $\L_0$ and to weighted $L^p$-spaces. Later on, in Theorem \ref{thm:APrioriEstimatesInLpRelativelyCompactPerturbation} we apply 
Theorem \ref{thm:WeightedResolventEstimates} to
 $B(x)=B_{\infty}-Q_{\mathrm{s}}(x)$, so that $\L_B$ agrees with $\L_s$ from
\eqref{equ:operators}.

\begin{theorem}[Resolvent estimates in weighted $L^p$-spaces]\label{thm:WeightedResolventEstimates}
  Assume \eqref{cond:A4DC} for $\K=\C$, $A\in\C^{N,N}$, let $1<p<\infty$
and assume \eqref{cond:A5} for $S\in\R^{d,d}$. Let $B\in L^{\infty}(\R^d,\C^{N,N})$ satisfy
  the strict accretivity condition
  \begin{align}
    \label{equ:StrictAccretivityForB}
    \Re\left\langle w,B(x)w\right\rangle\geqslant c_B|w|^2\;\forall\,x\in\R^d\;\forall\,w\in\C^N,\,\text{for some $c_B\in\R$.}
  \end{align}
  Moreover, let $\lambda\in\C$ with $\Re\lambda+c_B>0$ be given, and  let 
$\theta_1,\theta_2 \in C(\R^d,\R)$
 be positive  weight functions satisfying
  \begin{align}
    \label{equ:ConditionTheta1Version2}
    \theta_1(x) = \exp\left(-\mu_1\sqrt{|x|^2+1}\right)\quad\text{with}\quad
 0\leqslant\left|\mu_1\right|\leqslant \sqrt{\frac{(\Re\lambda+c_B)\gamma_A}{d|A|^2}},
  \end{align}
    \begin{align}
    \label{equ:RelationTheta1Theta2Version2}
    \theta_1(x) \leqslant C\theta_2(x)\;\forall\,x\in\R^d\text{ for some $C>0$,}
  \end{align}
  Finally, let $g\in L^p_{\theta_2}(\R^d,\C^N)$ and let $v\in W^{2,p}_{\mathrm{loc}}(\R^d,\C^N)\cap L^p_{\theta_1}(\R^d,\C^N)$ be a solution of
  \begin{align} \label{equ:resolveLB}
    \left(\lambda I-\L_B\right)v=g\quad\text{in $L^p_{\mathrm{loc}}(\R^d,\C^N)$.}
  \end{align}
  Then, $v$ is the unique solution of \eqref{equ:resolveLB}  in $W^{2,p}_{\mathrm{loc}}(\R^d,\C^N)\cap L^p_{\theta_1}(\R^d,\C^N)$ and satisfies the estimate
  \begin{equation} \label{equ:resolvetheta12}
    \left\|v\right\|_{L^p_{\theta_1}}\leqslant\frac{2C^{\frac{1}{p}}}{\Re\lambda+c_B}\left\|g\right\|_{L^p_{\theta_2}}.
  \end{equation}
  In addition, for $1<p\leqslant 2$ the following gradient estimate holds
  \begin{equation} \label{equ:gradientest}
    \left\|D_iv\right\|_{L^p_{\theta_1}}\leqslant\frac{2C^{\frac{1}{p}}\gamma_A^{-\frac{1}{2}}}{(\Re\lambda+c_B)^{\frac{1}{2}}}\left\|g\right\|_{L^p_{\theta_2}},
\quad i=1,\ldots,d,
  \end{equation}
  with $C$ from \eqref{equ:RelationTheta1Theta2Version2}, $\gamma_A$ from \eqref{cond:A4DC} and $c_B$ from \eqref{equ:StrictAccretivityForB}.
\end{theorem}

\begin{proof}
 Consider $v\in W^{2,p}_{\mathrm{loc}}(\R^d,\C^N)\cap L^p_{\theta_1}(\R^d,\C^N)$ 
which satisfies 
    \eqref{equ:resolveLB}
  for some $g\in L^p_{\theta_2}(\R^d,\C^N)$. For $n\in\R$ with $n>0$ let us define the 
cut-off functions
  \begin{align} \label{equ:definechi}
    \chi_n(x)=\chi_1\left(\frac{x}{n}\right),\quad 
\chi_1\in C_{\mathrm{c}}^{\infty}(\R^d,\R),\quad \chi_1(x)=\begin{cases}
                                                                                                          1                          &,\,|x|\leqslant 1 \\
                                                                                                          \in[0,1],\,\textrm{smooth} &,\,1<|x|<2 \\
                                                                                                          0                          &,\,|x|\geqslant 2
                                                                                                         \end{cases}.
  \end{align}
  \begin{itemize}[leftmargin=0.43cm]\setlength{\itemsep}{0.1cm}
  % 1.
  \item[1.] We multiply \eqref{equ:resolveLB} from left by 
$\chi_n^2\theta_1\overline{v}^T\left|v\right|^{p-2}$, $n\in\N$, 
  integrate over $\R^d$ and take real parts,
  \begin{equation} \label{equ:startenergy}
\begin{aligned}
       \Re\int_{\R^d}\chi_n^2 \theta_1 \left|v\right|^{p-2}\overline{v}^T g
    =& (\Re\lambda)\int_{\R^d}\chi_n^2\theta_1\left|v\right|^p
       - \Re\int_{\R^d}\chi_n^2\theta_1\overline{v}^T\left|v\right|^{p-2} A\triangle v \\
    - &  \Re\int_{\R^d}\chi_n^2\theta_1\overline{v}^T\left|v\right|^{p-2} \sum_{j=1}^{d}(Sx)_j D_j v 
      + \Re\int_{\R^d}\chi_n^2\theta_1\overline{v}^T\left|v\right|^{p-2} B v.
  \end{aligned}
\end{equation}

  % 2.
  \item[2.] Let us rewrite the third term on the right-hand side by using the formula
  \begin{align*}
        D_j\left(\left|v\right|^p\right)= 
        p |v|^{p-2} \Re\left(\overline{D_j v}^T v\right)
  \end{align*}
and the following identity obtained from \eqref{cond:A5} and integration by parts,
  \begin{align*}
    0 =& \frac{1}{p}\int_{\R^d}\chi_n^2\theta_1\big(\sum_{j=1}^{d}S_{jj}\big)\left|v
\right|^p
           =  \frac{1}{p}\sum_{j=1}^{d}\int_{\R^d}\chi_n^2 D_j\left((Sx)_j\right)\theta_1\left|v\right|^p \\
      % =& -\frac{1}{p}\sum_{j=1}^{d}\int_{\R^d}D_j\left(\chi_n^2\right) (Sx)_j \theta_1\left|v\right|^p 
      %    -\frac{1}{p}\sum_{j=1}^{d}\int_{\R^d}\chi_n^2 (Sx)_j \theta_1 D_j\left(\left|v\right|^p\right) \\
      %  & -\frac{1}{p}\sum_{j=1}^{d}\int_{\R^d}\chi_n^2 (Sx)_j (D_j\theta_1) \left|v\right|^p \\
      =& -\frac{2}{p}\sum_{j=1}^{d}\int_{\R^d}\chi_n (D_j \chi_n) (Sx)_j \theta_1 \left|v\right|^p
         -\sum_{j=1}^{d}\int_{\R^d}\chi_n^2 \theta_1 (Sx)_j \Re\left(\overline{D_j v}^T v\right)\left|v\right|^{p-2} \\
       & -\frac{1}{p}\sum_{j=1}^{d}\int_{\R^d}\chi_n^2 (Sx)_j (D_j\theta_1) \left|v\right|^p \\
      =& -\frac{2}{p}\int_{\R^d}\chi_n \theta_1 \left|v\right|^p \sum_{j=1}^{d}(D_j \chi_n)(Sx)_j 
         -\Re\int_{\R^d}\chi_n^2 \theta_1 \overline{v}^T\left|v\right|^{p-2}\sum_{j=1}^{d}(Sx)_j D_j v \\
       & -\frac{1}{p}\int_{\R^d}\chi_n^2 \left|v\right|^p\sum_{j=1}^{d}(Sx)_j (D_j\theta_1) .
  \end{align*}
We insert this into \eqref{equ:startenergy} and apply
  Lemma \ref{lem:LemmaForUniquenessInDpmax} to the second term with $\Omega=B_{2n}(0)$, $\eta=\chi_n^2 \theta_1$
  \begin{align*}
             &   \Re\int_{\R^d}\chi_n^2 \theta_1 \left|v\right|^{p-2} \overline{v}^T g 
            =   \left(\Re\lambda\right)\int_{\R^d}\chi_n^2 \theta_1 \left|v\right|^p 
               - \Re\int_{\R^d}\chi_n^2 \theta_1 \overline{v}^T \left|v\right|^{p-2} A\triangle v \\
             & + \frac{2}{p}\int_{\R^d}\chi_n \theta_1 \left|v\right|^p \sum_{j=1}^{d}(D_j \chi_n) (Sx)_j
               + \frac{1}{p}\int_{\R^d}\chi_n^2\left|v\right|^p \sum_{j=1}^{d}(D_j \theta_1) (Sx)_j \\
             & + \Re\int_{\R^d}\chi_n^2\theta_1\overline{v}^T\left|v\right|^{p-2} B v \\
    \geqslant&   (\Re\lambda)\int_{\R^d}\chi_n^2\theta_1\left|v\right|^p
               + \frac{2}{p}\int_{\R^d}\chi_n\theta_1\left|v\right|^p \sum_{j=1}^{d}(D_j \chi_n)(Sx)_j \\
             & + \frac{1}{p}\int_{\R^d}\chi_n^2\left|v\right|^p \sum_{j=1}^{d}(D_j \theta_1) (Sx)_j  
               + \Re\int_{\R^d}2\chi_n \theta_1\overline{v}^T\left|v\right|^{p-2}\sum_{j=1}^{d}D_j \chi_n A D_j v \\
             & + \Re\int_{\R^d}\chi_n^2\overline{v}^T\left|v\right|^{p-2}\sum_{j=1}^{d}(D_j \theta_1) A D_j v
               + \Re\int_{\R^d}\chi_n^2\theta_1\left|v\right|^{p-2}\sum_{j=1}^{d}\overline{D_j v}^T A D_j v \one_{\{v\neq 0\}} \\
             & + (p-2)\Re\int_{\R^d}\chi_n^2\theta_1\left|v\right|^{p-4}\sum_{j=1}^{d}\Re\left(\overline{D_j v}^T v\right)\overline{v}^T A D_j v \one_{\{v\neq 0\}}
               + \Re\int_{\R^d}\chi_n^2\theta_1\overline{v}^T\left|v\right|^{p-2} B v.
  \end{align*}
  % 3.
  \item[3.] 
Subtracting the $2$nd, $3$rd, $4$th and $5$th term of the right hand side, yields the upper bound 
%Rearranging the last
%three terms of the right-hand side and adding the first term of the right-hand
%side, we obtain the following  upper bound:
  \begin{align*}
             & (\Re\lambda)\int_{\R^d}\chi_n^2\theta_1\left|v\right|^p 
               + \Re\int_{\R^d}\chi_n^2\theta_1\left|v\right|^{p-2}\sum_{j=1}^{d}\overline{D_j v}^T A D_j v \one_{\{v\neq 0\}} \\
             & + (p-2)\Re\int_{\R^d}\chi_n^2\theta_1\left|v\right|^{p-4}\sum_{j=1}^{d}\Re\left(\overline{D_j v}^T v\right)\overline{v}^T
                 A D_j v \one_{\{v\neq 0\}}
              + \Re\int_{\R^d}\chi_n^2\theta_1\overline{v}^T\left|v\right|^{p-2} B v \\
    \leqslant& \Re\int_{\R^d}\chi_n^2\theta_1 \left|v\right|^{p-2} \overline{v}^T g 
               - \Re\int_{\R^d}2\chi_n\theta_1 \overline{v}^T\left|v\right|^{p-2}\sum_{j=1}^{d}D_j \chi_n A D_j v \\
             & - \frac{2}{p} \int_{\R^d}\chi_n\theta_1\left|v\right|^p \sum_{j=1}^{d}(D_j \chi_n)(Sx)_j 
               - \frac{1}{p} \int_{\R^d}\chi_n^2\left|v\right|^p\sum_{j=1}^{d}(D_j\theta_1)(Sx)_j \\
             & - \Re\int_{\R^d}\chi_n^2\overline{v}^T\left|v\right|^{p-2}\sum_{j=1}^{d}(D_j\theta_1) A D_j v =: T_1 + T_2 + T_3 +T_4+ T_5.
  \end{align*}
We estimate the terms successively. Using $\Re z\leqslant |z|$ and \eqref{equ:RelationTheta1Theta2Version2}, H{\"o}lder's inequality yields  
%We estimate the terms successively. From $\Re z\leqslant |z|$ and \eqref{equ:RelationTheta1Theta2Version2} we obtain by H{\"o}lder's inequality  
  \begin{align*}
            T_1= & 
              \int_{\R^d}\chi_n^2 \theta_1 \left|v\right|^{p-2}\Re\left(\overline{v}^T g\right) \leqslant
     \int_{\R^d}\chi_n^2 \theta_1 \left|v\right|^{p-1}\left|g\right|\\
    \leqslant & \left(\int_{\R^d}\left(\chi_n^{\frac{2(p-1)}{p}}\theta_1^{\frac{p-1}{p}}\left|v\right|^{p-1}\right)^{\frac{p}{p-1}}\right)^{\frac{p-1}{p}} 
               \left(\int_{\R^d}\left(\chi_n^{\frac{2}{p}}\theta_1^{\frac{1}{p}}\left|g\right|\right)^p\right)^{\frac{1}{p}} \\
    \leqslant& C^{\frac{1}{p}}\left(\int_{\R^d}\chi_n^{2}\theta_1\left|v\right|^p\right)^{\frac{p-1}{p}} \left(\int_{\R^d}\chi_n^2\theta_2\left|g\right|^p\right)^{\frac{1}{p}}.
  \end{align*}
  For the $2$nd term we use  H{\"o}lder's inequality with $p=q=2$ and Young's inequality with $\delta>0$
  %Similarly, for the $2$nd term we use  H{\"o}lder's inequality with $p=q=2$ and Young's inequality with some $\delta>0$
  \begin{align*}
          T_2 \leqslant& 2|A|\int_{\R^d}\chi_n\theta_1\left|v\right|^{p-1}\sum_{j=1}^{d}\left|D_j \chi_n\right| \left|D_j v\right|
    \leqslant  \frac{2|A| \left\|\chi_1\right\|_{1,\infty}}{n}\sum_{j=1}^{d}\int_{\R^d}\chi_n\theta_1\left|D_j v\right| \left|v\right|^{p-1} \\
    \leqslant& \frac{2|A| \left\|\chi_1\right\|_{1,\infty}}{n}\sum_{j=1}^{d}\left(\int_{\R^d}\chi_n^2\theta_1\left|D_j v\right|^2 \left|v\right|^{p-2}\one_{\{v\neq 0\}}\right)^{\frac{1}{2}}
               \left(\int_{\R^d}\theta_1\left|v\right|^p\right)^{\frac{1}{2}} \\
    \leqslant& \frac{2|A| \left\|\chi_1\right\|_{1,\infty}\delta}{n}\sum_{j=1}^{d}\int_{\R^d}\chi_n^2\theta_1\left|D_j v\right|^2 \left|v\right|^{p-2}\one_{\{v\neq 0\}}
               +\frac{2d|A| \left\|\chi_1\right\|_{1,\infty}}{4n\delta}\int_{\R^d}\theta_1\left|v\right|^p.
  \end{align*}
  Here we used that for every $x\in\R^d$ and $j=1,\ldots,d$
  \begin{align*}
              \left|D_j \chi_n(x)\right| 
            = \left|D_j\left(\chi_1\left(\frac{x}{n}\right)\right)\right|
               \leqslant \frac{1}{n} \max_{j=1,\ldots,d} \max_{y\in\R^d} \left|D_j \chi_1(y)\right| 
            = \frac{\left\|\chi_1\right\|_{1,\infty}}{n}.
  \end{align*}
  For the $3$rd term we use  $\chi_n(x)=0$ for $|x|\geqslant 2n$ and 
 $D_j \chi_n(x)=0$ for $|x|\leqslant n$ to obtain
  \begin{align*}
          T_3 \leqslant &
    \  \frac{2}{p} \sum_{j=1}^{d} \int_{\R^d}\chi_n\theta_1 \left|v\right|^p \left|(Sx)_j\right| \left|D_j \chi_n\right| \\
            =& \frac{2}{p} \sum_{j=1}^{d} \int_{n\leqslant|x|\leqslant 2n}\chi_n \theta_1\left|v\right|^p \left|(Sx)_j\right| \left|D_j \chi_n\right| 
    \leqslant  \frac{4d\left|S\right|\left\|\chi_1\right\|_{1,\infty}}{p} \int_{n\leqslant|x|\leqslant 2n}\theta_1\left|v\right|^p.
  \end{align*}
  For the last estimate note that $\chi_n(x)\leqslant 1$ and
  \begin{align*}
             & \left|(Sx)_j\right| \left|D_j \chi_n(x)\right| 
            =  \frac{1}{n} \left|(Sx)_j\right| \left|\left(D_j \chi_1\right)\left(\frac{x}{n}\right)\right|
    \leqslant  \frac{1}{n} |S| |x| \left|\left(D_j \chi_1\right)\left(\frac{x}{n}\right)\right| \\
    \leqslant& \frac{|S|}{n} \big(\sup_{n\leqslant|\xi|\leqslant 2n}\left|\xi \right|\big) \max_{j=1,\ldots,d} \max_{y\in\R^d} \left|D_j \chi_1(y)\right|
            =  2\left|S\right|\left\|\chi_1\right\|_{1,\infty}.
  \end{align*}
  The $4$th term vanishes, as follows from  \eqref{equ:ConditionTheta1Version2} and \eqref{cond:A5},
  \begin{align*}
    T_4
    = - \frac{1}{p} \int_{\R^d}\chi_n^2\frac{-\mu_1}{\sqrt{|x|^2+1}}\theta_1\left|v\right|^p\sum_{j=1}^{d}x_j(Sx)_j
    = 0.
  \end{align*}
  For the $5$th term we use again $\Re z\leqslant |z|$, H{\"o}lder's inequality with $p=q=2$ and Young's inequality with some $\rho>0$, \eqref{equ:ConditionTheta1Version2} and $|\mu_1|\leqslant\mu_0$ 
  for some $\mu_0\geqslant 0$ that will be specified below
  \begin{align*}
           T_5 \leqslant  &       
\int_{\R^d}\chi_n^2\left|v\right|^{p-1}\sum_{j=1}^{d}\left|\frac{-\mu_1 x_j}{\sqrt{|x|^2+1}}\right|\theta_1 |A|\left|D_j v\right| 
    \leqslant |\mu_1||A|\sum_{j=1}^{d}\int_{\R^d}\chi_n^2\theta_1\left|v\right|^{p-1}\left|D_j v\right|\\
    \leqslant & |\mu_1||A|\sum_{j=1}^{d}\left(\int_{\R^d}\chi_n^2\theta_1\left|v\right|^{p-2}\left|D_j v\right|^2 \one_{\{v\neq 0\}}\right)^{\frac{1}{2}}
               \left(\int_{\R^d}\chi_n^2\theta_1\left|v\right|^p\right)^{\frac{1}{2}} \\
    \leqslant& \frac{\mu_0|A|}{4\rho}\sum_{j=1}^{d}\int_{\R^d}\chi_n^2\theta_1\left|v\right|^{p-2}\left|D_j v\right|^2 \one_{\{v\neq 0\}}
               + \mu_0|A|\rho d\int_{\R^d}\chi_n^2\theta_1\left|v\right|^p.
  \end{align*}
  Summarizing, we arrive at the following estimate
  \begin{align*}
             & (\Re\lambda)\int_{\R^d}\chi_n^2\theta_1\left|v\right|^p 
               + \int_{\R^d}\chi_n^2\theta_1\left|v\right|^{p-2}\Re\left\langle v,Bv\right\rangle \\
            + & 
\int_{\R^d}\chi_n^2\theta_1\left|v\right|^{p-4} \one_{\{v\neq 0\}}\sum_{j=1}^{d}\bigg[\left|v\right|^2\Re\left\langle D_j v,A D_j v\right\rangle
+ (p-2)\Re\left\langle D_j v,v\right\rangle \Re\left\langle v,A D_j v\right\rangle\bigg]
                \\
    \leqslant& C^{\frac{1}{p}}\left(\int_{\R^d}\chi_n^{2}\theta_1\left|v\right|^p\right)^{\frac{p-1}{p}} \left(\int_{\R^d}\chi_n^2\theta_2\left|g\right|^p\right)^{\frac{1}{p}}
               + \frac{2d|A| \left\|\chi_1\right\|_{1,\infty}}{4n\delta}\int_{\R^d}\theta_1\left|v\right|^p
 \\ &  
 + \frac{4d\left|S\right|\left\|\chi_1\right\|_{1,\infty}}{p} \int_{n\leqslant|x|\leqslant 2n}\theta_1\left|v\right|^p+
\frac{2|A| \left\|\chi_1\right\|_{1,\infty}\delta}{n}\sum_{j=1}^{d}\int_{\R^d}\chi_n^2\theta_1\left|D_j v\right|^2 \left|v\right|^{p-2}\one_{\{v\neq 0\}} \\
            + &  \frac{\mu_0|A|}{4\rho}\sum_{j=1}^{d}\int_{\R^d}\chi_n^2\theta_1\left|v\right|^{p-2}\left|D_j v\right|^2 \one_{\{v\neq 0\}}
               + \mu_0|A|\rho d\int_{\R^d}\chi_n^2\theta_1\left|v\right|^p.
  \end{align*}
  % 4.
  \item[4.] The $L^p$-dissipativity assumption \eqref{cond:A4DC} guarantees positivity of the term appearing in brackets $\left[\cdots\right]$ and \eqref{equ:StrictAccretivityForB} 
  provides a lower bound for $\Re\left\langle v,Bv\right\rangle$. Therefore, putting the last $3$ terms from the right-hand to the 
left-hand 
  side in the last inequality from step 3, we obtain
  \begin{align*}
             & \left(\gamma_A-\frac{\mu_0|A|}{4\rho}-\frac{2|A|\left\|\chi_1\right\|_{1,\infty}\delta}{n}\right)
                 \sum_{j=1}^{d}\int_{\R^d}\chi_n^2 \theta_1\left|D_j v\right|^2\left|v\right|^{p-2}\one_{\{v\neq 0\}}
 \\ 
               + & (\Re\lambda+c_B-\mu_0|A|\rho d)\int_{\R^d}\chi_n^2\theta_1\left|v\right|^p  
    \leqslant C^{\frac{1}{p}}\left(\int_{\R^d}\chi_n^2\theta_1\left|v\right|^p\right)^{\frac{p-1}{p}} \left(\int_{\R^d}\chi_n^2\theta_2\left|g\right|^p\right)^{\frac{1}{p}}
             \\  + & \frac{2d|A| \left\|\chi_1\right\|_{1,\infty}}{4n\delta}\int_{\R^d}\theta_1\left|v\right|^p 
            + \frac{4d\left|S\right|\left\|\chi_1\right\|_{1,\infty}}{p} \int_{n\leqslant|x|\leqslant 2n}\theta_1\left|v\right|^p.
  \end{align*}
  Now, we choose $\rho=\sqrt{\frac{\Re\lambda+c_B}{4d\gamma_A}}$, $\mu_0=\sqrt{\frac{(\Re\lambda+c_B)\gamma_A}{d|A|^2}}$ so that
  \begin{align*}
    \Re\lambda+c_B-\mu_0|A|\rho d =\frac{\Re\lambda+c_B}{2}\quad\text{and}\quad\gamma_A-\frac{\mu_0|A|}{4\rho}=\frac{\gamma_A}{2}.
  \end{align*}
  Then our estimate reads
\begin{equation} \label{equ:endenergy}
  \begin{aligned}
             & \frac{\Re\lambda+c_B}{2}\int_{\R^d}\chi_n^2\theta_1\left|v\right|^p 
               + \left(\frac{\gamma_A}{2}-\frac{2|A|\left\|\chi_1\right\|_{1,\infty}\delta}{n}\right)\sum_{j=1}^{d}\int_{\R^d}\chi_n^2 \theta_1\left|D_j v\right|^2\left|v\right|^{p-2}\one_{\{v\neq 0\}} \\
    \leqslant& C^{\frac{1}{p}}\left(\int_{\R^d}\chi_n^2\theta_1\left|v\right|^p\right)^{\frac{p-1}{p}} \left(\int_{\R^d}\chi_n^2\theta_2\left|g\right|^p\right)^{\frac{1}{p}}
               + \frac{2d|A| \left\|\chi_1\right\|_{1,\infty}}{4n\delta}\int_{\R^d}\theta_1\left|v\right|^p \\
             & + \frac{4d\left|S\right|\left\|\chi_1\right\|_{1,\infty}}{p} \int_{n\leqslant|x|\leqslant 2n}\theta_1\left|v\right|^p.
  \end{aligned}
\end{equation}

  \item[5.] Let us choose $\delta>0$ such that $\frac{\gamma_A}{2}-2|A|\left\|\chi_1\right\|_{1,\infty}\delta>0$.
Then we apply Fatou's Lemma to \eqref{equ:endenergy}
and take the limit inferior $n\to\infty$. First observe
that the terms   $\chi_n^2\theta_1|v|^p$ and $\chi_n ^2\theta_1\left(\frac{\gamma_A}{2}-\frac{2|A|\left\|\eta\right\|_{1,\infty}\delta}{n}\right)\left|D_j v\right|^2 \left|v\right|^{p-2}\one_{\{v\neq 0\}}$ on the left-hand side are positive functions in $L^1(\R^d,\R)$
and converge pointwise.
The convergence of
the integrals on the right-hand side of \eqref{equ:endenergy}
is justified by Lebesgue's dominated convergence theorem.
We have the pointwise convergence $\chi_n^2\theta_1|v|^p\to\theta_1|v|^p$, 
  $\chi_n^2\theta_2|g|^p\to\theta_2|g|^p$, $\frac{1}{n}\theta_1|v|^p\to 0$ and $\theta_1\left|v\right|^p\one_{\{n\leqslant|x|\leqslant 2n\}}\to 0$ for almost every 
  $x\in\R^d$ as $n\to\infty$. They are dominated by $|\chi_n^2\theta_1|v|^p|\leqslant\theta_1|v|^p$, $|\chi_n^2\theta_2|g|^p|\leqslant\theta_2|g|^p$, 
  $\frac{1}{n}\theta_1|v|^p\leqslant\theta_1|v|^p$, $\theta_1\left|v\right|^p\one_{\{n\leqslant|x|\leqslant 2n\}}\leqslant\theta_1|v|^p$, and the 
  bounds belong to $L^1(\R^d,\R)$ since $v\in L^p_{\theta_1}(\R^d,\C^N)$ and $g\in L^p_{\theta_2}(\R^d,\C^N)$.
Thus we arrive at
  \begin{align*}
             & \frac{\Re\lambda+c_B}{2}\left\|v\right\|_{L^p_{\theta_1}}^p
    \leqslant  \frac{\Re\lambda+c_B}{2}\int_{\R^d}\theta_1\left|v\right|^p
               + \frac{\gamma_A}{2}\sum_{j=1}^{d}\int_{\R^d}\theta_1\left|D_j v\right|^2\left|v\right|^{p-2}\one_{\{v\neq 0\}} \\
            \leqslant& C^{\frac{1}{p}}\left(\int_{\R^d}\theta_1\left|v\right|^p\right)^{\frac{p-1}{p}} \left(\int_{\R^d}\theta_2\left|g\right|^p\right)^{\frac{1}{p}}
            =  C^{\frac{1}{p}}\left\|v\right\|_{L^p_{\theta_1}}^{p-1} \left\|g\right\|_{L^p_{\theta_2}}.
  \end{align*}
  The $L^p_{\theta_1}$--resolvent estimate \eqref{equ:resolvetheta12}
follows by dividing both sides by $\frac{\Re\lambda+c_B}{2}$ and 
  $\left\|v\right\|_{L^p_{\theta_1}}^{p-1}$. 
  \\
  % 6.
  \item[6.] Unique solvability of the linear equation $(\lambda I-\L_B)v = g$ in $W^{2,p}_{\mathrm{loc}}(\R^d,\C^N)\cap L^p_{\theta_1}(\R^d,\C^N)$ clearly follows from the resolvent estimate
\eqref{equ:resolvetheta12}. From step 5 we obtain for every $j=1,\ldots,N$
  \begin{align*}
    \int_{\R^d}\theta_1\left|D_j v\right|^2\left|v\right|^{p-2}\one_{\{v\neq 0\}} \leqslant \frac{2C^{\frac{1}{p}}}{\gamma_A}\left\|v\right\|_{L^p_{\theta_1}}^{p-1}
    \left\|g\right\|_{L^p_{\theta_2}}.
  \end{align*}
 We take into account that $|D_jv|=|D_jv|\one_{\{v\neq 0\}}$ a.e.
(see e.g. \cite[Cor.2.1.8]{Ziemer1989}) and use the $L^p$--resolvent estimate \eqref{equ:resolvetheta12} to deduce from H\"older's inequality for $1<p\leqslant 2$
  \begin{align*}
             & \left\|D_j v\right\|_{L^p_{\theta_1}}^p
            =  \int_{\R^d}\theta_1\left|D_j v\right|^p \one_{\{v\neq 0\}}
            =  \int_{\R^d}\theta_1^{\frac{p}{2}}\left|D_j v\right|^p \left|v\right|^{-\frac{p(2-p)}{2}}\one_{\{v\neq 0\}} \theta_1^{\frac{2-p}{2}}\left|v\right|^{\frac{p(2-p)}{2}} \\
    \leqslant& \bigg(\int_{\R^d}\theta_1\left|D_j v\right|^2 \left|v\right|^{p-2}\one_{\{v\neq 0\}}\bigg)^{\frac{p}{2}}
               \bigg(\int_{\R^d}\theta_1\left|v\right|^{p}\bigg)^{\frac{2-p}{2}}
    \leqslant  \left(\frac{4C^{\frac{2}{p}}}{(\Re\lambda+c_B)\gamma_A}\right)^{\frac{p}{2}}\left\|g\right\|_{L^p_{\theta_2}}^p.
  \end{align*}
  %Taking the sum over $j$ from $1$ to $d$ and the $p$th root we end up with \eqref{equ:gradientest}.
  % \begin{align*}
  %     \left|v\right|_{W^{1,p}_{\theta_1}(\R^d,\C^N)}
  %   = \bigg(\sum_{j=1}^{d}\left\|D_j v\right\|_{L^p_{\theta_1}(\R^d,\C^N)}^p\bigg)^{\frac{1}{p}}
  %   \leqslant \frac{2C^{\frac{1}{p}}d^{\frac{1}{p}}\gamma_A^{-\frac{1}{2}}}{(\Re\lambda+c_B)^{\frac{1}{2}}}\left\|g\right\|_{L^p_{\theta_2}(\R^d,\C^N)}.
  % \end{align*}
  \end{itemize}
\end{proof}

%----------------------------------------------------------------------------------
% SUBSECTION 3.4: (Compactly supported perturbations).
%----------------------------------------------------------------------------------
\subsection{Exponential decay for asymptotically small perturbations}
\label{subsec:3.4}
%----------------------------------------------------------------------------------
In this section we combine the results of Theorems \ref{thm:APrioriEstimatesInLpSmallPerturbation} and \ref{thm:WeightedResolventEstimates} 
to obtain exponential a-priori estimates of solutions to variable coefficient equations when the coefficients become small at infinity.

\begin{theorem}[A-priori estimates in weighted $L^p$-spaces]\label{thm:APrioriEstimatesInLpRelativelyCompactPerturbation}
  Let the assumptions \eqref{cond:A4DC}, \eqref{cond:A5}, \eqref{cond:A8B} and \eqref{cond:A10B} be satisfied for $1<p<\infty$ 
  and $\K=\C$. Consider the radial weight functions 
  \begin{align}
    \label{equ:WeightFunctionsLQ}
    \theta_j(x) = \exp\left(\mu_j\sqrt{|x|^2+1}\right),\;x\in\R^d,\;j=1,2,
  \end{align}
  with $\mu_1,\mu_2\in\R$, 
  \begin{align}
    \label{equ:ExponentialRatesLQ}
  -\sqrt{\varepsilon\frac{\gamma_A\beta_{\infty}}{2d|A|^2}}\leqslant\mu_1\leqslant
 0\leqslant\mu_2\leqslant\varepsilon\frac{\sqrt{\azero\bzero}}{\amax p}
  \end{align}
  for some $0<\varepsilon<1$. Moreover, let $Q\in L^{\infty}(\R^d,\C^{N,N})$ with
  \begin{align}
    \label{equ:ConditionOnQ}
    \underset{|x|\geqslant R_0}{\esssup}\left|Q(x)\right| \leqslant
\frac{\varepsilon}{2} \min\left\{\frac{\bzero}{\kappa\aone},\frac{\bzero}{C_{0,\varepsilon}},
\beta_{\infty}\right\}\,\text{for some $R_0>0$},
  \end{align}
  let $g\in L^p_{\theta_2}(\R^d,\C^N)$ and $\lambda\in\C$ be given
with $\Re\lambda\geqslant -(1-\varepsilon)\beta_{\infty}$ .\\
  Then every solution $v\in W^{2,p}_{\mathrm{loc}}(\R^d,\C^N)\cap L^p_{\theta_1}(\R^d,\C^N)$ of the resolvent equation $\left(\lambda I-\L_Q\right)v=g$ in 
  $L^p_{\mathrm{loc}}(\R^d,\C^N)$ satisfies $v\in W^{1,p}_{\theta_2}(\R^d,\C^N)$. Moreover, the following estimates hold:
  \begin{align}
    \left\|v\right\|_{L^p_{\theta_2}} \leqslant& \frac{2C_{0,\varepsilon}}{\Re\lambda+\bzero}\left(C\left\|v\right\|_{L^p_{\theta_1}}
                                                                    +\left\|g\right\|_{L^p_{\theta_2}}\right), \label{equ:equ:ExpDecStatVstarVariableCoefficientPerturbation}\\
    \left\|D_i v\right\|_{L^p_{\theta_2}} \leqslant& \frac{2C_{1,\varepsilon}}{\left(\Re\lambda+\bzero\right)^{\frac{1}{2}}}\left(C\left\|v\right\|_{L^p_{\theta_1}}+\left\|g\right\|_{L^p_{\theta_2}}\right), \quad i=1,\ldots,d, \label{equ:equ:equ:ExpDecStatDiVstarVariableCoefficientPerturbation}   
  \end{align}
  with constants $\azero,\aone,\amax$ from \eqref{equ:aminamaxazerobzero}, $\gamma_A$ from \eqref{cond:A4DC}, $\bzero,\kappa$ from \eqref{equ:constant2}, 
  $\beta_{\infty}$ from \eqref{cond:A10B}, $C_{0,\varepsilon},C_{1,\varepsilon}$ from Theorem \ref{thm:APrioriEstimatesInLpConstantCoefficients} (with $C_{\theta}=1$), 
  and $C:=\exp\left((\mu_2-\mu_1)(4R_0^2+1)^{\frac{1}{2}}\right)\left\|Q\right\|_{L^{\infty}}$.
  %where $K:=\exp\left((\mu_2-\mu_1)(4R_0^2+1)^{\frac{1}{2}}\right)\left\|Q\right\|_{L^{\infty}}$, and the constants $C_{0,\varepsilon},C_{1,\varepsilon},
  %\azero,\aone,\amax,b_0$,$\beta_{\infty}$ are from Theorem \ref{thm:APrioriEstimatesInLpConstantCoefficients} (with $C_{\theta}=1$), equations
  %\eqref{equ:aminamaxazerobzero},\eqref{equ:constant2} and Assumption \eqref{cond:A10B}.
\end{theorem}
\begin{remark} Note that the exponential decay rate $\mu_2$ in
  \eqref{equ:ExponentialRatesLQ} depend on the spectral data $a_0,b_0,\amax$, while
the growth rate $\theta_1$ allowing uniqueness, depends on the norm and accretivity
data $\gamma_A,\beta_{\infty},|A|$.
  \end{remark}
\begin{proof}
  The proof is structured as follows: First we decompose of $Q$ into the sum of $Q_{\mathrm{s}}$ and $Q_{\mathrm{c}}$, where $Q_{\mathrm{s}}$ is small according to 
  \eqref{equ:ConditionOnQ} and $Q_{\mathrm{c}}$ is compactly supported on $\R^d$ (step 1). We then consider the equation $(\lambda I-\L_{\mathrm{s}})u=Q_{\mathrm{c}} v+g$ 
  and prove existence of a solution in the smaller space $W^{2,p}_{\mathrm{loc}}(\R^d,\C^N)\cap L^p_{\theta_2}(\R^d,\C^N)$ by an application of Theorem 
  \ref{thm:APrioriEstimatesInLpSmallPerturbation} (step 2) and uniqueness  in the larger space $W^{2,p}_{\mathrm{loc}}(\R^d,\C^N)\cap L^p_{\theta_1}(\R^d,\C^N)$ 
  by an application of Theorem \ref{thm:WeightedResolventEstimates} (step 3).
  
  \begin{itemize}[leftmargin=0.43cm]\setlength{\itemsep}{0.1cm}
  % 1.
  \item[1.] Decomposition of $Q$:  With the cut-off function
$\chi_{R_0}$ from \eqref{equ:definechi} and $R_0$ from
\eqref{equ:ConditionOnQ} let us write 
% For $R>0$ choose a $C^{\infty}$ cut-off function
%   \begin{align*}
%     \chi_R:[0,\infty)\rightarrow[0,1],\quad\chi_R(r)=\begin{cases}0 &,\,r\leqslant R \\ \text{smooth} &,\, R\leqslant r\leqslant 2R \\ 1 &,\, r\geqslant 2R\end{cases}.
%   \end{align*}
    \begin{align*}
    Q(x) = Q_{\mathrm{s}}(x) + Q_{\mathrm{c}}(x),\quad Q_{\mathrm{s}}(x):=(1-\chi_{R_0}(x))Q(x),\quad  Q_{\mathrm{c}}(x):=\chi_{R_0}(x)Q(x).
  \end{align*}
  Then $Q_{\mathrm{c}}$ is compactly supported 
  and $Q_{\mathrm{s}}$ satisfies due to \eqref{equ:ConditionOnQ}
  \begin{align*}
              \left\|Q_{\mathrm{s}}\right\|_{L^{\infty}}
           \leqslant \left\|1-\chi_{R_0}\right\|_{\infty} \left\|Q\right\|_{L^{\infty}(\R^d\backslash B_{R_0},\C^{N,N})} 
    \leqslant \frac{\varepsilon}{2} \min\left\{\frac{\bzero}{\kappa\aone},\frac{\bzero}{C_{0,\varepsilon}},
\beta_{\infty}\right\}.
  \end{align*}
Let $v\in W^{2,p}_{\mathrm{loc}}(\R^d,\C^N)\cap L^p_{\theta_1}(\R^d,\C^N)$ be a solution of $\left(\lambda I-\L_Q\right)v=g$
   with $\Re\lambda\geqslant -(1-\varepsilon)\beta_{\infty}$ and
 $g\in L^p_{\theta_2}(\R^d,\C^N)$. 
  Then $v$ 
satisfies $\left(\lambda I-\L_{\mathrm{s}}\right)v = Q_{\mathrm{c}} v+g$ 
in $L^p_{\mathrm{loc}}(\R^d,\C^N)$.
Therefore, we consider the problem
  \begin{align}
    \label{equ:ustarEquation}
    \left(\lambda I-\L_{\mathrm{s}}\right)u = Q_{\mathrm{c}} v+g,\,\text{in $L^p_{\theta_1}(\R^d,\C^N)$ and in $L^p_{\theta_2}(\R^d,\C^N)$.}
  \end{align}
  
  % 2.
  \item[2.] Existence in $W^{2,p}_{\mathrm{loc}}(\R^d,\C^N)\cap L^p_{\theta_2}(\R^d,\C^N)$:  
  Let us to  apply Theorem \ref{thm:APrioriEstimatesInLpSmallPerturbation} to 
\eqref{equ:ustarEquation} with $\theta=\theta_2$, $\eta=|\mu_2|$ and $Q_{\mathrm{c}} v+g$ instead of $g$. 
  Note that $Q_{\mathrm{c}}v+g \in L^p_{\theta_2}(\R^d,\C^N)$ follows from
  \begin{align}
    \label{equ:EstimateInhomogeneity}
    \begin{split}
             & \left\|Q_{\mathrm{c}} v+g\right\|_{L^p_{\theta_2}}
     \leqslant  \left\|\theta_2 Q_{\mathrm{c}} v\right\|_{L^p}+\left\|g\right\|_{L^p_{\theta_2}} \\
     \leqslant& 
\left\|\theta_2\theta_1^{-1}\right\|_{L^{\infty}(B_{2R_0},\R)} 
\left\|Q\right\|_{L^{\infty}} 
\left\|v\right\|_{L^p_{\theta_1}}+\left\|g\right\|_{L^p_{\theta_2}} 
             = C\left\|v\right\|_{L^p_{\theta_1}}+\left\|g\right\|_{L^p_{\theta_2}},
    \end{split}
  \end{align}
  with $C:=\exp\left((\mu_2-\mu_1)(4R_0^2+1)^{\frac{1}{2}}\right)\left\|Q\right\|_{L^{\infty}}$. Further, \eqref{cond:A9B} follows from \eqref{cond:A10B}, $\beta_{\infty}\leqslant\bzero$ and $\varepsilon<1$  
  imply $\Re\lambda\geqslant -(1-\varepsilon)\bzero$ and  $\theta_2$ is radially nondecreasing since $\mu_2\geqslant 0$. 
Theorem 
  \ref{thm:APrioriEstimatesInLpSmallPerturbation} yields that there exists a (unique)  $u\in\D^p_{\mathrm{loc}}(\L_0) \subset W^{2,p}_{\mathrm{loc}}(\R^d,\C^N)\cap L^p(\R^d,\C^N)$ which solves 
   \eqref{equ:ustarEquation} in $L^p(\R^d,\C^N)$. 
Moreover, Theorem \ref{thm:APrioriEstimatesInLpSmallPerturbation} assures $u\in W^{1,p}_{\theta_2}(\R^d,\C^N)$ as well as the  estimates \eqref{equ:ExpDecStatVstarBoundedPerturbation} 
  and \eqref{equ:ExpDecStatDiVstarBoundedPerturbation}.
  % 3.
  \item[3.] Uniqueness in $W^{2,p}_{\mathrm{loc}}(\R^d,\C^N)\cap L^p_{\theta_1}(\R^d,\C^N)$: Now consider 
  \eqref{equ:ustarEquation} in $L^p_{\theta_1}(\R^d,\C^N)$. 
  We apply Theorem \ref{thm:WeightedResolventEstimates} with $B(x)=B_{\infty}-Q_{\mathrm{s}}(x)$ and $Q_{\mathrm{c}} v+g$ instead of $g$. 
  First, $B\in L^{\infty}(\R^d,\C^{N,N})$ follows from $B_{\infty}\in\C^{N,N}$ and $Q_{\mathrm{s}}\in L^{\infty}(\R^d,\C^{N,N})$.
Then, strict accretivity \eqref{equ:StrictAccretivityForB} with $c_B=\left(1-\frac{\varepsilon}{2}\right)\beta_{\infty}$
  is a consequence of \eqref{cond:A10B} and $\left\|Q_{\mathrm{s}}\right\|_{L^{\infty}}\leqslant\varepsilon\frac{\beta_{\infty}}{2}$,
  \begin{align*}
    \Re\left\langle w,B(x)w\right\rangle 
                                         \geqslant& \left(\beta_{\infty}-\left\|Q_{\mathrm{s}}\right\|_{L^{\infty}}\right)|w|^2
                                         \geqslant \left(1-\frac{\varepsilon}{2}\right)\beta_{\infty}|w|^2\quad\forall\,w\in\C^N\;\forall\,x\in\R^d.
  \end{align*}
  Moreover, we have $\Re\lambda \geqslant -\left(1-\varepsilon\right)\beta_{\infty}\geqslant -\left(1-\frac{\varepsilon}{2}\right)\beta_{\infty} = -c_B$. 
  The growth bound in \eqref{equ:ConditionTheta1Version2} is implied by \eqref{equ:ExponentialRatesLQ} and  $\Re\lambda+c_B\geqslant \frac{\varepsilon}{2}\beta_{\infty}$,
  \begin{align*}
    0 \leqslant |\mu_1| \leqslant \sqrt{\varepsilon\frac{\beta_{\infty}\gamma_A}{2d|A|^2}} \leqslant \sqrt{\frac{(\Re\lambda+c_B)\gamma_{A}}{d|A|^2}}.
    %0 \leqslant \mu_1^2 \leqslant \varepsilon\frac{\beta_{\infty}\gamma_A}{2d|A|^2} \leqslant \frac{(\Re\lambda+c_B)\gamma_{A}}{d|A|^2}.
  \end{align*}
  Finally, inequality \eqref{equ:RelationTheta1Theta2Version2} is obvious with $C=1$, since $\mu_1\leqslant 0\leqslant\mu_2$.
 % Note that $\mu_1\leqslant 0$ implies that 
 %  $\theta_1$ is radially nonincreasing and therefore, $L^p(\R^d,\C^N)\subset L^p_{\theta_1}(\R^d,\C^N)$.\\
  Let $u\in W^{2,p}_{\mathrm{loc}}(\R^d,\C^N)\cap L^p_{\theta_2}(\R^d,\C^N)\subseteq W^{2,p}_{\mathrm{loc}}(\R^d,\C^N)\cap L^p_{\theta_1}(\R^d,\C^N)$ denote the solution from step $2$  of the equation
   $\left(\lambda I-\L_{\mathrm{s}}\right)u = Q_{\mathrm{c}} v+g$ in $L^p_{\mathrm{loc}}(\R^d,\C^N)$. Since the given   
  $v\in W^{2,p}_{\mathrm{loc}}(\R^d,\C^N)\cap L^p_{\theta_1}(\R^d,\C^N)$ 
solves the same equation, the difference $w=u-v$ solves the
homogeneous equation $(\lambda I- \L_{\mathrm{c}})w=0$
in $L^p_{\theta_1}(\R^d,\C^N)$, 
and Theorem \ref{thm:WeightedResolventEstimates} implies
   $\left\|w\right\|_{L^p_{\theta_1}}=0$.  Therefore,
we obtain 
  $v=u\in W^{1,p}_{\theta_2}(\R^d,\C^N)$.
  % 4.
  \item[4.] $L^p_{\theta_2}$- and $W^{1,p}_{\theta_2}$-estimates: 
  The $L^p_{\theta_2}$-estimate follows from \eqref{equ:ExpDecStatVstarBoundedPerturbation} and \eqref{equ:EstimateInhomogeneity}
  \begin{align*}
               \left\|v\right\|_{L^p_{\theta_2}} = \left\|u\right\|_{L^p_{\theta_2}}
    \leqslant& \frac{2C_{0,\varepsilon}}{\Re\lambda+\bzero}\left\|Q_{\mathrm{c}} v+g\right\|_{L^p_{\theta_2}} 
    \leqslant \frac{2C_{0,\varepsilon}}{\Re\lambda+\bzero}\left(C\left\|v\right\|_{L^p_{\theta_1}}+\left\|g\right\|_{L^p_{\theta_2}}\right).
  \end{align*}
  Analogously, the $W^{1,p}_{\theta_2}$-estimate follows from  \eqref{equ:ExpDecStatDiVstarBoundedPerturbation} and \eqref{equ:EstimateInhomogeneity}.
  \end{itemize}
\end{proof}

%\begin{remark}
%  radially: Necessary for the heat kernel estimates.
% (ii) nondecreasing & nonincreasing: can be omitted if semigroup approach (i.e. solvability & identification problem) is performed in exponentially 
%      weighted spaces
% (iii) (A11B) implies (A10B), but (A10B) implies only that there exists an inner product such that (A11B) is satisfied. In order to guarantee equivalence
%      one has to use the H-inner product <z,w>:=<z,Hw>, z,w\in\C^N  for some hermitian, positive-definite Matrix H. Due to (A9B) we believe that the 
%      H generating the H-inner products in (A3) and (A10B) are the same. One has to extend the resolvent estimates and the characterization of the 
%      first antieigenvalue.
%\end{remark}

%---------------------------------------------------------------------------------------------------------------------------------------------------
%
%  SECTION 4: (Proof of exponential decay and applications to complex-valued systems)
%
%---------------------------------------------------------------------------------------------------------------------------------------------------
\sect{Exponential decay of rotating nonlinear waves}
\label{sec:4}
%---------------------------------------------------------------------------------------------------------------------------------------------------

%----------------------------------------------------------------------------------
% SUBSECTION 4.1: (Proof of main result).
%----------------------------------------------------------------------------------
\subsection{Proof of main result}
\label{subsec:4.1}
%----------------------------------------------------------------------------------

\begin{proof}[Proof (of Theorem \ref{thm:NonlinearOrnsteinUhlenbeckSteadyState})]
  The proof is structured as follows: First we decompose the nonlinearity $f(v_{\star}(x))$ and derive an equation $\L_Q w_{\star}=0$ 
  solved by $w_{\star}=v_{\star}-v_{\infty}$ (step 1). Then we apply Theorem \ref{thm:APrioriEstimatesInLpRelativelyCompactPerturbation} (step 2) and check 
its assumptions (steps 3,4).

  \begin{itemize}[leftmargin=0.43cm]\setlength{\itemsep}{0.1cm}
  % 1.
  \item[1.] Let $v_{\star}$ be a classical solution of \eqref{equ:NonlinearProblemRealFormulation} satisfying \eqref{equ:BoundednessConditionForVStar}. 
  Note that this implies $v_{\star}\in C_{\mathrm{b}}(\R^d,\R^N)$. Using  \eqref{cond:A6} and \eqref{cond:A7} we obtain from the mean value theorem 
  \begin{align*}
    f(v_{\star}(x)) =& f(v_{\infty}) + Df(v_{\infty})\left(v_{\star}(x)-v_{\infty}\right) \\
                     &  +\int_{0}^{1}\left(Df(v_{\infty}+t(v_{\star}(x)-v_{\infty}))-Df(v_{\infty})\right)dt \left(v_{\star}(x)-v_{\infty}\right) \\
                    =& -B_{\infty}\left(v_{\star}(x)-v_{\infty}\right)+Q(x)\left(v_{\star}(x)-v_{\infty}\right),\,x\in\R^d
  \end{align*}
  with
  \begin{align}
    \label{equ:MatrixBandQ}
    B_{\infty}:=-Df(v_{\infty}),\quad Q(x):=\int_{0}^{1}\left(Df(v_{\infty}+t(v_{\star}(x)-v_{\infty}))-Df(v_{\infty})\right)dt.
  \end{align}
   For $w_{\star}:=v_{\star}-v_{\infty}$, we have $w_{\star}\in C^2(\R^d,\R^N)\cap C_{\mathrm{b}}(\R^d,\R^N)$ and 
  \begin{align*}
    0 =& A\triangle v_{\star}(x) + \left\langle Sx,\nabla v_{\star}(x)\right\rangle + f(v_{\star}(x)) \\
      =& A\triangle \left(v_{\star}(x)-v_{\infty}\right) + \left\langle Sx,\nabla \left(v_{\star}(x)-v_{\infty}\right)\right\rangle
         -B_{\infty}\left(v_{\star}(x)-v_{\infty}\right) + Q(x)\left(v_{\star}(x)-v_{\infty}\right) \\
      =& A\triangle w_{\star}(x) + \left\langle Sx,\nabla w_{\star}(x)\right\rangle -B_{\infty}w_{\star}(x) + Q(x)w_{\star}(x) = \left[\L_{Q}w_{\star}\right](x),\,x\in\R^d.
  \end{align*}
  % 2.
  \item[2.] Let us apply Theorem \ref{thm:APrioriEstimatesInLpRelativelyCompactPerturbation} with $B_{\infty},Q$ from \eqref{equ:MatrixBandQ}, $\theta_2=\theta$, 
  $\mu_2=\mu$, $\mu_1<0$, $\lambda=0$ and $g=0$. For this purpose, we have to check the assumptions: 
  Assumptions \eqref{cond:A4DC} and \eqref{cond:A5} are directly satisfied, \eqref{cond:A8B} follows from \eqref{cond:A8}, and \eqref{cond:A10B} from \eqref{cond:A10}, 
  using the relation $B_{\infty}=-Df(v_{\infty})$. In the following let $0<\varepsilon<1$ be fixed and let $\theta_1,\theta_2\in C(\R^d,\R)$ be given by \eqref{equ:WeightFunctionsLQ} 
  satisfying $\mu_1<0$ and \eqref{equ:ExponentialRatesLQ}. First, note that $w_{\star}\in W^{2,p}_{\mathrm{loc}}(\R^d,\C^N)\cap L^p_{\theta_1}(\R^d,\C^N)$ follows from 
  $w_{\star}\in C^2(\R^d,\R^N)\cap C_{\mathrm{b}}(\R^d,\R^N)$ and $C_{\mathrm{b}}(\R^d,\R^N)\subset L^p_{\theta_1}(\R^d,\C^N)$ due to $\mu_1 <0$. It remains to verify that 
  $Q\in L^{\infty}(\R^d,\C^{N,N})$ (step 3) and that \eqref{equ:ConditionOnQ} is satisfied.
  % 3.
  \item[3.]  Since $w_{\star}\in C_{\mathrm{b}}(\R^d,\R^N)$ we obtain
  \begin{align*}
              \left|v_{\infty}+t w_{\star}(x)\right|
    \leqslant \left|v_{\infty}\right|+t\left|w_{\star}(x)\right|
    \leqslant \left|v_{\infty}\right|+\left\|w_{\star}\right\|_{\infty} =: R_1
  \end{align*}
  for every $x\in\R^d$ and $0\leqslant t\leqslant 1$. Due to \eqref{cond:A6}, we have $f\in C^1(\R^N,\R^N)$ which implies 
  \begin{align*}
              \left|Q(x)\right|
    \leqslant  \int_{0}^{1}\left|Df(v_{\infty}+t w_{\star}(x))\right| + \left|Df(v_{\infty})\right| dt 
    \leqslant  \sup_{z\in B_{R_1}(0)}\left|Df(z)\right| + \left|Df(v_{\infty})\right| < \infty
  \end{align*}
  for all $x\in\R^d$. Taking the suprema over $x\in\R^d$ we find 
  $Q\in C_{\mathrm{b}}(\R^d,\R^{N,N}) \subset L^{\infty}(\R^d,\C^{N,N})$.
  % 4.
  \item[4.] We finally verify \eqref{equ:ConditionOnQ}: Let us choose $K_1=K_1(A,f,v_{\infty},d,p,\varepsilon)>0$ such that
  \begin{align}
    \label{equ:thresholdconstant}
    K_1\left(\sup_{z\in B_{K_1}(v_{\infty})}\left|D^2 f(z)\right|\right)\leqslant
    \varepsilon \min\left\{\frac{\bzero}{\kappa\aone},
    \frac{\bzero}{C_{0,\varepsilon}},\beta_{\infty}\right\}=: K(\varepsilon)
  \end{align}
  is satisfied, with the constants $C_{0,\varepsilon}=C_{0,\varepsilon}(A,d,p,\varepsilon,\kappa)$ from Theorem \ref{thm:APrioriEstimatesInLpConstantCoefficients}, 
  $\bzero:=-s(Df(v_{\infty}))$ and $\aone$ from \eqref{equ:aminamaxazerobzero}, $\beta_{\infty}:=\beta_{-Df(v_{\infty})}$ from \eqref{cond:A10}, and
  \begin{align*}
    \left|D^2 f(z)\right| := \left\|D^2 f(z)\right\|_{\L(\R^N,\R^{N,N})} := \sup_{v\in\R^N\atop |v|=1}\left|D^2 f(z)v\right|.
  \end{align*}
  Since $f\in C^2(\R^N,\R^N)$ by \eqref{cond:A6}, inequalities  \eqref{equ:BoundednessConditionForVStar} and \eqref{equ:thresholdconstant} lead to
  \begin{align*}
               \left|Q(x)\right|
            =& \left|\int_{0}^{1}Df(v_{\infty}+tw_{\star}(x))-Df(v_{\infty})dt\right|
            %=& \left|\int_{0}^{1}\int_{0}^{1}D^2 f(v_{\infty}+s(v_{\infty}+t w_{\star}(x)-v_{\infty}))ds(v_{\infty}+t w_{\star}(x)-v_{\infty})dt\right| \\
            = \left|\int_{0}^{1}\int_{0}^{1}D^2 f(v_{\infty}+st w_{\star}(x))[tw_{\star}(x)]ds dt\right| \\
    \leqslant& \int_{0}^{1}\int_{0}^{1}\sup_{|x|\geqslant R_0}\left|D^2 f(v_{\infty}+st(v_{\star}(x)-v_{\infty}))\right| ds\cdot t 
               \sup_{|x|\geqslant R_0}\left|v_{\star}(x)-v_{\infty}\right| dt \\
    \leqslant& \frac{K_1}{2}\left(\sup_{z\in B_{K_1}(v_{\infty})}\left|D^2 f(z)\right|\right)
    \leqslant \frac{\varepsilon}{2} \min\left\{\frac{\bzero}{\kappa\aone},\frac{\bzero}{C_{0,\varepsilon}},\beta_{\infty}\right\}
  \end{align*}
  for every $|x|\geqslant R_0$. Taking the supremum over $|x|\geqslant R_0$ yields condition \eqref{equ:ConditionOnQ}.
%   \begin{align*}
%     \sup_{|x|\geqslant R_0}\left|Q(x)\right| \leqslant \frac{\varepsilon}{2} \min\left\{\frac{\bzero}{\kappa\aone},\varepsilon\frac{\bzero}{C_{0,\varepsilon}},
% \beta_{\infty}\right\}.
%   \end{align*}
  % 6.
  % \item[6.] Now we verify that $w_{\star}\in W^{2,p}_{\mathrm{loc}}(\R^d,\C^N)\cap L^p_{\theta_1}(\R^d,\C^N)$: From $w_{\star}\in C^2(\R^d,\R^N)\cap C_{\mathrm{b}}(\R^d,\R^N)$,
  % $C^2(\R^d,\R^N)\subset W^{2,p}_{\mathrm{loc}}(\R^d,\R^N)$ and $C_{\mathrm{b}}(\R^d,\R^N)\subset L^p_{\theta_1}(\R^d,\R^N)$ (since $\mu_1<0$ implies that $\theta_1$ is decreasing) 
  % we directly deduce $w_{\star}\in W^{2,p}_{\mathrm{loc}}(\R^d,\R^N)\cap L^p_{\theta_1}(\R^d,\R^N)$.
  % % 7.
  % \item[7.] Finally, we verify $\L_{Q} w_{\star}=0$ in $L^p_{\mathrm{loc}}(\R^d,\C^N)$: Since $w_{\star}$ satisfies $\left[\L_{Q}w_{\star}\right](x)=0$ for every 
  % $x\in\R^d$ (step 2), it follows obviously that $w_{\star}$ satisfies $\L_{Q} w_{\star}=0$ in $L^p_{\mathrm{loc}}(\R^d,\R^N)$.
  \end{itemize}
  This justifies the application of Theorem \ref{thm:APrioriEstimatesInLpRelativelyCompactPerturbation} which then shows $w_{\star}=v_{\star}-v_{\infty}\in W^{1,p}_{\theta}(\R^d,\R^N)$.
\end{proof}

%----------------------------------------------------------------------------------
% SUBSECTION 4.2: (Exponential decay of higher order derivatives).
%----------------------------------------------------------------------------------
\subsection{Exponential decay of higher order derivatives}
\label{subsec:4.2}
%----------------------------------------------------------------------------------

For estimating higher order derivatives, recall the Sobolev embedding for $0\leqslant l\leqslant k$,
\begin{equation} 
  \label{eq:sobolevembed}
  W^{k,p}(\R^d,\R^N)\subseteq W^{l,q}(\R^d,\R^N), \quad \text{if} \; 1<p<q\leqslant\infty,\quad \frac{d}{p}-k\leqslant\frac{d}{q}-l,
\end{equation}
where at least one of the inequalities '$\leqslant$' is strict. Moreover, the embedding is continuous, i.e.
\begin{align*}
  \exists\,C_{p,q,k,l}>0:\;\left\|u\right\|_{W^{l,q}(\R^d)}\leqslant C_{p,q,k,l}\left\|u\right\|_{W^{k,p}(\R^d)}\;\forall\,u\in W^{k,p}(\R^d).
\end{align*}
For the Sobolev embedding we refer to \cite[Theorem 5.4]{Adams1975}, \cite[Chapter 6]{Nikolskij1975}, \cite[Chapter 8]{Alt2006} as general reference, and to 
\cite[Theorem 3, Exercise 24]{Tao2009} for the compact version used in \eqref{eq:sobolevembed}.
%See \cite[Theorem 5.4]{Adams1975}, \cite[Chapter 6]{Nikolskij1975}, \cite[Chapter 8]{Alt2006} as general reference. For the compact version of Sobolev embedding used here 
%see \cite[Theorem 3, Exercise 24]{Tao2009}. 
For the corresponding weighted spaces it is important to note that
\begin{equation} 
  \label{eq:sobolevweights}
  u \in W_{\theta}^{k,p}(\R^d,\R^N) \Longrightarrow \theta D^{\alpha}u \in W^{k-|\alpha|,p}(\R^d,\R^N) \quad \text{for} \quad 0 \leqslant |\alpha| \leqslant k.
\end{equation}
Here $\theta$ is chosen as in Theorem \ref{thm:NonlinearOrnsteinUhlenbeckSteadyState} for some $0<\varepsilon<1$. By definition, $u\in W^{k,p}_{\theta}(\R^d,\R^N)$ 
implies $D^{\alpha}u\in W^{k-|\alpha|,p}_{\theta}(\R^d,\R^N)$ for every $0\leqslant|\alpha|\leqslant k$. Since $\theta$ belongs to $C^{\infty}(\R^d,\R)$ and satisfies 
$\left\|\frac{D^{\gamma}\theta}{\theta}\right\|_{L^{\infty}}\leqslant C(\gamma)$ for every $\gamma\in\N_0^d$, we obtain $\theta D^{\alpha}u \in W^{k-|\alpha|,p}(\R^d,\R^N)$ 
from
\begin{align*}
             \left\|\theta D^{\alpha}u\right\|_{W^{k-|\alpha|,p}}^p 
          =& \sum_{|\beta|\leqslant k-|\alpha|}\left\|D^{\beta}(\theta D^{\alpha}u)\right\|_{L^p}^p
          =  \sum_{|\beta|\leqslant k-|\alpha|}\bigg\|\sum_{|\gamma|\leqslant|\beta|}\binom{\beta}{\gamma} (D^{\gamma}\theta) (D^{\alpha+\beta-\gamma}u)\bigg\|_{L^p}^p \\
  \leqslant& \sum_{|\beta|\leqslant k-|\alpha|}\bigg(\sum_{|\gamma|\leqslant|\beta|}\binom{\beta}{\gamma}\left\|\frac{D^{\gamma}\theta}{\theta}\right\|_{L^{\infty}} 
             \left\|\theta D^{\alpha+\beta-\gamma}u\right\|_{L^p}\bigg)^p\leqslant C \|u\|_{W_{\theta}^{k,p}}^p.
\end{align*}
Finally, recall the generalized H\"older's inequality
\begin{equation} 
  \label{eq:hoeldergeneral}
  \big\| \prod_{j=1}^{\ell} u_j \big\|_{L^p(\R^d,\R)} \le \prod_{j=1}^{\ell} \|u_j\|_{L^{p_j}(\R^d,\R)},
  \quad \text{for} \; u_j \in L^{p_j}(\R^d,\R), \; 1 \leqslant p,\,p_j \leqslant \infty, \; \sum_{j=1}^{\ell} \frac{1}{p_j} = \frac{1}{p}.
\end{equation}

The following corollary shows, that if we assume more regularity for the solution $v_{\star}$ of \eqref{equ:NonlinearProblemRealFormulation} in Theorem 
\ref{thm:NonlinearOrnsteinUhlenbeckSteadyState}, i.e. $v_{\star}\in C^3(\R^d,\R^N)$, then $v_{\star}$ even belongs to $W^{2,p}_{\theta}(\R^d,\R^N)$. 
The proof is based on the results of Theorem \ref{thm:NonlinearOrnsteinUhlenbeckSteadyState} and on a further application of Theorem 
\ref{thm:APrioriEstimatesInLpRelativelyCompactPerturbation}. For the proof it is crucial that we allow inhomogeneities $g$ in Theorem 
\ref{thm:APrioriEstimatesInLpRelativelyCompactPerturbation}. The argument can be continued to higher order weighted Sobolev spaces.

\begin{corollary}[Exponential decay of $v_{\star}$ with higher regularity]\label{cor:NonlinearOrnsteinUhlenbeckSteadyStateMoreRegularity}
  Let the assumptions \eqref{cond:A4DC}, \eqref{cond:A5}--\eqref{cond:A8} and \eqref{cond:A10} be satisfied for $\K=\R$ and for some $1<p<\infty$. 
  Moreover, let $\amax=\rho(A)$ denote the spectral radius of $A$, $-\azero=s(-A)$ the spectral bound of $-A$ and $-\bzero=s(Df(v_{\infty}))$ 
  the spectral bound of $Df(v_{\infty})$. Further, let $\theta(x)=\exp\left(\mu\sqrt{|x|^2+1}\right)$ denote a weight function for $\mu\in\R$.
  Then, for every $0<\varepsilon<1$ there is a constant $K_1=K_1(A,f,v_{\infty},d,p,\varepsilon)>0$ with the following property: 
  Every classical solution $v_{\star}$ of
  \begin{align}
    \label{equ:NonlinearProblemRealFormulationMoreRegularity}
    A\triangle v(x)+\left\langle Sx,\nabla v(x)\right\rangle+f(v(x))=0,\,x\in\R^d,
  \end{align}
  with $v_{\star}\in C^3(\R^d,\R^N)$ and
  \begin{align}
    \label{equ:BoundednessConditionForVStarMoreRegularity}
    \sup_{|x|\geqslant R_0}\left|v_{\star}(x)-v_{\infty}\right|\leqslant K_1\text{ for some $R_0>0$}
  \end{align}
  satisfies
  \begin{align*}
    v_{\star}-v_{\infty}\in W^{2,p}_{\theta}(\R^d,\R^N)
  \end{align*}
  for every exponential decay rate
  \begin{align*}
    0\leqslant\mu\leqslant\varepsilon\frac{\sqrt{\azero\bzero}}{\amax p}.
  \end{align*}
  If additionally, $p\geqslant\frac{d}{2}$ and $f\in C^{k-1}(\R^N,\R^N)$, $v_{\star}\in C^{k+1}(\R^d,\R^N)$ for some $k\in\N$ with $k\geqslant 3$, 
  then $v_{\star}$ even satisfies $v_{\star}-v_{\infty}\in W^{k,p}_{\theta}(\R^d,\R^N)$.
\end{corollary}

\begin{proof}
  % The proof is structured as follows: The additional regularity allows us to derive an equation $\L_Q w_{\star}=g$  solved by $w_{\star}:=D_i v_{\star}$. Then we apply again 
  % Theorem \ref{thm:APrioriEstimatesInLpRelativelyCompactPerturbation}. This needs $v_{\star}\in W^{1,p}_{\theta}$ which we already know from 
  % Theorem \ref{thm:NonlinearOrnsteinUhlenbeckSteadyState}. The 
  % additional part follows by iterating this procedure.
  Let $v_{\star}$ be a classical solution of \eqref{equ:NonlinearProblemRealFormulationMoreRegularity} satisfying \eqref{equ:BoundednessConditionForVStarMoreRegularity}. 
  Again this implies $v_{\star}\in C_{\mathrm{b}}(\R^d,\R^N)$.
  \begin{itemize}[leftmargin=0.43cm]\setlength{\itemsep}{0.1cm}
  % 1.
  \item[1.] The additional regularity $v_{\star}\in C^3(\R^d,\R^N)$ allows us to apply $D_i=\frac{\partial}{\partial x_i}$ to equation \eqref{equ:NonlinearProblemRealFormulationMoreRegularity} 
  \begin{align*}
    0 =&
     A\triangle D_i v_{\star}(x)+\left\langle Sx,\nabla D_i v_{\star}(x)\right\rangle+ Df(v_{\star}(x))D_i v_{\star}(x) +\sum_{j=1}^{d}S_{ji}D_j v_{\star},\,x\in\R^d.
  \end{align*}
  For $w_{\star}:=D_i v_{\star} \in C^2(\R^d,\R^N)$  we obtain using  $S_{ii}=0$,
  \begin{align*}
    0 = A\triangle w_{\star}(x)+\left\langle Sx,\nabla w_{\star}(x)\right\rangle+ Df(v_{\star}(x))w_{\star}(x) +\sum_{\substack{j=1\\j\neq i}}^{d}S_{ji}D_j v_{\star} 
      = \left[\L_Q w_{\star}\right](x) + g(x),\,x\in\R^d
  \end{align*}
  with the settings
  \begin{align}
    \label{equ:MatrixBandQMoreRegularity}
    B_{\infty} := -Df(v_{\infty}),\quad Q(x) := Df(v_{\star}(x))-Df(v_{\infty}),\quad g(x):=\sum_{\substack{j=1\\j\neq i}}^d S_{ji}D_j v_{\star}.
  \end{align}
  % 2.
  \item[2.] We now apply Theorem \ref{thm:APrioriEstimatesInLpRelativelyCompactPerturbation} with $B_{\infty},Q,g$ from \eqref{equ:MatrixBandQMoreRegularity}, 
  $\theta_2=\theta$, $\mu_2=\mu$, $\mu_1<0$ and $\lambda=0$.  The assumptions \eqref{cond:A4DC} and \eqref{cond:A5} are directly satisfied while \eqref{cond:A8B}, 
  \eqref{cond:A10B} follow from \eqref{cond:A8}, \eqref{cond:A10}. In the following let $0<\varepsilon<1$ be fixed and let $\theta_1,\theta_2\in C(\R^d,\R)$ be given 
  by \eqref{equ:WeightFunctionsLQ} satisfying $\mu_1<0\leqslant\mu_2$ and \eqref{equ:ExponentialRatesLQ}. 
  %As in the proof of Theorem \ref{thm:NonlinearOrnsteinUhlenbeckSteadyState}, the relation $w_{\star}\in W^{2,p}_{\mathrm{loc}}(\R^d,\C^N)\cap L^p_{\theta_1}(\R^d,\C^N)$ 
  %is a consequence of $D_i v_{\star}\in C^2(\R^d,\R^N)\cap C_{\mathrm{b}}(\R^d,\R^N)$ and the fact that $\theta_1$ is decreasing.
  The relation $w_{\star}\in W^{2,p}_{\mathrm{loc}}(\R^d,\C^N)\cap L^p_{\theta_1}(\R^d,\C^N)$ is a consequence of Theorem \ref{thm:NonlinearOrnsteinUhlenbeckSteadyState} 
  which implies $D_i v_{\star}\in L^p_{\theta}(\R^d,\R^N)\subseteq L^p_{\theta_1}(\R^d,\R^N)$ since $\theta_1$ is decreasing. Moreover, $-\L_{Q} w_{\star}=g$ holds in 
  $L^p_{\mathrm{loc}}(\R^d,\C^N)$ by construction. Then, $Q\in L^{\infty}(\R^d,\C^{N,N})$ follows from \eqref{equ:MatrixBandQMoreRegularity} since $v_{\star}$ is bounded, 
  and we also have $g\in L^p_{\theta}(\R^d,\C^N)$ since Theorem \ref{thm:NonlinearOrnsteinUhlenbeckSteadyState} shows $v_{\star}\in W^{1,p}_{\theta}(\R^d,\R^N)$ which leads 
  to the estimate
  \begin{align*}
    \left\|g\right\|_{L^p_{\theta}} \leqslant \sum_{\substack{j=1\\j\neq i}}^d|S_{ji}|\left\|D_j v_{\star}\right\|_{L^p_{\theta}}\leqslant C.
  \end{align*}
  % 3. 
  \item[3.] We finally verify condition \eqref{equ:ConditionOnQ}: Let us choose $K_1=K_1(A,f,v_{\infty},d,p,\varepsilon)>0$ such that  \eqref{equ:thresholdconstant} holds 
  with $\frac{K(\varepsilon)}{2}$ instead of $K(\varepsilon)$ on the right-hand side. Since $f\in C^2(\R^N,\R^N)$ by \eqref{cond:A6}, equations \eqref{equ:BoundednessConditionForVStar} 
  and \eqref{equ:thresholdconstant} imply for all $|x|\geqslant R_0$
  \begin{align*}
               \left|Q(x)\right|
            =& \left|Df(v_{\star}(x))-Df(v_{\infty})\right|
    \leqslant  \int_{0}^{1}\left|D^2 f(v_{\infty}+s\left(v_{\star}(x)-v_{\infty}\right)\right| ds \left|v_{\star}(x)-v_{\infty}\right| \\
      \leqslant& K_1\left(\sup_{z\in B_{K_1}(v_{\infty})}\left|D^2 f(z)\right|\right)
    \leqslant \frac{\varepsilon}{2}  \min\left\{\frac{\bzero}{\kappa\aone},\frac{\bzero}{C_{0,\varepsilon}},\beta_{\infty}\right\}.
  \end{align*}
  Taking the supremum over $|x|\geqslant R_0$ we have shown \eqref{equ:ConditionOnQ} with $R_0$ from \eqref{equ:BoundednessConditionForVStar}. By applying 
  Theorem \ref{thm:APrioriEstimatesInLpRelativelyCompactPerturbation} we obtain $w_{\star}=D_i v_{\star}\in W^{1,p}_{\theta}(\R^d,\R^N)$ for every $i=1,\ldots,d$, 
  thus $v_{\star}-v_{\infty}\in W^{2,p}_{\theta}(\R^d,\R^N)$. 
  % 4.
  \item[4.] For the final assertion we consider $f\in C^{k-1}(\R^N,\R^N)$ and $v_{\star}\in C^{k+1}(\R^d,\R^N)$ for some $k\in\N$ with $k\geqslant 3$ and show 
  $v_{\star}-v_{\infty}\in W^{k,p}_{\theta}(\R^d,\R^N)$ by induction with respect to $k$. Let $\alpha\in\N_0^d$ be a multi-index of length $|\alpha|=k-1$ for 
  some $k\geqslant 3$. Applying $D^{\alpha}$ to \eqref{equ:NonlinearProblemRealFormulationMoreRegularity} yields that $w_{\star}:=D^{\alpha}v_{\star}$ satisfies 
  \begin{align}
    \label{equ:LQGleichung}
    0=[\L_Q w_{\star}](x)+g(x),\,x\in\R^d
  \end{align}
  with $B_{\infty}$ and $Q(x)$ as in \eqref{equ:MatrixBandQMoreRegularity}, and $g(x):=g_1(x)+g_2(x)$ defined by
  \begin{align}
    \label{equ:g1g2}
    \begin{split}
    g_1(x) :=& \sum_{i=1}^{d} \sum_{\substack{j=1\\e_j\leqslant\alpha}}^{d} S_{ij} \binom{\alpha}{e_j} D^{\alpha-e_j+ei}v_{\star}(x), \\ 
    g_2(x) :=& \sum_{\ell=2}^{k-1}\sum_{\pi\in\mathcal{P}_{\ell,k-1}}(D^{\ell}f)(v_{\star}(x))\left[D^{|\pi_1|}v_{\star}(x)h_{\pi_1},\ldots,
                                                                                                                            D^{|\pi_{\ell}|}v_{\star}(x)h_{\pi_{\ell}}\right].
    \end{split}
  \end{align}
  The first term $g_1$ arises from the Leibniz rule applied to $\langle Sx,\nabla v_{\star}(x)\rangle$,
  \begin{equation}
    \label{equ:DalphaRotationalTerm}
    \begin{aligned}
       & D^{\alpha}(\left\langle Sx,\nabla v_{\star}(x)\right\rangle)
      =   \sum_{i=1}^{d}\sum_{j=1}^{d} S_{ij}D^{\alpha}(x_j D_i v_{\star}(x)) \\
    %=  \sum_{i=1}^{d}\sum_{j=1}^{d} S_{ij}\sum_{\substack{\beta\leqslant\alpha\\\beta\in \N_0^d}} \binom{\alpha}{\beta} (D^{\beta} x_j)(D^{\alpha-\beta}D_i v_{\star}(x)) \nonumber\\
   % =& \left\langle Sx,\nabla D^{\alpha} v_{\star}(x)\right\rangle 
    %   + \sum_{i=1}^{d}\sum_{j=1}^{d} S_{ij}\sum_{\substack{\beta\leqslant\alpha\\|\beta|=1\\\beta\in\N_0^d}} \binom{\alpha}{\beta} \underbrace{(D^{\beta} x_j)}_{=0\text{, for }|\beta|\geqslant 2}(D^{\alpha-\beta}D_i v_{\star}(x)) \nonumber\\
   % =& \left\langle Sx,\nabla D^{\alpha} v_{\star}(x)\right\rangle 
    %   + \sum_{i=1}^{d}\sum_{j=1}^{d} S_{ij}\sum_{\substack{k=1\\e_k\leqslant\alpha}}^d \binom{\alpha}{e_k} \underbrace{(D^{e_k} x_j)}_{=\delta_{jk}}(D^{\alpha-e_k+e_i}v_{\star}(x)) \label{equ:DalphaRotationalTerm}\\
      =& \left\langle Sx,\nabla D^{\alpha} v_{\star}(x)\right\rangle
         + \sum_{i=1}^{d}\sum_{\substack{j=1\\e_j\leqslant\alpha}}^{d} S_{ij} \binom{\alpha}{e_j} D^{\alpha-e_j+e_i}v_{\star}(x) 
      =  \left\langle Sx,\nabla D^{\alpha} v_{\star}(x)\right\rangle + g_1(x).
    \end{aligned}
  \end{equation}
  The second term is obtained by applying Fa\'{a} di Bruno's formula for multivariate calculus,
  \begin{equation}
    \label{equ:DalphaNonlinearity}
    \begin{aligned}
       & D^{\alpha}(f(v_{\star}(x)))
      =  \sum_{\ell=1}^{k-1}\sum_{\pi\in\mathcal{P}_{\ell,k-1}}(D^{\ell}f)(v_{\star}(x))\left[D^{|\pi_1|}v_{\star}(x)h_{\pi_1},\ldots,
                                                                          D^{|\pi_{\ell}|}v_{\star}(x)h_{\pi_{\ell}}\right] \\
      =& Df(v_{\star}(x))D^{\alpha}v_{\star}(x) 
         + \sum_{\ell=2}^{k-1}\sum_{\pi\in\mathcal{P}_{\ell,k-1}}(D^{\ell}f)(v_{\star}(x))\left[D^{|\pi_1|}v_{\star}(x)h_{\pi_1},\ldots,
                                                                          D^{|\pi_{\ell}|}v_{\star}(x)h_{\pi_{\ell}}\right] \\
      =& Df(v_{\star}(x))D^{\alpha}v_{\star}(x) + g_2(x),\,x\in\R^d.
    \end{aligned}
  \end{equation}
  Here $\mathcal{P}_{\ell,k-1}$ denotes the set of all $\ell$-partitions of the set $\langle k-1\rangle := \{1,\ldots,k-1\}$, given by
  \begin{align*}
    \mathcal{P}_{\ell,k-1}= \bigg\{\pi=\{\pi_1,\ldots,\pi_{\ell}\} \subset 2^{\langle k-1 \rangle}:\; 
    \bigcup_{j=1}^{\ell}\pi_j = \langle k-1 \rangle,\; \pi_i \cap \pi_j = \emptyset\; \forall i \neq j \bigg\}.
  \end{align*}
  Moreover, we introduced the multilinear argument
  \begin{align} \label{equ:multilineararg}
    (h_1,\ldots,h_{k-1}) = (\underbrace{e_1,\ldots,e_1}_{\alpha_1},\ldots,\underbrace{e_d,\ldots,e_d}_{\alpha_d}), \quad e_j=j\text{-th unit vector in }\R^d,
  \end{align}
  and the short-hand $h_{\rho}=(h_{\rho_1}, \ldots h_{\rho_\nu})$ for index sets $\rho = \{\rho_{1}, \ldots \rho_{\nu}\} \subseteq \langle k-1 \rangle$.
  For $\ell=1$ the only partition is $\pi_1=\{\langle k-1\rangle \}$ and $Df(v_{\star})[D^{k-1}v_{\star} h_{\pi_1}]$ agrees with $Df(v_{\star})D^{\alpha}v_{\star}$.
  % Adding up the equations \eqref{equ:DalphaLaplacian}, \eqref{equ:DalphaRotationalTerm}, \eqref{equ:DalphaNonlinearity} and $\pm Df(v_{\infty})D^{\alpha}v_{\star}$ implies 
  %\eqref{equ:LQGleichung}. 
  Below we show $g\in L^p_{\theta}(\R^d,\R^N)$. Then Theorem \ref{thm:APrioriEstimatesInLpRelativelyCompactPerturbation} applies to \eqref{equ:LQGleichung} 
  and yields $w_{\star}=D^{\alpha}v_{\star}\in W^{1,p}_{\theta}(\R^d,\R^N)$ and thus our assertion $v_{\star}-v_{\infty}\in W^{k,p}_{\theta}(\R^d,\R^N)$.
  % For the application of Theorem \ref{thm:APrioriEstimatesInLpRelativelyCompactPerturbation} it remains to 
  %  verify that applying $D^{\alpha}$ to \eqref{equ:NonlinearProblemRealFormulationMoreRegularity} takes the form \eqref{equ:LQGleichung} and that $g=g_1+g_2$ with $g_1$ 
  %  and $g_2$ from \eqref{equ:g1g2} belongs to $L^p_{\theta}(\R^d,\R^N)$.
  %  To verify \eqref{equ:LQGleichung}, we apply $D^{\alpha}$ to \eqref{equ:NonlinearProblemRealFormulationMoreRegularity} term by term, since $v_{\star}\in C^{k+1}(\R^d,\R^N)$. 
  % Obviously, it holds
  % \begin{align}
  %   \label{equ:DalphaLaplacian}
  %   D^{\alpha}\triangle v_{\star}(x) = \triangle D^{\alpha} v_{\star}(x),\,x\in\R^d.
  % \end{align}
  % 5.
  \item[5.] 
  To verify $g\in L^p_{\theta}(\R^d,\R^N)$, consider first $g_1$. By the first part of the Corollary (base case $k=3$) and by the induction hyperthesis 
  (induction step $k>3$) we have $v_{\star}-v_{\infty}\in W^{k-1,p}_{\theta}(\R^d,\R^N)$. The indices  $\gamma:=\alpha-e_j+e_i$ with $|\alpha|=k-1$ and 
  $e_j\leqslant\alpha$ satisfy $|\gamma|=k-1$, hence $D^{\alpha-e_j+e_i}v_{\star}\in L^p_{\theta}(\R^d,\R^N)$, and we deduce $g_1\in L^p_{\theta}(\R^d,\R^N)$.

  Next we consider $g_2$. We show $\theta g_2\in L^p(\R^d,\R^N)$ by using the generalized H\"older's inequality \eqref{eq:hoeldergeneral} with $p_j:=\frac{k-1}{|\pi_j|}p$ 
  and $u_j:=\theta\left|D^{|\pi_j|}v_{\star}h_{\pi_j}\right|$ for $j=1,\ldots,\ell$. Note that $\sum_{j=1}^{\ell}\frac{1}{p_j}=\frac{1}{p}$ since $\sum_{j=1}^{\ell}|\pi_j|=k-1$. 
  We obtain
  \begin{align*}
             & \left\|\theta(D^{\ell}f)(v_{\star})\left[D^{|\pi_1|}v_{\star}h_{\pi_1},\ldots,D^{|\pi_{\ell}|}v_{\star}h_{\pi_{\ell}}\right]\right\|_{L^p}
    \leqslant  \left\|(D^{\ell}f)(v_{\star})\right\|_{L^{\infty}} \Big\|\prod_{j=1}^{\ell}\theta^{\frac{1}{j}}\left|D^{|\pi_j|}v_{\star}h_{\pi_j}\right|\Big\|_{L^p} \\ 
    \leqslant& \left\|(D^{\ell}f)(v_{\star})\right\|_{L^{\infty}} \prod_{j=1}^{\ell}\left\|\theta^{\frac{1}{j}}D^{|\pi_j|}v_{\star}h_{\pi_j}\right\|_{L^{p_j}}
    \leqslant \left\|(D^{\ell}f)(v_{\star})\right\|_{L^{\infty}} \prod_{j=1}^{\ell}\left\|\theta D^{|\pi_j|}v_{\star}h_{\pi_j}\right\|_{L^{p_j}},
  \end{align*}
  where we used $\theta(x)\geqslant 1$ and  $\theta^{\frac{1}{j}}(x)\leqslant \theta(x)$. Note that $v_{\star}\in C_{\mathrm{b}}(\R^d,\R^N)$, 
  $f\in C^{k-1}(\R^N,\R^N)$ and $\ell\leqslant k-1$ imply the total derivative $(D^{\ell} f)(v_{\star})$ to be bounded on $\R^d$. It remains to 
  verify that $\theta D^{|\pi_j|}v_{\star}h_{\pi_j}\in L^{p_j}(\R^d,\R^N)$: From $v_{\star}-v_{\infty}\in W^{k-1,p}_{\theta}(\R^d,\R^N)$ and 
  \eqref{eq:sobolevweights} we infer $\theta D^{\gamma}(v_{\star}-v_{\infty})\in W^{k-1-|\gamma|,p}(\R^d,\R^N)$ for all $0\leqslant|\gamma|\leqslant k-1$, 
  % or equivalently
  %   \begin{align*}
  %     \theta D^{|\gamma|}(v_{\star}-v_{\infty})h_{\gamma}\in W^{k-1-|\gamma|,p}(\R^d,\R^N)\quad\forall\,0\leqslant|\gamma|\leqslant k-1.
  %   \end{align*}
  Using $1\leqslant |\pi_j|\leqslant k-2$ this proves the first assertion in
  \begin{align*}
    \theta D^{|\pi_j|}v_{\star}h_{\pi_j}\in W^{k-1-|\pi_j|,p}(\R^d,\R^N)\subset L^{p_j}(\R^d,\R^N).
  \end{align*}
  The second assertion '$\subset$' follows from the Sobolev embedding \eqref{eq:sobolevembed}, provided that $1<p<p_j<\infty$ and $\frac{d}{p}-(k-1-|\pi_j|)\leqslant\frac{d}{p_j}$. 
  The first inequality is implied by $1\leqslant |\pi_j|\leqslant k-2$,
  \begin{align*}
    1< p < \frac{k-1}{k-2}p   \leqslant  \frac{k-1}{|\pi_j|}p= p_j < \infty,
  \end{align*}
  while the second is implied by our assumptions $p\geqslant \frac{d}{2}$ and $k\geqslant 3$,
  \begin{align*}
    \frac{d}{p} - \frac{d}{p_j}
    = \frac{d}{p}\big(1-\frac{|\pi_j|}{k-1}\big) \leqslant 2 \big(1- \frac{|\pi_j|}{k-1}\big)
    \leqslant (k-1)\big(1-\frac{|\pi_j|}{k-1}\big) = k-1 - |\pi_j|.
  \end{align*}
\end{itemize}
\end{proof}

\begin{remark}[Higher regularity of the profile $v_{\star}$]\label{rem:HigherRegularity}
  Collecting the results of Theorem \ref{thm:NonlinearOrnsteinUhlenbeckSteadyState} and Corollary \ref{cor:NonlinearOrnsteinUhlenbeckSteadyStateMoreRegularity}, 
  we obtain
  \begin{align*}
    f\in C^{\max\{2,k-1\}}(\R^N,\R^N),\;v_{\star}\in C^{k+1}(\R^d,\R^N)\quad\Longrightarrow\quad v_{\star}-v_{\infty}\in W^{k,p}_{\theta}(\R^d,\R^N)
  \end{align*} 
  for any $k\in\N$, provided that $1<p<\infty$ from \eqref{cond:A4DC} satisfies $p\geqslant\frac{d}{2}$ if $k\geqslant 3$.
\end{remark}

%----------------------------------------------------------------------------------
% SUBSECTION 4.3: (Pointwise estimates for exponential decay).
%----------------------------------------------------------------------------------
\subsection{Pointwise exponential decay}
\label{subsec:4.3}
%----------------------------------------------------------------------------------

For the pointwise estimates we use the embedding in $L^{\infty}$. The particular choice $q=\infty$ and $l=0$ in \eqref{eq:sobolevweights} leads to
\begin{align}
  \label{equ:SobolevEmbedding}
  W^{k,p}(\R^d)\subseteq L^{\infty}(\R^d),\quad k\geqslant 0,\;1<p<\infty,\;d<kp
\end{align}
and to the inequality
\begin{align}
  \label{equ:SobolevInequality}
  \left\|u\right\|_{L^{\infty}(\R^d)}\leqslant C_{p,\infty,k,0}\left\|u\right\|_{W^{k,p}(\R^d)}\;\forall\,u\in W^{k,p}(\R^d).
\end{align}

\begin{corollary}[Pointwise exponentially decaying estimates]\label{cor:pointwise}
  Let the assumptions of Corollary \ref{cor:NonlinearOrnsteinUhlenbeckSteadyStateMoreRegularity} be satisfied. Moreover, let 
  $f\in C^{\max\{2,k-1\}}(\R^N,\R^N)$, $v_{\star}\in C^{k+1}(\R^d,\R^N)$ for some $k\in\N$ and let $p\geqslant\frac{d}{2}$ if $k\geqslant 3$.
  Then $v_{\star}-v_{\infty}\in W^{k,p}_{\theta}(\R^d,\R^N)$ satisfies the following estimate
  \begin{align}
    \label{equ:PointwiseEstimateVstar}
    \left|D^{\alpha}\left(v_{\star}(x)-v_{\infty}\right)\right| \leqslant C\exp\left(-\mu\sqrt{|x|^2+1}\right)
    \quad\forall\,x\in\R^d%\;\forall\,0\leqslant\mu\leqslant\varepsilon\frac{\sqrt{\azero\bzero}}{\amax p}
  \end{align}
  for every exponential decay rate $0\leqslant\mu\leqslant\varepsilon\frac{\sqrt{\azero\bzero}}{\amax p}$ and for every multi-index $\alpha\in\N_0^d$ 
  satisfying $d<(k-|\alpha|)p$.
\end{corollary}

\begin{proof}
  The proof follows directly from \eqref{eq:sobolevweights} and the Sobolev embedding \eqref{equ:SobolevEmbedding}.
\end{proof}

\begin{remark}
  In case of $d\in\{2,3\}$ and $p=2$ it is sufficient to choose $k=4$ to obtain pointwise estimates for $D^{\alpha}v_{\star}$ of order 
  $0\leqslant|\alpha|\leqslant 2$. This requires to assume $f\in C^{3}(\R^N,\R^N)$ and $v_{\star}\in C^{5}(\R^d,\R^N)$. Note that the 
  authors of \cite{BeynLorenz2008}  consider the case $d=p=2$ and assume $f\in C^4(\R^N,\R^N)$ for their stability analysis of rotating 
  patterns. Our results show that $f\in C^3(\R^N,\R^N)$ is sufficient to guarantee \cite[Assumption 1]{BeynLorenz2008} which, therefore, 
  can be omitted. 
\end{remark}

%----------------------------------------------------------------------------------
% SUBSECTION 4.4: (Application to complex-valued systems).
%----------------------------------------------------------------------------------
\subsection{Application to complex-valued systems}
\label{subsec:4.4}
%----------------------------------------------------------------------------------

The next corollary extends the results from Theorem \ref{thm:NonlinearOrnsteinUhlenbeckSteadyState}, Corollary \ref{cor:NonlinearOrnsteinUhlenbeckSteadyStateMoreRegularity} 
and Corollary \ref{cor:pointwise} to complex-valued systems of type \eqref{equ:complexversion} which appear in several applications. 

\begin{corollary}[Exponential decay of $v_{\star}$ for $\K=\C$]\label{cor:NonlinearOrnsteinUhlenbeckSteadyStateComplexVersion}
  Let the assumptions \eqref{cond:A4DC} and \eqref{cond:A5} be satisfied for $\K=\C$ and for some $1<p<\infty$. Moreover, let $g:\R\rightarrow\C^{N,N}$ satisfy the following 
  properties
  \begin{flalign}
    &g\in C^2(\R,\C^{N,N}),                                                                                           \tag{A7$_g$}\label{cond:A6g} &\\
    &\text{$A,g(0)\in\C^{N,N}$ are simultaneously diagonalizable (over $\C$),}                                        \tag{A9$_g$}\label{cond:A8g} &\\
    &\Re\left\langle w,-g(0)w\right\rangle\geqslant\beta_{-g(0)}\;\forall\,w\in\C^N,\,|w|=1\text{ for some $\beta_{\infty}:=\beta_{-g(0)}>0$} \tag{A11$_g$}\label{cond:A10g} &
  \end{flalign}
  and define
  \begin{align}
    \label{equ:FormOfNonlinearityInComplexCase}
    f:\C^N\rightarrow\C^N,\quad f(u)=g\left(|u|^2\right)u.
  \end{align} 
  Further, let $\amax=\rho(A)$ denote the spectral radius of $A$, $-\azero=s(-A)$ the spectral bound of $-A$, $-\bzero=s(g(0))$ the spectral bound of $g(0)$ and 
  $\theta(x)=\exp\left(\mu\sqrt{|x|^2+1}\right)$ a weight function with $\mu\in\R$. Then, for every $0<\varepsilon<1$ there is a constant $K_1=K_1(A,g,d,p,\varepsilon)>0$ 
  with the following property: Every classical solution $v_{\star}\in C^k(\R^d,\C^N)$ of
  \begin{align}
    \label{equ:NonlinearProblemComplexFormulation}
    A\triangle v(x)+\left\langle Sx,\nabla v(x)\right\rangle+f(v(x))=0,\,x\in\R^d,
  \end{align}
  such that $g\in C^{\max\{2,k-1\}}(\R,\C^{N,N})$ for some $k\in\N$, $p\geqslant\frac{d}{2}$ if $k\geqslant 3$, and
  \begin{align}
    \label{equ:BoundednessConditionForVStarComplex}
    \sup_{|x|\geqslant R_0}\left|v_{\star}(x)\right|\leqslant K_1\text{ for some $R_0>0$,}
  \end{align}
  satisfies
  \begin{align*}
    v_{\star}\in W^{k,p}_{\theta}(\R^d,\C^N)
  \end{align*}
  for every exponential decay rate
  \begin{align*}
    0\leqslant\mu\leqslant\varepsilon\frac{\sqrt{\azero\bzero}}{\amax p}.
  \end{align*} 
  Moreover, $v_{\star}$ satisfies the following pointwise estimate
  \begin{align*}
    \left|D^{\alpha}v_{\star}(x)\right| \leqslant C\exp\left(-\mu\sqrt{|x|^2+1}\right)
    \quad\forall\,x\in\R^d%\;\forall\,0\leqslant\mu\leqslant\varepsilon\frac{\sqrt{\azero\bzero}}{\amax p}
  \end{align*}
  for every exponential decay rate $0\leqslant\mu\leqslant\varepsilon\frac{\sqrt{\azero\bzero}}{\amax p}$ and for every multi-index $\alpha\in\N_0^d$ 
  satisfying $d<(k-|\alpha|)p$.
\end{corollary}

\begin{proof}
  % The proof is structured as follows: First we transform the $N$-dimensional complex-valued system \eqref{equ:NonlinearProblemComplexFormulation} into a 
  % coupled $2N$-dimensional system (step 1). Then, we check the assumptions of Theorem \ref{thm:NonlinearOrnsteinUhlenbeckSteadyState} and deduce from an application 
  % of Theorem \ref{thm:NonlinearOrnsteinUhlenbeckSteadyState} that $v_{\star}\in W^{1,p}_{\theta}(\R^d,\C^N)$. Then, an application of Corollary 
  % \ref{cor:NonlinearOrnsteinUhlenbeckSteadyStateMoreRegularity} directly yields $v_{\star}\in W^{k,p}_{\theta}(\R^d,\C^N)$ for $k\geqslant 2$ and an application of 
  % \ref{cor:pointwise} implies the pointwise estimates, which completes the proof.

  % \begin{itemize}[leftmargin=0.43cm]\setlength{\itemsep}{0.1cm}
  % 1.
  We transform the $N$-dimensional complex-valued system \eqref{equ:NonlinearProblemComplexFormulation} into the $2N$-dimensional real-valued system
  \begin{align}
    \mathbf{A}\triangle \mathbf{v}(x) + \left\langle Sx,\nabla \mathbf{v}(x)\right\rangle + \mathbf{f}(\mathbf{v}(x)) = 0,\,x\in\R^d,
    \label{equ:NonlinearProblemTransformedRealFormulation}
  \end{align}
  For this purpose, we decompose $A=A_1+iA_2$ with $A_1,A_2\in\R^{N,N}$, $v=v_1+iv_2$ with $v_1,v_2:\R^d\rightarrow\R^N$, $f_1,f_2:\R^{2N}\rightarrow\R^N$ 
  with $f_1(u_1,u_2)=\Re f(u_1+iu_2)$, $f_2(u_1,u_2)=\Im f(u_1+iu_2)$, $g=g_1+ig_2$ with $g_1,g_2:\R\rightarrow\R^{N,N}$. Moreover, we define 
  $\mathbf{A}\in\R^{2N,2N}$, $\mathbf{v}\in\R^{2N}$ and $\mathbf{f}:\R^{2N}\rightarrow\R^{2N}$ by
  \begin{align*}
    \mathbf{A}:=\begin{pmatrix}A_1 &-A_2\\ A_2 &A_1\end{pmatrix},\quad \mathbf{v}:=\begin{pmatrix}v_1\\v_2\end{pmatrix},\quad
    \mathbf{f}(\mathbf{v}):=\begin{pmatrix}f_1(\mathbf{v})\\f_2(\mathbf{v})\end{pmatrix}=\begin{pmatrix}g_1(|\mathbf{v}|^2) &-g_2(|\mathbf{v}|^2)\\g_2(|\mathbf{v}|^2)&g_1(|\mathbf{v}|^2)\end{pmatrix}\mathbf{v}.
  \end{align*} 
  % 2. 
  Let us apply Theorem \ref{thm:NonlinearOrnsteinUhlenbeckSteadyState} to the $2N$-dimensional problem \eqref{equ:NonlinearProblemTransformedRealFormulation} 
  and check its assumptions.  First, we collect the following relations of $A$ and $\mathbf{A}$:
  \begin{align}
    &\lambda\in\sigma(A) \quad\Longleftrightarrow\quad \lambda,\overline{\lambda}\in\sigma(\mathbf{A}),  
     \label{equ:EigenvalueRelation} \\
    &Y^{-1}AY=\Lambda_A\quad\Longleftrightarrow\quad\begin{pmatrix}iY & \overline{Y}\\ Y &-\overline{iY}\end{pmatrix}^{-1}\mathbf{A}\begin{pmatrix}iY & \overline{Y}\\ Y &-\overline{iY}\end{pmatrix} =\begin{pmatrix}\Lambda_A &0\\ 0 &\overline{\Lambda_A}\end{pmatrix}, 
     \label{equ:DiagonalizabilityRelation} \\
    &\Re\left\langle v,Av\right\rangle = \left\langle \mathbf{v},\mathbf{A}\mathbf{v}\right\rangle,\;\left|v\right|=\left|\mathbf{v}\right|,\;\left|Av\right|=\left|\mathbf{A}\mathbf{v}\right|.
    \label{equ:NormRelation}
  \end{align}
  Since $A$ satisfies \eqref{cond:A4DC} for some $1<p<\infty$ and $\K=\C$, we deduce from \eqref{equ:NormRelation}, that $\mathbf{A}$ satisfies \eqref{cond:A4DC} 
  for the same $1<p<\infty$ and $\K=\R$. In particular, we have $\gamma_A=\gamma_{\mathbf{A}}$ in \eqref{cond:A4DC}. Note that if $A$ satisfies \eqref{cond:A1}, 
  \eqref{cond:A2}, \eqref{cond:A3} for $\K=\C$ then $\mathbf{A}$ satisfies \eqref{cond:A1}, \eqref{cond:A2}, \eqref{cond:A3} for $\K=\R$, as follows from 
  \eqref{equ:DiagonalizabilityRelation}, \eqref{equ:EigenvalueRelation}, \eqref{equ:NormRelation}. Assumption \eqref{cond:A5} is not affected by the transformation. 
  From \eqref{cond:A6g} we deduce that $\mathbf{f}\in C^2(\R^{2N},\R^{2N})$, so that assumption \eqref{cond:A6} is satisfied for $\K=\R$. Obviously, $\mathbf{f}(v_{\infty})=0$ 
  holds for $v_{\infty}=0\in\R^{2N}$, so that condition \eqref{cond:A7} is satisfied. Since $A$ and $g(0)$ are simultaneously diagonalizable (over $\C$), cf. 
  \eqref{cond:A8g}, we deduce from \eqref{equ:DiagonalizabilityRelation} that $\mathbf{A}$ and
  \begin{align*}
    D\mathbf{f}(0) = \begin{pmatrix}g_1(0) &-g_2(0)\\g_2(0) & g_1(0)\end{pmatrix}
  \end{align*}
  are simultaneously diagonalizable (over $\C$). This proves assumption \eqref{cond:A8} for $\K=\R$. Finally, \eqref{cond:A10g} implies \eqref{cond:A10} with 
  $\beta_{\infty}=\beta_{-g(0)}$. Every classical solution $v_{\star}$ of \eqref{equ:NonlinearProblemComplexFormulation} satisfying $v_{\star}\in C_{\mathrm{b}}(\R^d,\C^N)$ 
  and \eqref{equ:BoundednessConditionForVStarComplex} leads to a classical solution
  \begin{align*}
    \mathbf{v}_{\star}:=\begin{pmatrix}\Re v_{\star}\\ \Im v_{\star}\end{pmatrix}
  \end{align*}
  of \eqref{equ:NonlinearProblemTransformedRealFormulation} satisfying $\mathbf{v}_{\star}\in C_{\mathrm{b}}(\R^d,\R^{2N})$ and \eqref{equ:BoundednessConditionForVStarComplex}. 

  Summarizing, Theorem \ref{thm:NonlinearOrnsteinUhlenbeckSteadyState} yields $\mathbf{v}_{\star}\in W^{1,p}_{\theta}(\R^d,\R^{2N})$, and thus $v_{\star}\in W^{1,p}_{\theta}(\R^d,\C^N)$.
\end{proof}

In Section \ref{sec:6} we will apply this result to the cubic-quintic complex Ginzburg-Landau equation which is of the form \eqref{equ:complexversion}. 
Other examples fitting into this class include the Schr\"odinger and the Gross-Pitaevskii equation.

%---------------------------------------------------------------------------------------------------------------------------------------------------
%
%  SECTION 5: (Exponential decay of eigenfunctions)
%
%---------------------------------------------------------------------------------------------------------------------------------------------------
\sect{Exponential decay of eigenfunctions}
\label{sec:5}
%---------------------------------------------------------------------------------------------------------------------------------------------------

Consider the eigenvalue problem
\begin{align}
  \label{equ:EigenvalueProb}
  A\triangle v(x) + \left\langle Sx,\nabla v(x)\right\rangle + Df(v_{\star}(x))v(x) = \lambda v(x),\,x\in\R^d,\,d\geqslant 2.
\end{align}

We are interested in \begriff{classical solutions of \eqref{equ:EigenvalueProb}}, i.e. $\lambda\in\C$ and $v\in C^2(\R^d,\C^N)$ solves \eqref{equ:EigenvalueProb} 
pointwise (cf. Definition \ref{def:ClassicalSolution}).

%----------------------------------------------------------------------------------
% SUBSECTION 5.1: (Exponential decay of eigenfunctions).
%----------------------------------------------------------------------------------
\subsection{Sobolev and pointwise estimates of eigenfunctions}
\label{subsec:5.1}
%----------------------------------------------------------------------------------

The following theorem states that every classical solution $v$ of the eigenvalue problem \eqref{equ:EigenvalueProb} decays exponentially in space, provided its 
associated (isolated) eigenvalue $\lambda$ satisfies $\Re\lambda>-\beta_{\infty}$. The proof is similar to those of Theorem \ref{thm:NonlinearOrnsteinUhlenbeckSteadyState} 
and Corollary \ref{cor:NonlinearOrnsteinUhlenbeckSteadyStateMoreRegularity}, but now it is crucial that Theorem \ref{thm:APrioriEstimatesInLpRelativelyCompactPerturbation} 
can be employed for cases where $\lambda\neq 0$, $g\neq 0$ and $\K=\C$.

\begin{theorem}[Exponential decay of eigenfunctions]\label{thm:LinearizedOrnsteinUhlenbeckExponentialDecayInLp}
  \enum{1} Let the assumptions \eqref{cond:A4DC}, \eqref{cond:A5}--\eqref{cond:A8} and \eqref{cond:A10} be satisfied for $\K=\R$ and for some $1<p<\infty$. 
  Moreover, let $\amax=\rho(A)$ denote the spectral radius of $A$, $-\azero=s(-A)$ the spectral bound of $-A$, $-\bzero=s(Df(v_{\infty}))$ the spectral bound of $Df(v_{\infty})$ 
  and let $\beta_{\infty}$ be from \eqref{cond:A10}. Further, let 
  \begin{align*}
    \theta_j(x)=\exp\left(\mu_j\sqrt{|x|^2+1}\right),\quad x\in\R^d,\quad j=1,2 
  \end{align*}
  denote a weight function for $\mu_1,\mu_2\in\R$. Then, for every $0<\varepsilon<1$ there is a constant $K_1=K_1(A,f,v_{\infty},d,p,\varepsilon)>0$ such that 
  for every classical solution $v_{\star}$ of \eqref{equ:NonlinearProblemRealFormulation} satisfying \eqref{equ:BoundednessConditionForVStar} the following property holds: 
  Every classical solution $v$ of the eigenvalue problem 
  \begin{align}
    \label{equ:NonlinearEigenvalueProblemRealFormulation}
    A\triangle v(x)+\left\langle Sx,\nabla v(x)\right\rangle+Df(v_{\star}(x))v(x)=\lambda v(x),\,x\in\R^d,
  \end{align}
  with $\lambda\in\C$ and $\Re\lambda\geqslant -(1-\varepsilon)\beta_{\infty}$, such that $v\in L^p_{\theta_1}(\R^d,\C^N)$ for some exponential growth rate
  \begin{align}
    \label{equ:ratebound3}
    -\sqrt{\varepsilon\frac{\gamma_A\beta_{\infty}}{2d|A|^2}}\leqslant\mu_1\leqslant 0,
  \end{align}
  satisfies
  \begin{align*}
    v\in W^{1,p}_{\theta_2}(\R^d,\C^N)
  \end{align*}
  for every exponential decay rate
  \begin{align} 
    \label{equ:ratebound2}
    0\leqslant\mu_2\leqslant\varepsilon\frac{\sqrt{\azero\bzero}}{\amax p}.
  \end{align} 
  \enum{2} If additionally, $p\geqslant\frac{d}{2}$, $f\in C^{k}(\R^N,\R^N)$, $v_{\star}\in C^{k+1}(\R^d,\R^N)$ and $v\in C^{k+1}(\R^d,\C^N)$ for some $k\in\N$ with 
  $k\geqslant 2$, then $v\in W^{k,p}_{\theta}(\R^d,\C^N)$ holds. Moreover, $v$ satisfies the pointwise estimate
    \begin{align} \label{equ:ratederiv}
    \left|D^{\alpha}v(x)\right| \leqslant C\exp\left(-\mu_2\sqrt{|x|^2+1}\right),
    \quad x\in\R^d
  \end{align}
  for every exponential decay rate $0\leqslant\mu_2\leqslant\varepsilon\frac{\sqrt{\azero\bzero}}{\amax p}$ and for every multi-index $\alpha\in\N_0^d$ satisfying 
  $d<(k-|\alpha|)p$.
\end{theorem}

\begin{remark}
  In case of $d\in\{2,3\}$ and $p=2$ it is sufficient to choose $k=5$ to obtain pointwise estimates for $v$ of order $0\leqslant|\alpha|\leqslant 2$. 
  This requires to assume $f\in C^{5}(\R^N,\R^N)$, $v_{\star}\in C^{6}(\R^d,\R^N)$ and $v\in C^{6}(\R^d,\C^N)$. Moreover, note that the space 
  $L^p_{\theta_1}(\R^d,\C^N)$ does not only allow bounded but even exponentially growing eigenfunctions.
\end{remark}

\begin{proof}
  \enum{1} Let $v_{\star}$ be a classical solution of \eqref{equ:NonlinearProblemRealFormulation} satisfying \eqref{equ:BoundednessConditionForVStar} 
  and let $v$ be a classical solution of \eqref{equ:NonlinearEigenvalueProblemRealFormulation} satisfying $v\in L^p_{\theta_1}(\R^d,\C^N)$ with $\theta_1$ 
  from \eqref{equ:WeightFunctionsLQ} and $\mu_1$ such that \eqref{equ:ratebound3}. Then $v$ satisfies
  \begin{align*}
    0 = \lambda v(x) - \left(A\triangle v(x) + \left\langle Sx,\nabla v(x)\right\rangle - B_{\infty}v(x) + Q(x)v(x)\right)
      = [\left(\lambda I-\L_{Q}\right)v](x),\,x\in\R^d
  \end{align*}
  with $B_{\infty} := -Df(v_{\infty})$ and $Q(x) := Df(v_{\star}(x))-Df(v_{\infty})$ as in \eqref{equ:MatrixBandQMoreRegularity}. Now, $v\in W^{1,p}_{\theta_2}(\R^d,\C^N)$ 
  follows from Theorem \ref{thm:APrioriEstimatesInLpRelativelyCompactPerturbation} with $\mu_1<0$, $g=0$ and $\lambda\in\C$ with 
  $\Re\lambda\geqslant -(1-\varepsilon)\beta_{\infty}$ . Note, that the assumptions of Theorem \ref{thm:APrioriEstimatesInLpRelativelyCompactPerturbation} 
  are satisfied as shown in the proof of Corollary \ref{cor:NonlinearOrnsteinUhlenbeckSteadyStateMoreRegularity}. 

  \enum{2} Assuming more smoothness for $f,v,v_{\star}$, we prove $v\in W^{k,p}_{\theta_2}(\R^d,\C^N)$  by induction on $k$. Let $\alpha\in\N_0^d$ be a 
  multi-index of length $|\alpha|=k-1$. Similar to the proof of Corollary \ref{cor:NonlinearOrnsteinUhlenbeckSteadyStateMoreRegularity}, an application 
  of $D^{\alpha}$ to \eqref{equ:NonlinearEigenvalueProblemRealFormulation} yields an inhomogenous equation for $w:=D^{\alpha}v$,  
  \begin{align}
    \label{equ:LQGleichungEigenfunctions}
    [(\lambda I- \L_Q) w](x)=g(x),\,x\in\R^d
  \end{align}
  with $B_{\infty}$ and $Q(x)$ as in \eqref{equ:MatrixBandQMoreRegularity}, $\lambda\in\C$ with $\Re\lambda\geqslant -(1-\varepsilon)\beta_{\infty}$, and 
  $g(x):=g_1(x)+g_2(x)$ where 
  \begin{align}
    \label{equ:g1g2Eigenfunctions}
    \begin{split}
    g_1 :=& \sum_{i=1}^{d} \sum_{\substack{j=1\\e_j\leqslant\alpha}}^{d} S_{ij} \binom{\alpha}{e_j} D^{\alpha-e_j+e_i}v, \\ 
    g_2 :=& \sum_{\substack{\beta\leqslant\alpha\\|\beta|\geqslant 1}} \binom{\alpha}{\beta}
\sum_{\ell=1}^{|\beta|}\sum_{\pi\in\mathcal{P}_{\ell,|\beta|}}(D^{\ell+1}f)(v_{\star})\left[D^{|\pi_1|}v_{\star}h_{\pi_1},\ldots,D^{|\pi_{\ell}|}v_{\star}h_{\pi_{\ell}},D^{\alpha-\beta}v\right].
    \end{split}
  \end{align}
  In this expression, the multilinear argument $h$ is defined as in \eqref{equ:multilineararg} with $\alpha$ replaced by $\beta$. Below we prove 
  $g_1,g_2\in L^p_{\theta_2}(\R^d,\C^N)$, so that Theorem \ref{thm:APrioriEstimatesInLpRelativelyCompactPerturbation} implies $w=D^{\alpha}v\in W^{1,p}_{\theta_2}(\R^d,\C^N)$ 
  and therefore, $v\in W^{k,p}_{\theta_2}(\R^d,\C^N)$.\\
  First we consider $g_1$: By the first part of this Corollary (base case $k=2$) and by the induction hyperthesis (induction step $k>2$) we have 
  $v\in W^{k-1,p}_{\theta_2}(\R^d,\C^N)$. As  in the proof of Corollary \ref{cor:NonlinearOrnsteinUhlenbeckSteadyStateMoreRegularity}, we then deduce 
  $g_1\in L^p_{\theta_2}(\R^d,\R^N)$. 
  Finally, we show $g_2\in L^p_{\theta_2}(\R^d,\R^N)$ by applying H\"older's inequality 
  \eqref{eq:hoeldergeneral} with $p_j:=\frac{k}{|\pi_j|}p$ and $u_j:=\theta_2^{\frac{1}{j}}\left|D^{|\pi_j|}v_{\star}h_{\pi_j}\right|$ for $j=1,\ldots,\ell$, 
  $p_{\ell+1}:=\frac{k}{k-|\beta|}p$ and $u_{\ell+1}:=\left|D^{\alpha-\beta}v\right|$. Note that 
$\sum_{j=1}^{\ell+1}\frac{1}{p_j}=\frac{1}{p}$ follows from $\sum_{j=1}^{\ell}|\pi_j|=|\beta|$.
  We obtain
  \begin{align*}
              & \left\|\theta_2(D^{\ell+1}f)(v_{\star})\left[D^{|\pi_1|}v_{\star}h_{\pi_1},\ldots,D^{|\pi_{\ell}|}v_{\star}h_{\pi_{\ell}},D^{\alpha-\beta}v\right]\right\|_{L^p} \\
    \leqslant & \left\|(D^{\ell+1}f)(v_{\star})\right\|_{L^{\infty}} \Big\|\left|D^{\alpha-\beta}v\right|\prod_{j=1}^{\ell}\theta_2^{\frac{1}{j}}\left|D^{|\pi_j|}v_{\star}h_{\pi_j}\right|\Big\|_{L^p} \\
    \leqslant & \left\|(D^{\ell+1}f)(v_{\star})\right\|_{L^{\infty}} \left\|D^{\alpha-\beta}v\right\|_{L^{p_{\ell+1}}} \prod_{j=1}^{\ell}\left\|\theta_2^{\frac{1}{j}}D^{|\pi_j|}v_{\star}h_{\pi_j}\right\|_{L^{p_j}} \\
    \leqslant & \left\|(D^{\ell+1}f)(v_{\star})\right\|_{L^{\infty}} \left\|D^{\alpha-\beta}v\right\|_{L^{p_{\ell+1}}} \prod_{j=1}^{\ell}\left\|\theta_2 D^{|\pi_j|}v_{\star}h_{\pi_j}\right\|_{L^{p_j}}
  \end{align*}
  since $\theta_2(x)\geqslant 1$ and $j\geqslant 1$ imply $\theta_2^{\frac{1}{j}}(x)\leqslant \theta_2(x)$. Note that $v_{\star}\in C_{\mathrm{b}}(\R^d,\R^N)$, 
  $f\in C^{k}(\R^N,\R^N)$ and $\ell\leqslant|\beta|\leqslant|\alpha|=k-1$ imply the boundedness of 
$(D^{\ell+1} f)(v_{\star})$. 
  As in the proof of Corollary \ref{cor:NonlinearOrnsteinUhlenbeckSteadyStateMoreRegularity}, both $\theta_2 D^{|\pi_j|}v_{\star}h_{\pi_j}\in L^{p_j}(\R^d,\R^N)$ and 
  $\theta_2D^{\alpha-\beta}v\in L^{p_{\ell+1}}(\R^d,\C^N)$ follow from the Sobolev embedding \eqref{eq:sobolevembed}: Since $f\in  C^{k-1}(\R^N,\R^N)$ 
  and $v_{\star}\in C^{k+1}(\R^d,\R^N)$, Corollary \ref{cor:NonlinearOrnsteinUhlenbeckSteadyStateMoreRegularity} implies $v_{\star}-v_{\infty}\in W^{k,p}_{\theta_2}(\R^d,\R^N)$, 
  therefore $\theta_2 D^{|\pi_j|}v_{\star}h_{\pi_j}\in  W^{k-|\pi_j|,p}(\R^d,\R^N)$. The Sobolev embedding \eqref{eq:sobolevembed} shows 
  \begin{align*}
    \theta_2 D^{|\pi_j|}v_{\star}h_{\pi_j}\in W^{k-|\pi_j|,p}(\R^d,\R^N) \subseteq L^{p_j}(\R^d,\R^N)\;\forall\,1\leqslant|\pi_j|\leqslant|\beta|,
  \end{align*}
  provided that $1<p<p_j<\infty$ and $\frac{d}{p}-(k-|\pi_j|)\leqslant\frac{d}{p_j}$. These conditions are obviously satisfied, 
  since $p\geqslant\frac{d}{2}$, $1\leqslant|\pi_j|\leqslant k-1$.
  Next, since $v\in W^{k-1,p}_{\theta_2}(\R^d,\C^N)$, \eqref{eq:sobolevweights} implies $\theta_2D^{\gamma}v\in W^{k-1-|\gamma|,p}(\R^d,\C^N)$ 
  for all $0\leqslant|\gamma|\leqslant k-1$. For $\gamma=\alpha-\beta$ with $|\alpha|=k-1$ and $\beta\leqslant\alpha$ we have $k-1-|\gamma|=|\beta|$ 
  and therefore, the Sobolev embedding \eqref{eq:sobolevembed} implies
  \begin{align*}
   \theta_2 D^{\alpha-\beta}v\in W^{|\beta|,p}(\R^d,\R^N)\subset L^{p_{\ell+1}}(\R^d,\R^N),
  \end{align*}
  provided that $1<p<p_{\ell+1} <\infty$ and $\frac{d}{p}-|\beta|\leqslant\frac{d}{p_{\ell+1}}$.
  These conditions are satisfied by the same arguments as above. This concludes the proof of $g_2\in L^p_{\theta_2}(\R^d,\R^N)$.

  Finally, the pointwise estimates  follow when combining our previous Sobolev estimates with \eqref{eq:sobolevweights} and the embedding \eqref{equ:SobolevEmbedding}, 
  in a similar manner  as in Corollary \ref{cor:pointwise} 
\end{proof}

\begin{remark}[Higher regularity of the eigenfunction $v$]\label{rem:HigherRegularityEigenfunctions}
  Collecting the results of Theorem \ref{thm:LinearizedOrnsteinUhlenbeckExponentialDecayInLp} and Remark \ref{rem:HigherRegularity}, we obtain
  \begin{align*}
    f\in C^{\max\{2,k\}},\;v_{\star},v\in C^{k+1} \;\Longrightarrow\; v_{\star}-v_{\infty},v\in W^{k,p}_{\theta_2}
  \end{align*} 
  for any $k\in\N$, provided that $1<p<\infty$ from \eqref{cond:A4DC} satisfies $p\geqslant\frac{d}{2}$ if $k\geqslant 2$. We also recall the role 
  of the parameter $\varepsilon$ in the exponential estimates. Theorem \ref{thm:LinearizedOrnsteinUhlenbeckExponentialDecayInLp} shows that every 
  eigenfunction associated to an eigenvalue $\lambda$ with $\Re\lambda>-\beta_{\infty}$ decays exponentially in space. Usually, one expects this 
  behavior even for $\Re\lambda>-\bzero$. The rate of decay is controlled by $\varepsilon \in (0,1)$. If $\Re\lambda\geqslant -(1+\varepsilon)\beta_{\infty}$ 
  is close to $-\beta_{\infty}$ we may take $\varepsilon$ close to $0$ and obtain a small rate $\mu_2$ of decay according to \eqref{equ:ratebound2}. 
  On the other hand, if $\lambda$ is close to the imaginary axis we may take $\varepsilon$ close to $1$ and obtain a higher rate of decay. 
\end{remark}

%----------------------------------------------------------------------------------
% SUBSECTION 5.2: (Exponential decay of rotational term).
%----------------------------------------------------------------------------------
\subsection{Eigenfunctions belonging to eigenvalues on the imaginary axis}
\label{subsec:5.2}
%----------------------------------------------------------------------------------

Some classical solutions of the eigenvalue problem \eqref{equ:EigenvalueProb} are due to equivariance of the underlying equations and can be expressed in terms of the rotating wave itself. The following result  proved in \cite[Theorem 9.4]{Otten2014}, specifies these eigenfunctions.

\begin{theorem}[Point spectrum on the imaginary axis]\label{thm:EigenfunctionsOfTheLinearizedOrnsteinUhlenbeckInLp}
  Let $S\in\R^{d,d}$ be skew-symmetric and let $U\in\C^{d,d}$ denote the unitary matrix satisfying $\Lambda_S=\bar{U}^TSU$ with $\Lambda_S=\diag(\lambda_1^S,\ldots,\lambda_d^S)$ 
  and $\lambda_1^S,\ldots,\lambda_d^S\in\sigma(S)$. Moreover, let $v_{\star}\in C^3(\R^d,\R^N)$ be a classical solution of \eqref{equ:NonlinearProblemRealFormulation}, 
  then the function $v:\R^d\rightarrow\C^N$ given by
  \begin{align}
    \label{equ:EigenfunctionsOniR}
    v(x) = \left\langle C^{\mathrm{rot}}x+C^{\mathrm{tra}},\nabla v_{\star}(x)\right\rangle
         = \sum_{i=1}^{d-1}\sum_{j=i+1}^{d}C_{ij}^{\mathrm{rot}}(x_j D_i-x_i D_j)v_{\star}(x) + \sum_{l=1}^{d}C_l^{\mathrm{tra}}D_l v_{\star}(x)
  \end{align}
  is a classical solution of the eigenvalue problem \eqref{equ:EigenvalueProb} if $C^{\mathrm{rot}}\in\C^{d,d}$ and $C^{\mathrm{tra}}\in\C^d$ either satisfy
  \begin{align*}
    \lambda=-\lambda_l^S,\quad C^{\mathrm{rot}}=0,\quad C^{\mathrm{tra}}=Ue_l
  \end{align*}
  for some $l=1,\ldots,d$, or
  \begin{align*}
    \lambda=-(\lambda_n^S+\lambda_m^S),\quad C^{\mathrm{rot}}=U(I_{nm}-I_{mn})U^T,\quad C^{\mathrm{tra}}=0
  \end{align*}
  for some $n=1,\ldots,d-1$ and $m=n+1,\ldots,d$. Here $I_{nm}\in\R^{d,d}$ denotes the matrix having the entries $1$ at the $n$-th row and $m$-th column and $0$ otherwise. 
  All the eigenvalues above lie on the imaginary axis.
\end{theorem}

A direct consequence of Theorem \ref{thm:LinearizedOrnsteinUhlenbeckExponentialDecayInLp} and Theorem \ref{thm:EigenfunctionsOfTheLinearizedOrnsteinUhlenbeckInLp} is 
that the eigenfunctions $v$ from \eqref{equ:EigenfunctionsOniR} belong to $W^{1,p}_{\theta_2}(\R^d,\C^N)$ and decay exponentially in space, \cite[Theorem 9.8]{Otten2014}.

\begin{corollary}[Exponential decay of eigenfunctions for eigenvalues on $i\R$]\label{cor:ExponentialDecayOfEigenfunctions}
  Let all assumptions of the statements \enum{1} and \enum{2} of Theorem \ref{thm:LinearizedOrnsteinUhlenbeckExponentialDecayInLp} be satisfied. 
  % Let the assumptions \eqref{cond:A4}--\eqref{cond:A8} and \eqref{cond:A10} be satisfied for some $1<p<\infty$ and $\K=\R$. Moreover, let $\amax=\rho(A)$ denote the 
  % spectral radius of $A$, $-\azero=s(-A)$ the spectral bound of $-A$ and $-\bzero=s(Df(v_{\infty}))$ the spectral bound of $Df(v_{\infty})$. 
  % Then, for every $0<\varepsilon<1$ there is a constant $K_1=K_1(A,f,v_{\infty},d,p,\varepsilon)>0$ with the following property: \\
  % Given a classical solution $v_{\star}\in C^{k+2}(\R^d,\R^N)$ of \eqref{equ:NonlinearProblemRealFormulation} with $f\in C^{\max\{2,k\}}(\R^N,\R^N)$ for some $k\in\N$ 
  % such that $v_{\star}\in C_{\mathrm{b}}(\R^d,\R^N)$ and \eqref{equ:BoundednessConditionForVStar} hold and let $p\geqslant\frac{d}{2}$ if $k\geqslant 2$. 
  Then the classical solution 
  \begin{align*}
    v(x) = \left\langle C^{\mathrm{rot}}x+C^{\mathrm{tra}},\nabla v_{\star}(x)\right\rangle, \quad x\in \R^d
  %       = \sum_{i=1}^{d-1}\sum_{j=i+1}^{d}C_{ij}^{\mathrm{rot}}(x_j D_i-x_i D_j)v_{\star}(x) + \sum_{l=1}^{d}C_l^{\mathrm{tra}}D_l v_{\star}(x)
  \end{align*}
  of the eigenvalue problem \eqref{equ:NonlinearEigenvalueProblemRealFormulation} with $\lambda$, $C^{\mathrm{rot}}$ and $C^{\mathrm{tra}}$ from 
  Theorem \ref{thm:EigenfunctionsOfTheLinearizedOrnsteinUhlenbeckInLp} lies in $W^{k,p}_{\theta_2}(\R^d,\C^N)$ for every exponential decay rate \eqref{equ:ratebound2}. 
  % $\theta(x)=\exp\left(\mu\sqrt{|x|^2+1}\right)$ with exponential rate
  % \begin{align*}
  %   0\leqslant\mu\leqslant\varepsilon\frac{\sqrt{\azero\bzero}}{\amax p}.
  % \end{align*}
  Moreover, the function $v\in W^{k,p}_{\theta_2}(\R^d,\C^N)$ satisfies the pointwise estimate \eqref{equ:ratederiv}.
  %  \begin{align*}
  %     \left|D^{\alpha}v(x)\right| \leqslant C\exp\left(-\mu\sqrt{|x|^2+1}\right)
  %     \quad\forall\,x\in\R^d\;\forall\,0\leqslant\mu\leqslant\varepsilon\frac{\sqrt{\azero\bzero}}{\amax p}
  %   \end{align*}
  %   and for every multi-index $\alpha\in\N_0^d$ satisfying $d<(k-|\alpha|)p$.
  % 
\end{corollary}

\begin{proof}
  In order to apply Theorem \ref{thm:LinearizedOrnsteinUhlenbeckExponentialDecayInLp} to $v(x)=\left\langle C^{\mathrm{rot}}x+C^{\mathrm{tra}},\nabla v_{\star}(x)\right\rangle$, 
  we observe that the map $x \mapsto \left\langle C^{\mathrm{rot}}x+C^{\mathrm{tra}},\nabla v_{\star}(x)\right\rangle$ is of class $ C^{k+1}$ since $v_{\star}\in C^{k+2}(\R^d,\R^N)$. 
  In particular, $\beta_{\infty}>0$ allows to deal with eigenvalues $\lambda\in i\R$. 
\end{proof}

\begin{remark}\label{rem:PointEssential}
  Later on, we numerically approximate the spectrum of the linearization $\L$ for \eqref{equ:FarField2}. For this purpose we decompose the $L^p$-spectrum $\sigma(\L)$ of $\L$ into the disjoint 
  union of point spectrum $\sigma_{\mathrm{point}}(\L)$ and essential spectrum $\sigma_{\mathrm{ess}}(\L)$ 
  \begin{align}
    \label{equ:SpectrumDecomposition}
    \sigma(\L) = \sigma_{\mathrm{point}}(\L) \cup \sigma_{\mathrm{ess}}(\L)
  \end{align}
  The point spectrum of $\L$ is affected by the symmetries of the group action and contains the eigenvalues described in Theorem \ref{thm:EigenfunctionsOfTheLinearizedOrnsteinUhlenbeckInLp}. 
  In particular, it contains the spectrum of $S$ and the sum of its different eigenvalues, i.e.
  \begin{align}
    \label{equ:PointSpectrum}
    \sigma_{\mathrm{point}}^{\mathrm{part}}(\L):=\sigma(S)\cup\{\lambda_1+\lambda_2\mid \lambda_1,\lambda_2\in\sigma(S),\,\lambda_1\neq\lambda_2\}
    \subseteq\sigma_{\mathrm{point}}(\L)
  \end{align}
  The associated eigenfunctions are explictly known, see \eqref{equ:EigenfunctionsOniR}, and they are exponentially localized as shown in Corollary 
  \ref{cor:ExponentialDecayOfEigenfunctions}. In general, $\sigma_{\mathrm{point}}(\L)$ contains further eigenvalues. Neither these additional eigenvalues 
  nor their associated eigenfunctions can usually be determined explicitly. The essential spectrum of $\L$ depends on the asymptotic behavior of the wave at 
  infinity. Under the same assumptions as in Theorem \ref{thm:NonlinearOrnsteinUhlenbeckSteadyState} one derives a dispersion relation for rotating waves, 
  see \cite[Section 9.5]{Otten2014},
  \begin{align}
    \label{equ:DispersionRelation}
    \det\left(\lambda I_N+\omega^2 A - Df(v_{\infty}) + i\sum_{l=1}^{m}n_l\sigma_l I_N\right)=0\text{ for some $\omega\in\R$, $n_l\in\Z$}
  \end{align}
  which then yields information about the essential spectrum
  \begin{align}
    \label{equ:EssentialSpectrum}
    \sigma_{\mathrm{ess}}^{\mathrm{part}}(\L):=\{\lambda\in\C\mid\text{$\lambda$ satisfies \eqref{equ:DispersionRelation}}\}\subseteq\sigma_{\mathrm{ess}}(\L).
  \end{align}
  Here, $S$ has the $(d-2m)$--fold eigenvalue $0$ and nontrivial eigenvalues
$\pm i \sigma_l,l=1,\ldots,m$ on the imaginary axis.
For more details we refer to the examples in Section \ref{sec:6} and to \cite[Chapter 9,10]{Otten2014}.  
\end{remark}

%---------------------------------------------------------------------------------------------------------------------------------------------------
%
%  SECTION 6: (Rotating waves in reaction diffusion systems: The cubic-quintic complex Ginzburg-Landau equation)
%
%---------------------------------------------------------------------------------------------------------------------------------------------------
\sect{Rotating waves in reaction diffusion systems: \\The cubic-quintic complex Ginzburg-Landau equation}
\label{sec:6}
%---------------------------------------------------------------------------------------------------------------------------------------------------

% a) Introduction of the complex and real QCGL
Consider the \begriff{cubic-quintic complex Ginzburg-Landau equation (QCGL)}, \cite{GinzburgLandau1950},
\begin{align}
  \label{equ:ComplexQuinticCubicGinzburgLandauEquation}
  u_t = \alpha\triangle u + u\left(\delta + \beta\left|u\right|^2 + \gamma\left|u\right|^4\right)
\end{align}
where $u:\R^d\times[0,\infty)\rightarrow\C$, $d\in\{2,3\}$, $\alpha,\beta,\gamma,\delta\in\C$ with $\Re\alpha>0$ and $f:\C\rightarrow\C$ given by
\begin{align}
  \label{equ:NonlinearityGLComplexVersion}
  f(u) := u\left(\delta + \beta\left|u\right|^2 + \gamma\left|u\right|^4\right).
\end{align}
The real-valued version of \eqref{equ:ComplexQuinticCubicGinzburgLandauEquation} reads as follows
\begin{align}
  \label{equ:ComplexQuinticCubicGinzburgLandauEquationRealVersion}
  \mathbf{u}_t = \mathbf{A}\triangle\mathbf{u} + \mathbf{f}(\mathbf{u})\quad\text{with}\quad 
  \textbf{A}:= \begin{pmatrix} \alpha_1 & -\alpha_2 \\ \alpha_2 & \alpha_1\end{pmatrix},\quad
  \mathbf{u}=\begin{pmatrix} u_1 \\ u_2 \end{pmatrix}
\end{align}
and $\mathbf{f}:\R^2\rightarrow\R^2$ given by
\begin{align}
  \label{equ:NonlinearityGLRealVersion}
  \mathbf{f}\begin{pmatrix} u_1 \\ u_2 \end{pmatrix} := 
  \begin{pmatrix} \left(u_1\delta_1-u_2\delta_2\right) + \left(u_1\beta_1-u_2\beta_2\right)\left(u_1^2+u_2^2\right) + \left(u_1\gamma_1-u_2\gamma_2\right)\left(u_1^2+u_2^2\right)^2 \\
                  \left(u_1\delta_2+u_2\delta_1\right) + \left(u_1\beta_2+u_2\beta_1\right)\left(u_1^2+u_2^2\right) + \left(u_1\gamma_2+u_2\gamma_1\right)\left(u_1^2+u_2^2\right)^2\end{pmatrix},
\end{align}
where $u=u_1+iu_2$, $\alpha=\alpha_1+i\alpha_2$, $\beta=\beta_1+i\beta_2$, $\gamma=\gamma_1+i\gamma_2$, $\delta=\delta_1+i\delta_2$.

% b) Historical background of QCGL  
This equation describes different aspects of signal propagation in heart tissue, superconductivity, superfluidity, nonlinear optical systems, see \cite{Moores1993}, 
photonics, plasmas, physics of lasers, Bose-Einstein condensation, liquid crystals, fluid dynamics, chemical waves, quantum field theory, granular media and is used 
in the study of hydrodynamic instabilities, see \cite{Mielke2002}. It shows a variety of coherent structures like stable and unstable pulses, fronts, sources and sinks 
in 1D, see \cite{AfanasjevAkhmedievSoto-Crespo1996,ThualFauve1988,TrilloTorruellas2010,VanSaarloosHohenberg1992}, vortex solitons, see \cite{CrasovanMalomedMihalache2000}, 
spinning solitons, see \cite{CrasovanMalomedMihalache2001}, dissipative ring solitons, see \cite{Soto-Crespo2009}, rotating spiral waves, propagating clusters, see \cite{RosanovFedorovShatsev2006}, and exploding dissipative 
solitons, see \cite{AkhmedievAnkiewiczSoto-Crespo2000} in 2D as well as scroll waves and spinning solitons in 3D, see \cite{MihalacheMaziluCrasovanMalomedLederer2000}.

% c) Discussion of assumptions
We are interested in exponentially localized rotating wave solutions $u_{\star}:\R^d\times[0,\infty)\rightarrow\C$ of \eqref{equ:ComplexQuinticCubicGinzburgLandauEquation} and 
$\mathbf{u_{\star}}:\R^d\times[0,\infty)\rightarrow\R^2$ of \eqref{equ:ComplexQuinticCubicGinzburgLandauEquationRealVersion}. Note that, given some skew-symmetric 
$S\in\R^{d,d}$ and some vector $x_{\star}\in\R^d$, $u_{\star}(x,t)=v_{\star}(e^{-tS}(x-x_{\star}))$ with $v_{\star}:\R^d\rightarrow\C$ is a rotating wave of 
\eqref{equ:ComplexQuinticCubicGinzburgLandauEquation} if and only if $\mathbf{u_{\star}}(x,t)=\mathbf{v_{\star}}(e^{-tS}(x-x_{\star}))$  
is a rotating wave of \eqref{equ:ComplexQuinticCubicGinzburgLandauEquationRealVersion}, where $\mathbf{u_{\star}}=\begin{pmatrix}\Re u_{\star}\\\Im u_{\star}\end{pmatrix}$ 
and $\mathbf{v_{\star}}=\begin{pmatrix}\Re v_{\star}\\\Im v_{\star}\end{pmatrix}$.
 We are going to show that $v_{\star}$ (and $\mathbf{v_{\star}}$) are exponentially localized by applying Theorem \ref{thm:NonlinearOrnsteinUhlenbeckSteadyState} and Corollaries
\ref{cor:NonlinearOrnsteinUhlenbeckSteadyStateMoreRegularity}, 
\ref{cor:pointwise} to the real-valued system \eqref{equ:ComplexQuinticCubicGinzburgLandauEquationRealVersion}  
and Corollary \ref{cor:NonlinearOrnsteinUhlenbeckSteadyStateComplexVersion}
to the complex equation \eqref{equ:ComplexQuinticCubicGinzburgLandauEquation}.
% c.1) Assumptions for real-valued system 

First, consider the assumptions \eqref{cond:A1}--\eqref{cond:A10} for $\K=\R$: With $\mathbf{A}$ from \eqref{equ:ComplexQuinticCubicGinzburgLandauEquationRealVersion} 
and $\mathbf{f}$ from \eqref{equ:NonlinearityGLRealVersion}, Assumption \eqref{cond:A1} is satisfied for every $\alpha\in\C$, since 
\begin{align}
  \label{equ:Diagonalization}
  Y^{-1}\mathbf{A}Y = \Lambda_{\mathbf{A}},\quad 
  \Lambda_{\mathbf{A}}=\begin{pmatrix} \alpha & 0 \\ 0 & \overline{\alpha}\end{pmatrix},\quad
  Y=\begin{pmatrix} i & 1 \\ 1 & i\end{pmatrix},\quad
  Y^{-1}=\frac{1}{2}\begin{pmatrix} -i & 1 \\ 1 & -i\end{pmatrix}.
\end{align}
Assumption \eqref{cond:A2} follows from $\Re\alpha>0$, since $\sigma(\mathbf{A})=\{\alpha,\overline{\alpha}\}$. Assumption \eqref{cond:A3} holds with $\beta_{\mathbf{A}}=\Re\alpha$ if $\Re\alpha>0$, 
since $\Re\left\langle w,\mathbf{A}w\right\rangle = \Re\alpha$ for $w\in\R^2$ with $|w|=1$. Condition \eqref{cond:A4}, which is equivalent to \eqref{cond:A4DC},
requires $\alpha\neq 0$ and $\frac{\Re\alpha}{|\alpha|}=\mu_1(\mathbf{A})>\frac{|p-2|}{p}$.
The latter condition is equivalent to
\begin{align*}
  \left|\arg\alpha\right|<\arctan\left(\frac{2\sqrt{p-1}}{\left|p-2\right|}\right)\text{ for some $1<p<\infty$,}
\end{align*}
or alternatively to
\begin{align}
  \label{equ:PGeneralBoundGinzburgLandau}
  p_{\mathrm{min}}:=\frac{2|\alpha|}{|\alpha|+\Re\alpha}<p<\frac{2|\alpha|}{|\alpha|-\Re\alpha}=:p_{\mathrm{max}}.
\end{align}
The condition \eqref{cond:A5} is satisfied with $S\in\R^{d,d}$ given by
\begin{align}
  \label{equ:MatrixS}
  S=\begin{pmatrix} 0 & S_{12} \\ -S_{12} & 0\end{pmatrix}\quad\text{and}\quad S=\begin{pmatrix} 0 & S_{12} & S_{13} \\ -S_{12} & 0 & S_{23} \\ -S_{13} & -S_{23} & 0\end{pmatrix}
\end{align}
for $d=2$ and $d=3$, respectively. 
Below we specify the entries $S_{12}$, $S_{13}$, $S_{23}\in\R$ and the point $x_{\star}\in\R^d$, 
that will be the center of rotation if $d=2$ and a support vector of the axis of rotation if $d=3$, cf. \eqref{equ:RotatingWave}. All this information 
come actually from a simulation. First we simulate the original system for some time. Then we switch to the freezing method, which yields the profile 
$v_{\star}$, its center of rotation, and its rotational velocities. For more details on the computation, see \cite[Section 10.3]{Otten2014}. 
Some  general theory and applications of the freezing method may be found in \cite{BeynOttenRottmannMatthes2013, BeynThuemmler2004, BeynThuemmler2007, BeynThuemmler2009, 
Thuemmler2006}. Note that in case $d=2$ we have a clockwise rotation, if $S_{12}>0$, and a counter clockwise rotation, if $S_{12}<0$. 
Assumption \eqref{cond:A6} is obviously satisfied, even with  $\mathbf{f}\in C^{\infty}(\R^2,\R^2)$, since every component of $\mathbf{f}$ is a polynomial.
With  $\mathbf{v_{\infty}}=(0,0)^T$, the assumption \eqref{cond:A7} is satisfied, and for this choice we have
\begin{align*}
  D\mathbf{f}(\mathbf{v_{\infty}})=\begin{pmatrix} \delta_1 & -\delta_2 \\
                                 \delta_2 &  \delta_1\end{pmatrix}.
\end{align*}
Assumption \eqref{cond:A8} holds for the same transformation matrix $Y$ as in \eqref{equ:Diagonalization}. The condition $\Re\delta<0$ implies both, 
Assumption \eqref{cond:A9}  and Assumption \eqref{cond:A10} with $\beta_{\infty}=-\Re\delta$. 

% c.2) Assumptions for complex equation
Next we consider assumptions \eqref{cond:A6g}, \eqref{cond:A8g} and \eqref{cond:A10g}: Writing $f$ as $f(u)=g(|u|^2)u$ with
\begin{align*}
  g:\R\rightarrow\C,\quad g(v)=\delta+\beta v+\gamma v^2,
\end{align*}
Assumption \eqref{cond:A6g} is obivously satisfied and we even have $g\in C^{\infty}(\R,\C)$. Assumption \eqref{cond:A8g} is 
satisfied with $g(0)=\delta$ for every $\alpha,\delta\in\C$ and assumption \eqref{cond:A10g} with $\beta_{\infty}=-\Re\delta$ if $\Re\delta<0$. 
The assumptions \eqref{cond:A1}--\eqref{cond:A4} for $A=\alpha\in\C$ lead to the same requirements as in the real-valued case.

Our discussion shows that if we assume 
\begin{align}
  \label{equ:AssumptionsQCGL}
  \Re\alpha>0,\quad \Re\delta<0,\quad p_{\min}= \frac{2|\alpha|}{|\alpha|+\Re\alpha}<p<\frac{2|\alpha|}{|\alpha|-\Re\alpha}= p_{\max},
\end{align}
we can apply Theorem \ref{thm:NonlinearOrnsteinUhlenbeckSteadyState}, Corollary \ref{cor:NonlinearOrnsteinUhlenbeckSteadyStateMoreRegularity} 
and Corollary \ref{cor:pointwise} to the real-valued system \eqref{equ:ComplexQuinticCubicGinzburgLandauEquationRealVersion}, and Corollary 
\ref{cor:NonlinearOrnsteinUhlenbeckSteadyStateComplexVersion} to the complex-valued equation \eqref{equ:ComplexQuinticCubicGinzburgLandauEquation}. 
In both cases, the bound for the rate of the exponential decay reads
\begin{align}
  \label{equ:QCGLBound1}
  0\leqslant \mu \leqslant \varepsilon \frac{\nu}{p},\quad \text{for } \nu=\frac{\sqrt{\Re\alpha\left(-\Re\delta\right)}}{|\alpha| }\text{ and some $0<\varepsilon<1$,}
\end{align}
since $\azero=\Re\alpha$, $\bzero=-\Re\delta$ and $\amax=|\alpha|$.

%----------------------------------------------------------------------------------
% SUBSECTION 6.1: (Spinning solitons).
%----------------------------------------------------------------------------------
\subsection{Spinning solitons}
\label{subsec:6.1}
%----------------------------------------------------------------------------------
% Discussion of the assumptions for spinning solitons
For the parameter values from \cite{CrasovanMalomedMihalache2001}, given by
\begin{align}
  \label{equ:ParameterValuesQCGLNr1}
  \alpha=\frac{1}{2}+\frac{1}{2}i,\quad\beta=\frac{5}{2}+i,\quad\gamma=-1-\frac{1}{10}i,\quad\delta=-\frac{1}{2},
\end{align}
equation \eqref{equ:ComplexQuinticCubicGinzburgLandauEquation} exhibits so called \begriff{spinning soliton} solutions for space dimensions $d=2$ and $d=3$, 
see Figure \ref{fig:QCGLSpinningSolitonProfile}. The parameter values \eqref{equ:ParameterValuesQCGLNr1} satisfy the requirements from \eqref{equ:AssumptionsQCGL}, 
and therefore our assumptions \eqref{cond:A1}--\eqref{cond:A10} for every $p$ with
\begin{align}
  \label{equ:PBoundGinzburgLandau}
  1.1716\approx\frac{4}{2+\sqrt{2}}=p_{\min}<p<p_{\max}=\frac{4}{2-\sqrt{2}}\approx 6.8284,
\end{align}
e.g. for $p=2,3,4,5,6$. Therefore, the solitons (and their derivatives) are exponentially localized in the sense of Theorem \ref{thm:NonlinearOrnsteinUhlenbeckSteadyState} and 
Corollary \ref{cor:NonlinearOrnsteinUhlenbeckSteadyStateMoreRegularity}, i.e. $\mathbf{v_{\star}}$ belongs to $W^{2,p}_{\theta}(\R^d,\R^2)$ for $p\in (p_{\min},p_{\max})$ and for the  
weight function $\theta(x)=\exp\left(\mu\sqrt{|x|^2+1}\right)$ with exponential decay rate 
\begin{align}
  \label{equ:muBoundGinzburgLandau}
  0\leqslant \mu \leqslant \frac{\varepsilon}{\sqrt{2} p}.
\end{align}
Corollary \ref{cor:NonlinearOrnsteinUhlenbeckSteadyStateMoreRegularity} implies that $\mathbf{v_{\star}}$ even belongs to $W^{k,p}_{\theta}(\R^d,\R^2)$
for every $k \geqslant 0$, 
provided $d=2$, $p\in (p_{\min},p_{\max})$ or $d=3$, $p\in [\frac{3}{2},p_{\max})$ since $\mathbf{f}\in C^{\infty}(\R^2,\R^2)$. Moreover, Corollary \ref{cor:pointwise} 
shows that the solitons satisy the pointwise estimates \eqref{equ:PointwiseEstimateVstar}. In Section \ref{subsec:6.3}  we will compare
this with the rate of decay measured from numerical experiments. 

% Discussion of numerical results (Profile and velocities)
Next, we discuss the numerical results from Figure \ref{fig:QCGLSpinningSolitonProfile} and explain how to compute the profile and velocities numerically. 
For all numerical computations including the eigenvalue computations we used Comsol Multiphysics 5.2, \cite{ComsolMultiphysics52}.

\begin{figure}[ht]
  \centering
  \subfigure[]{\includegraphics[page=6,height=4.0cm] {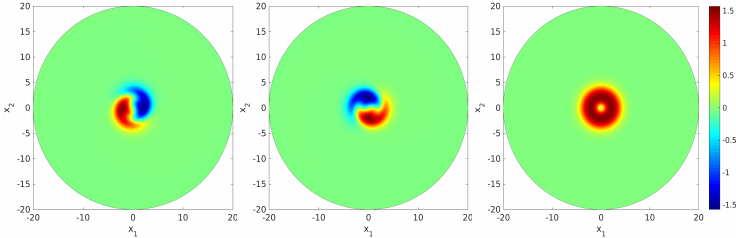} \label{fig:QCGLSpinningSoliton2DProfile}}\\
  \subfigure[]{\includegraphics[page=7,height=4.0cm] {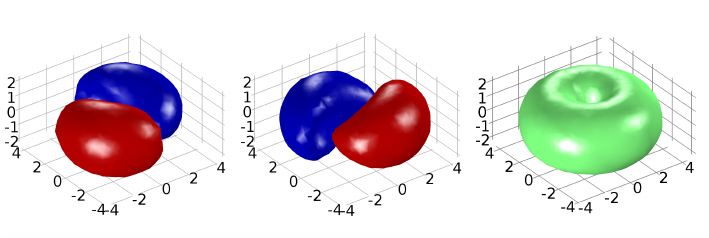} \label{fig:QCGLSpinningSoliton3DProfile}}
  \caption{Spinning soliton of QCGL \eqref{equ:ComplexQuinticCubicGinzburgLandauEquation} for $d=2$ (a) and $d=3$ (b) with real part (left), imaginary part (middle) 
  and absolute value (right). The colorbar in (a) reaches from $-1.6$ (blue) to $1.6$ (red). The isosurfaces in (b) have values $-0.5$ (blue), $0.5$ (red) and $0.5$ (green).}
  \label{fig:QCGLSpinningSolitonProfile}
\end{figure}

Figure \ref{fig:QCGLSpinningSolitonProfile}(a) shows the spinning soliton in $\R^2$ as the solution of 
\eqref{equ:ComplexQuinticCubicGinzburgLandauEquation} on a circular disk of radius $R=20$ centered at the origin at time $t=150$. For the computation 
we used continuous piecewise linear finite elements with maximal stepsize $\triangle x=0.25$, the BDF method of order $2$ with absolute tolerance 
$\mathrm{atol}=10^{-5}$, relative tolerance $\mathrm{rtol}=10^{-3}$ and maximal stepsize $\triangle t=0.1$, homogeneous Neumann boundary conditions 
and initial data
\begin{align*}
  u_0^{2D}(x_1,x_2)=\frac{1}{5}\left(x_1+ix_2\right)\exp\left(-\frac{x_1^2+x_2^2}{49}\right), \quad x_1^2 + x_2^2 \le R^2.
\end{align*}
  
Figure \ref{fig:QCGLSpinningSolitonProfile}(b) shows isosurfaces of the spinning soliton in $\R^3$ obtained from the solution of 
\eqref{equ:ComplexQuinticCubicGinzburgLandauEquation} on a cube 
$[-10,10]^3$ at time $t=100$. For the computation we 
used continuous piecewise linear finite elements with maximal stepsize $\triangle x=0.8$, the BDF method of order $2$ with absolute tolerance $\mathrm{atol}=10^{-4}$, 
relative tolerance $\mathrm{rtol}=10^{-2}$ and maximal stepsize $\triangle t=0.1$, homogeneous Neumann boundary conditions and (discontinuous) initial data
\begin{align*}
  u_0^{3D}(x_1,x_2,x_3)=u_0^{2D}(x_1,x_2)\text{ for $|x_3|<9$ and $0$ otherwise}.
\end{align*}

Using these solutions as initial data, the freezing method from \cite{BeynOttenRottmannMatthes2013,BeynThuemmler2004}
provides an approximate rotating wave with profile $\mathbf{w}_{\star}$
 in the following format
 \begin{equation} \label{equ:freezingformat}
\mathbf{u}_{\star}(x,t) =\mathbf{w}_{\star}(\exp(-tS)(x- tE(tS) \tau)),\quad x \in \R^d,\;t\in \R,
\end{equation}
where $S\in \R^{d,d}$ is skew symmetric, $\tau \in \mathrm{range}(S) \subset \R^d$ and $E$ is 
the analytic function  $E(z) = \sum_{j=1}^{\infty}\frac{z^{j-1}}{j!}$ satisfying
$E(z)z=\exp(z)-1, z \in \C$. In order to put this into the standard
form \eqref{equ:RotatingWave} used in our theory, we determine the position
vector $x_{\star}\in \R^d$ by solving
\begin{equation} \label{equ:shiftwave}
Sx_{\star} + \tau = 0.
\end{equation}
Then we may write $\exp(-tS)(x- tE(tS) \tau)=
\exp(-tS)(x+ E(tS)(tS)x_{\star})= \exp(-tS)(x-x_{\star})+ x_{\star}$, so that
\eqref{equ:freezingformat} turns into
\begin{equation} \label{equ:freezingformat2}
\mathbf{u}_{\star}(x,t) =\mathbf{w}_{\star}(\exp(-tS)(x-x_{\star})+x_{\star})
= \mathbf{v}_{\star}(\exp(-tS)(x-x_{\star})).
\end{equation}
Thus the numerical profile $\mathbf{w}_{\star}$ is a slightly shifted version
of the profile $\mathbf{v}_{\star}$ used for the theory. In practice, we solve
\eqref{equ:shiftwave} directly for $d=2$ and by rank-deficient least squares for $d=3$, see
\cite[Example 10.8]{Otten2014} for details. 

% These computations are necessary, because
% in general we cannot expect the solution
% of the evolution equation to converge to a profile rotating about the origin
% but about a slightly perturbed position.

Using for $d=2$ the data at time $t=400$, one obtains the following values for $S_{12}$, $\tau$ and the center of rotation $x_{\star}^{2D}$ 
of the spinning soliton
\begin{align}
  \label{equ:velocities2D}
  S_{12}=1.0286,\quad \tau=\begin{pmatrix}-0.0054\\-0.0071\end{pmatrix},\quad 
  x_{\star}^{2D} := -S^{-1}\tau = \frac{1}{S_{12}}\begin{pmatrix}\tau_2\\-\tau_1\end{pmatrix} = \begin{pmatrix}-0.0069\\0.0052\end{pmatrix}.
\end{align}
The rotating wave $\mathbf{u_{\star}}:\R^2\times[0,\infty)\rightarrow\R^2$ satisfies
\begin{align*}
  \mathbf{u_{\star}}(x,t) = \mathbf{w_{\star}}\left(e^{-t S}(x-x_{\star}^{2D})+x_{\star}^{2D}\right) 
  = \mathbf{v_{\star}}\left(e^{-t S}(x-x_{\star}^{2D})\right).
\end{align*}

In case $d=3$ at time $t=500$ the rotational velocities $S_{12},S_{13},S_{23}$ and the translational vector $\tau$ of the spinning soliton are
found to be
\begin{align}
  \label{equ:velocities3D}
  \begin{pmatrix}S_{12}\\S_{13}\\S_{23}\end{pmatrix}=\begin{pmatrix}0.6888\\-0.0043\\-0.0043\end{pmatrix},\quad
  \tau=\begin{pmatrix}0.0023\\-0.0415\\0.0005\end{pmatrix}.
\end{align}
The axis of rotation is $\{ x_{\star}^{3D} + r x_{\mathrm{rot}}^{3D}, r \in \R\}$ with 
support vector $x_{\star}^{3D}$ and the direction 
$x_{\mathrm{rot}}^{3D}$  spanning the null space of $S$, is given by
\begin{align*}
  &x_{\mathrm{rot}}^{3D} := \begin{pmatrix}S_{23}\\-S_{13}\\S_{12}\end{pmatrix} = \begin{pmatrix}-0.0043\\0.0043\\0.6888\end{pmatrix},\quad
  x_{\star}^{3D} := \frac{1}{S_{12}^2+S_{13}^2+S_{23}^2}\begin{pmatrix}S_{12}\tau_2+S_{13}\tau_3\\-S_{12}\tau_1+S_{23}\tau_3\\-S_{13}\tau_1-S_{23}\tau_2\end{pmatrix}
        = \begin{pmatrix}-0.0602\\-0.0033\\-0.0004\end{pmatrix}.
% \\
%   &a_{\mathrm{rot}}^{3D}(r) := x_{\star}^{3D} + r x_{\mathrm{dv}}^{3D} = \begin{pmatrix}-0.0603\\-0.0033\\-0.0004\end{pmatrix} + r\begin{pmatrix}0.0005\\0.0415\\0.0023\end{pmatrix}, r \in \R.
\end{align*}

% Note the equalities $x_{\mathrm{dv}}^{3D}=\omega$, $x_{\star}^{3D}=\frac{\tau\times\omega}{|\omega|^2}$ and $a_{\mathrm{rot}}^{3D}(r)=\omega+\frac{r(\tau\times\omega)}{|\omega|^2}$ 
% with angular velocity tensor $\omega=(S_{23},-S_{13},S_{12})^T$ and translational vector $\tau=(\tau_1,\tau_2,\tau_3)^T$.
As above, the rotating wave 
$\mathbf{u_{\star}}:\R^3\times[0,\infty)\rightarrow\R^2$ satisfies
\begin{align*}
  \mathbf{u_{\star}}(x,t) = \mathbf{w_{\star}}\left(e^{-t S}(x-x_{\star}^{3D})+x_{\star}^{3D}\right) 
  = \mathbf{v_{\star}}\left(e^{-t S}(x-x_{\star}^{3D})\right).
\end{align*}
The  periods of rotation for the spinning solitons in $\R^2$ and $\R^3$ are
determined by
\begin{align*}
 T^{2D}=\frac{2\pi}{|S_{12}|}=6.1085\quad\text{and}\quad T^{3D}=\frac{2\pi}{\left|\sqrt{S_{12}^2+S_{13}^2+S_{23}^2}\right|}=9.1216. 
\end{align*}

%----------------------------------------------------------------------------------
% SUBSECTION 6.2: (Spectrum and eigenfunctions at rotating solitons).
%----------------------------------------------------------------------------------
\subsection{Spectrum and eigenfunctions at spinning solitons}
\label{subsec:6.2}
%----------------------------------------------------------------------------------
We now consider the eigenvalue problem for the real-valued version of the
QCGL-equation, cf. \eqref{equ:ComplexQuinticCubicGinzburgLandauEquationRealVersion},
\begin{align}
  \label{equ:ComplexQuinticCubicGinzburgLandauEquationEigenvalueProblem}
  \L \mathbf{v}(x)=\mathbf{A}\triangle \mathbf{v}(x) + \left\langle Sx,\nabla \mathbf{v}(x)\right\rangle 
  +D\mathbf{f}\left(\mathbf{v_{\star}}(x)\right)\mathbf{v}(x) = \lambda \mathbf{v}(x),\,x\in\R^d,\,d\in\{2,3\},
\end{align}
with $\mathbf{v}:\R^d\rightarrow\C^2$, $\mathbf{A}\in\R^{2,2}$ from \eqref{equ:ComplexQuinticCubicGinzburgLandauEquationRealVersion}, 
$\mathbf{f}:\R^2\rightarrow\R^2$ from \eqref{equ:NonlinearityGLRealVersion} and $S\in\R^{d,d}$ from \eqref{equ:MatrixS}. 

Recall from Section \ref{subsec:6.1}  that the Ginzburg-Landau equation exhibits
spinning soliton solutions for space dimensions $d=2$ and $d=3$
and for the parameter values from \eqref{equ:ParameterValuesQCGLNr1}.
% that this parameter values satisfy the assumptions \eqref{cond:A1}--\eqref{cond:A10} for every 
% $p\in (p_{\min},p_{\max})$ with $p_{\min},p_{\max}$ from \eqref{equ:PBoundGinzburgLandau}, i.e. $p=2,3,4,5,6$, and that the solitons are exponentially localized 
% in the sense of Theorem \ref{thm:NonlinearOrnsteinUhlenbeckSteadyState} (Corollary \ref{cor:NonlinearOrnsteinUhlenbeckSteadyStateMoreRegularity} and 
% Corollary \ref{cor:pointwise}) for every exponential decay rate $\mu$ satisfying \eqref{equ:QCGLBound1}.

Below we approximate solutions $(\lambda,\mathbf{v})$ 
of the eigenvalue problem \eqref{equ:ComplexQuinticCubicGinzburgLandauEquationEigenvalueProblem} and apply Theorem \ref{thm:LinearizedOrnsteinUhlenbeckExponentialDecayInLp} 
and Corollary \ref{cor:ExponentialDecayOfEigenfunctions} to \eqref{equ:ComplexQuinticCubicGinzburgLandauEquationEigenvalueProblem}. 
Instead of \eqref{equ:ComplexQuinticCubicGinzburgLandauEquationEigenvalueProblem}  we solve the eigenvalue problem  
\begin{align}
  \label{equ:ComplexQuinticCubicGinzburgLandauEquationEigenvalueProblemNumerically}
  \mathbf{A}\triangle \mathbf{w}(x) + \left\langle S(x-x_{\star}),\nabla \mathbf{w}(x)\right\rangle 
  +D\mathbf{f}\left(\mathbf{w_{\star}}(x)\right)\mathbf{w}(x) = \lambda \mathbf{w}(x),\,x\in\R^d,\,d\in\{2,3\},
\end{align}
with $S\in\R^{d,d}$ and $\tau\in\R^d$ from \eqref{equ:velocities2D} and \eqref{equ:velocities3D}.
Note that \eqref{equ:ComplexQuinticCubicGinzburgLandauEquationEigenvalueProblemNumerically} is just the shifted version of \eqref{equ:ComplexQuinticCubicGinzburgLandauEquationEigenvalueProblem} 
where $\mathbf{w}(x)=\mathbf{v}(x-x_{\star}),x\in \R^d$. Hence the 
eigenvalues are the same.

We use the following values 
 \begin{align}
  \label{equ:ParametersSpectrumQCGLSpinningSolitons}
  %\begin{split}
    &\mathbf{v_{\infty}}=\begin{pmatrix}0\\0\end{pmatrix},\quad \sigma\left(D\mathbf{f}(\mathbf{v_{\infty}})\right)=\left\{\delta,\bar{\delta}\right\}=\left\{-\frac{1}{2}\right\},\quad
    \bzero = -s\left(D\mathbf{f}(\mathbf{v_{\infty}})\right) = -\Re\delta = \frac{1}{2}.
  %\end{split}
\end{align}
The spectrum of $S\in\R^{d,d}$ is given by
\begin{align}
  \label{equ:QCGLSpinningSoliton2DEssentialSpectrumParameters}
  d=2:&\quad\sigma(S)=\left\{\pm\sigma_1 i\right\},\quad \sigma_1=S_{12}=1.0286, \\
  \label{equ:QCGLSpinningSoliton3DEssentialSpectrumParameters}
  d=3:&\quad\sigma(S)=\left\{0,\pm\sigma_1 i\right\},\quad \sigma_1=\sqrt{S_{12}^2+S_{13}^2+S_{23}^2}=0.6888.
\end{align}
For the solution of the eigenvalue problem \eqref{equ:ComplexQuinticCubicGinzburgLandauEquationEigenvalueProblemNumerically} we use in both cases, $d=2$ and $d=3$, 
continuous piecewise linear finite elements with maximal stepsize $\triangle x=0.25$ (if $d=2$) and $\triangle x=0.8$ (if $d=3$), homogeneous Neumann boundary 
conditions and the following parameters for the eigenvalue solver
\begin{align}
  \label{equ:FemeigSolverSettingsQCGLSpinningSoliton}
  \textrm{neigs}=800,\quad \sigma=-1,\quad \textrm{etol}=10^{-7},
\end{align}
i.e. we approximate $\textrm{neigs}=800$ eigenvalues that are located nearest to $\sigma=-1$  and satisfy the eigenvalue tolerance $\textrm{etol}=10^{-7}$.
As above, the profile $\mathbf{w_{\star}}$ and the pair $(S,x_{\star})$ in \eqref{equ:ComplexQuinticCubicGinzburgLandauEquationEigenvalueProblemNumerically} are
obtained from simulating the freezing system until $t=400$ for $d=2$  and
until $t=500$ for $d=3$. With the data from the last time instance we then solve
the eigenvalue problem \eqref{equ:ComplexQuinticCubicGinzburgLandauEquationEigenvalueProblemNumerically}.

\begin{figure}[ht]
  \centering
  \subfigure[]{\includegraphics[page=8,width=0.48\textwidth] {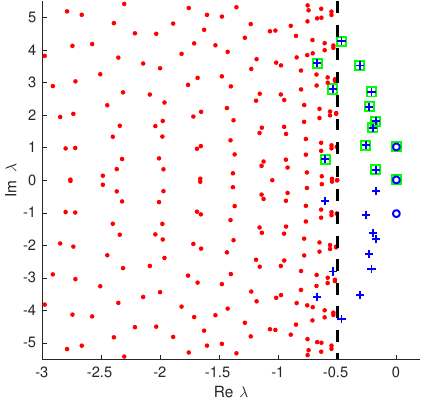}       \label{fig:QCGLSpinningSoliton2DSpectrum}}
  \subfigure[]{\includegraphics[page=9,width=0.48\textwidth] {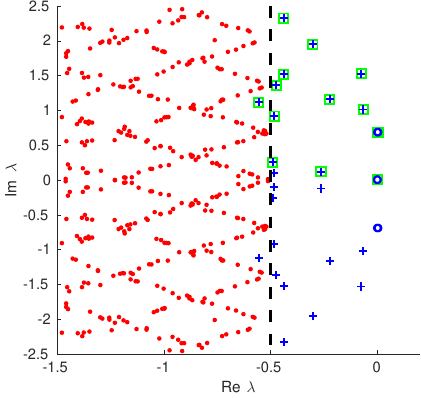}       \label{fig:QCGLSpinningSoliton3DSpectrum}}
  \caption{Numerical spectra of QCGL \eqref{equ:ComplexQuinticCubicGinzburgLandauEquation} linearized about a spinning soliton for $d=2$ (a), and $d=3$ (b).}
  \label{fig:QCGLSpinningSolitonSpectrum}
\end{figure} 

Figure \ref{fig:QCGLSpinningSolitonSpectrum} shows the approximation $\sigma^{\mathrm{approx}}(\L)$ of the spectrum $\sigma(\L)=\sigma_{\mathrm{point}}(\L)\cup \sigma_{\mathrm{ess}}(\L)$ of $\L$
obtained by linearizing about the spinning 
soliton $\mathbf{v_{\star}}$ for $d=2$ (a), and $d=3$ (b). Let us discuss the numerical spectra in more detail and compare them with our theoretical results:

\textbf{Essential spectrum:} 
We replace $\mathbf{v}_{\star}$ in $\L$ by its limiting
value zero and find a dispersion relation \eqref{equ:DispersionRelation}, leading to 
\begin{align}
  \label{equ:EssentialSpectrumQCGLSpinningSolitons}
  \sigma_{\mathrm{ess}}^{\mathrm{part}}(\L)
  =\left\{\lambda=-\omega^2\alpha_1 + \delta_1 + i\left(\mp 
\omega^2\alpha_2\pm\delta_2-n\sigma_1\right)\mid \omega\in\R,\;n\in\Z\right\} \subseteq \sigma_{\mathrm{ess}}(\L)
\end{align}
with $\sigma_1$ from \eqref{equ:QCGLSpinningSoliton2DEssentialSpectrumParameters} for $d=2$ and from \eqref{equ:QCGLSpinningSoliton3DEssentialSpectrumParameters} 
for $d=3$, cf. Remark \ref{rem:PointEssential} and \cite[Theorem 9.10]{Otten2014}. 
 Taking the parameter values 
\eqref{equ:ParameterValuesQCGLNr1} into account, \eqref{equ:EssentialSpectrumQCGLSpinningSolitons} reads as
\begin{align*}
  \sigma_{\mathrm{ess}}^{\mathrm{part}}(\L)
  =\left\{\lambda=-\frac{1}{2}(\omega^2+1)+i\left(\mp \frac{1}{2}\omega^2-n\sigma_1\right)\mid \omega\in\R,\;n\in\Z\right\} \subseteq \sigma_{\mathrm{ess}}(\L)
\end{align*}
In both cases the part $\sigma_{\mathrm{ess}}^{\mathrm{part}}(\L)$ of the essential spectrum forms a zig-zag-structure that can be considered as 
the union of infinitely many copies of cones. The tips of the cones  $-\frac{1}{2}-in\sigma_1$, $n\in\Z$ lie on the line 
$\delta_1+i\R=-\frac{1}{2}+i\R$. Therefore, the distance between two neighboring tips equals $\sigma_1$. The gap between 
the whole essential spectrum and the imaginary axis equals $\bzero=\frac{1}{2}$, since $\Re\sigma_{\mathrm{ess}}(\L)\leqslant-\bzero=\Re\delta=-\frac{1}{2}$. 
The inclusion $\sigma_{\mathrm{ess}}^{\mathrm{part}}(\L)\subseteq\sigma_{\mathrm{ess}}(\L)$ is proved in \cite[Theorem 9.10]{Otten2014a} for the 
$L^p$-spectrum of $\L$. We believe that even equality holds, i.e. $\sigma_{\mathrm{ess}}^{\mathrm{part}}(\L)=\sigma_{\mathrm{ess}}(\L)$, but this has 
not been proved so far.
%By the way note that the essential spectrum $\sigma_{\mathrm{ess}}(\L)$ in general might be larger than $\sigma_{\mathrm{ess}}^{\mathrm{part}}(\L)$.

Let us now consider the numerical results: The red dots in Figure \ref{fig:QCGLSpinningSolitonSpectrum} represent the approximation 
$\sigma_{\mathrm{ess}}^{\mathrm{approx}}(\L)$ of the essential spectrum $\sigma_{\mathrm{ess}}(\L)$. They approximate the collection of cones in  the essential spectrum. As expected, the tips are approximatively located on $-\frac{1}{2}+i\R$, indicated by the black dashed line. 
Our results show that the distance between two neighboring tips of the cones agrees with $\sigma_1$  from \eqref{equ:QCGLSpinningSoliton2DEssentialSpectrumParameters} 
for $d=2$ and from \eqref{equ:QCGLSpinningSoliton3DEssentialSpectrumParameters} for $d=3$. In particular, the approximation suggests that we have  
equality in \eqref{equ:EssentialSpectrumQCGLSpinningSolitons}.
 The case $d=2$ has been also treated in \cite[Section 8]{BeynLorenz2008}. 

\textbf{Point spectrum:} From Theorem \ref{thm:EigenfunctionsOfTheLinearizedOrnsteinUhlenbeckInLp} and Remark \ref{rem:PointEssential}
we have the relation
\begin{align}
  \label{equ:PointSpectrumQCGLSpinningSolitons}
  \sigma_{\mathrm{point}}^{\mathrm{part}}(\L) = \{0,\pm i\sigma_1\} \subseteq \sigma_{\mathrm{point}}(\L)
\end{align}
with $\sigma_1$ from \eqref{equ:QCGLSpinningSoliton2DEssentialSpectrumParameters} for $d=2$ and \eqref{equ:QCGLSpinningSoliton3DEssentialSpectrumParameters} 
for $d=3$ (cf. \cite[Theorem 9.4]{Otten2014} for 
more details). The eigenvalues $0,\pm i\sigma_1$ are located on the imaginary axis and have (at least) algebraic multiplicities $1$ for $d=2$ and $2$ for $d=3$, respectively, 
see Theorem \ref{thm:EigenfunctionsOfTheLinearizedOrnsteinUhlenbeckInLp}. % Similarly as for travelling waves, it can be proved by an application of Fredholm 
% theory that the algebraic multiplicity of these eigenvalues is even exactly $1$ for $d=2$ and $2$ for $d=3$, but this has not been proved so far.
The numerical results below suggest that we do not have equality in \eqref{equ:PointSpectrumQCGLSpinningSolitons}, i.e. in general there are further isolated 
eigenvalues which have to be determined numerically.

Let us now consider the associated numerical results: The blue circles and the blue plus signs in Figure \ref{fig:QCGLSpinningSolitonSpectrum} represent an 
approximation $\sigma_{\mathrm{point}}^{\mathrm{approx}}(\L)$ of the point spectrum $\sigma_{\mathrm{point}}(\L)$. The approximate eigenvalues 
$0,\pm i\sigma_1$ are visualized by  blue circles. The plus signs indicate further isolated eigenvalues which belong to the point spectrum, 
but cannot be determined explicitly. In case $d=2$, there are $11$
 additional complex-conjugate pairs of isolated eigenvalues, of which 
$8$ pairs are located to the right of the vertical black dashed line $-\frac{1}{2}+i\R$ and $3$ pairs to the left in between the zig-zags.  We emphasize
that these three pairs are not numerical artifacts but are robust to spatial mesh refinement and to increase of spatial domain.
Similarly, in case $d=3$, we find $12$ additonal complex-conjugate pairs of isolated eigenvalues, of which $11$ pairs are located to the right of the vertical 
line $-\frac{1}{2}+i\R$ and $1$ pair to the left in between the zig-zags. Numerical values of the isolated eigenvalues are given in Table \ref{tab:DecayRatesSpinningSoliton} 
below.  

\begin{figure}[ht]
  \centering
  \includegraphics[page=10,width=0.99\textwidth] {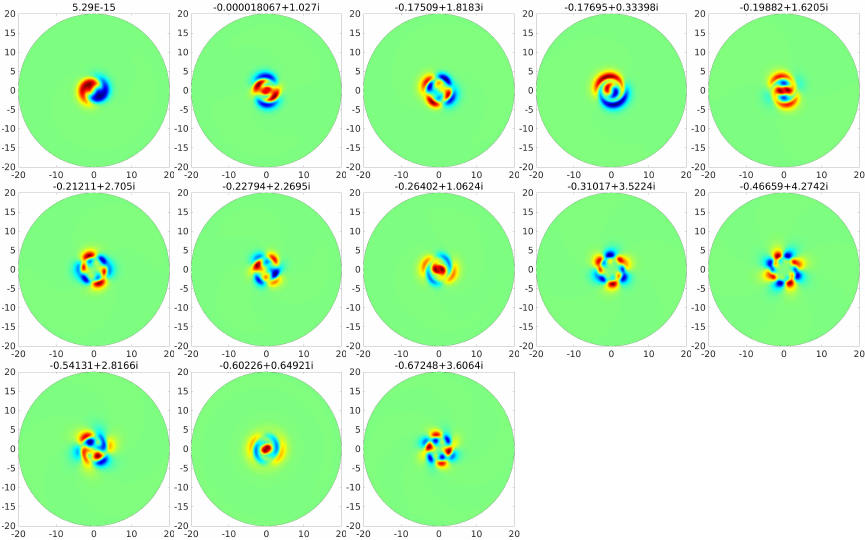}
  \caption{Real parts of eigenfunctions of $2D$-QCGL \eqref{equ:ComplexQuinticCubicGinzburgLandauEquation} for a spinning soliton.}
  \label{fig:QCGLSpinningSoliton2DEigenfunctions}
\end{figure} 

\begin{figure}[ht]
  \centering
  \includegraphics[page=11,width=0.99\textwidth] {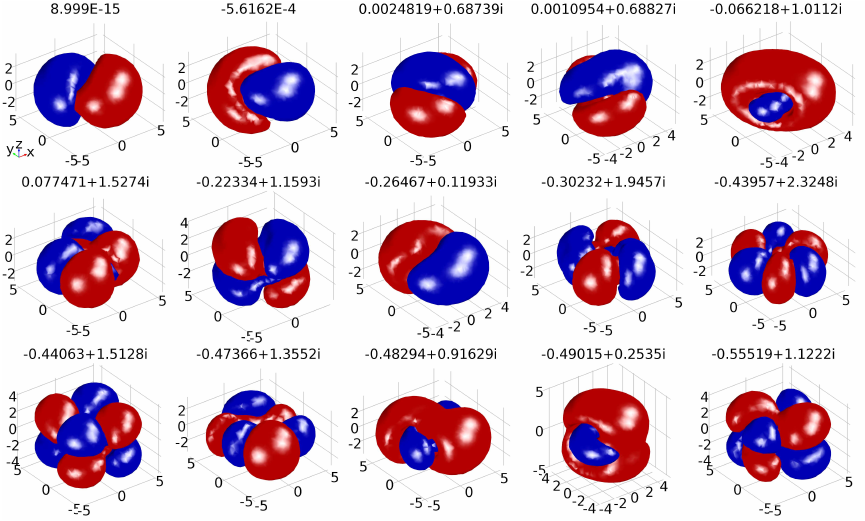}
  \caption{Isosurfaces of real parts of eigenfunctions of $3D$-QCGL \eqref{equ:ComplexQuinticCubicGinzburgLandauEquation} for a spinning soliton.}
  \label{fig:QCGLSpinningSoliton3DEigenfunctions}
\end{figure}

\textbf{Eigenfunctions:} The eigenfunctions associated to eigenvalues from the essential spectrum are explicitly known and bounded but  
never localized, i.e. they do not decay in space, see \cite[Theorem 7.9 and 9.10]{Otten2014}. In contrast to this, 
all eigenfunctions associated to eigenvalues from the point spectrum, in particular those on the imaginary axis, are 
exponentially localized. For eigenvalues $\Re\lambda>-\beta_{\infty}=-\frac{1}{2}$ this follows from our theory. To be more precise, 
let us introduce angular derivatives by 
\begin{align*}
  D^{(1,2)}:=x_2 D_1-x_1 D_2,\quad D^{(1,3)}:=x_3 D_1-x_1 D_3,\quad D^{(2,3)}:=x_3 D_2-x_2 D_3.
\end{align*}
Then Theorem \ref{thm:EigenfunctionsOfTheLinearizedOrnsteinUhlenbeckInLp} asserts  eigenfunctions associated to eigenvalues from 
$\sigma_{\mathrm{point}}^{\mathrm{part}}(\L)$ as follows
\begin{equation}
  \begin{aligned}
  \label{equ:Eigenfunctions2D}
  \begin{split}
    \lambda_1     &= 0,          &&\mathbf{v_1}=D^{(1,2)}\mathbf{v_{\star}},\\
    \lambda_{2,3} &= \pm i\sigma_1,  &&\mathbf{v_{2,3}}=D_1\mathbf{v_{\star}}\pm iD_2 \mathbf{v_{\star}}
  \end{split}
  \end{aligned}
\end{equation}
for $d=2$, see \cite[Example 9.6]{Otten2014}, and by
\begin{equation}
  \begin{aligned}
  \label{equ:Eigenfunctions3D}
  \begin{split}
    \lambda_1 &= 0,          &&\mathbf{v_1}=S_{12}D^{(1,2)}\mathbf{v_{\star}} + S_{13}D^{(1,3)}\mathbf{v_{\star}} + S_{23}D^{(2,3)}\mathbf{v_{\star}},\\
    \lambda_2 &= 0,          &&\mathbf{v_2}=S_{23}D_1 \mathbf{v_{\star}} - S_{13}D_2 \mathbf{v_{\star}} + S_{12}D_3 \mathbf{v_{\star}},\\
    \lambda_{3,4} &= \pm i\sigma_1,  &&\mathbf{v_{3,4}}=(\sigma_1 S_{13}\pm iS_{12}S_{23})D_1 \mathbf{v_{\star}} 
                                      + (\sigma_1 S_{23}\pm iS_{12}S_{13})D_2 \mathbf{v_{\star}} \\
                &&&\qquad\quad        \pm i(S_{13}^2+S_{23}^2)D_3 \mathbf{v_{\star}}, \\
    \lambda_{5,6} &= \pm i\sigma_1,  &&\mathbf{v_{5,6}}=-(S_{13}^2+S_{23}^2)D^{(1,2)}\mathbf{v_{\star}}-(-S_{12}S_{13}\pm i\sigma_1 S_{23})D^{(1,3)}\mathbf{v_{\star}} \\
                &&&\qquad\quad        +(S_{12}S_{23}\pm i\sigma_1 S_{13})D^{(2,3)}\mathbf{v_{\star}}
  \end{split}
  \end{aligned}
\end{equation}
for $d=3$, see \cite[Example 9.7]{Otten2014}. We next study the asymptotic behavior of eigenfunctions with eigenvalues in $\sigma_{\mathrm{point}}(\L)$.
As shown in the previous section,  Theorem \ref{thm:LinearizedOrnsteinUhlenbeckExponentialDecayInLp} and Corollary \ref{cor:ExponentialDecayOfEigenfunctions} 
 imply that all eigenfunctions with eigenvalues $\Re\lambda>-\frac{1}{2}$
 are exponentially localized, in the $L^p$- and in the pointwise sense. The maximal exponential rate of decay for the eigenfunctions 
will depend on $\lambda$ as we will see in Section \ref{subsec:6.3} below. Note that Theorem \ref{thm:LinearizedOrnsteinUhlenbeckExponentialDecayInLp} 
does not apply to eigenvalues satisfying $\Re\lambda \leqslant -\frac{1}{2}$.

Let us now discuss the numerical results: In Figure \ref{fig:QCGLSpinningSolitonSpectrum}, there are some isolated eigenvalues labeled by a green square. 
Their eigenfunctions are visualized in Figure \ref{fig:QCGLSpinningSoliton2DEigenfunctions} for $d=2$ and in Figure \ref{fig:QCGLSpinningSoliton3DEigenfunctions} 
for $d=3$. Both pictures show the real parts of the first component of the associated eigenfunction $\mathbf{w}:\R^d\rightarrow\C^2$. 
The first two eigenfunctions in Figure \ref{fig:QCGLSpinningSoliton2DEigenfunctions} are approximations of $\mathbf{v_1},\mathbf{v_2}$ from \eqref{equ:Eigenfunctions2D}. 
Their corresponding eigenvalues approximate $\lambda_1,\lambda_2$ from \eqref{equ:Eigenfunctions2D}, as specified in the title of the figure. 
Similarly, the first four eigenfunctions in Figure \ref{fig:QCGLSpinningSoliton3DEigenfunctions} approximate  $\mathbf{v_1},\mathbf{v_2},\mathbf{v_3},\mathbf{v_5}$ 
from \eqref{equ:Eigenfunctions3D}. Their associated eigenvalues are approximations of $\lambda_1,\lambda_2,\lambda_3,\lambda_5$ from \eqref{equ:Eigenfunctions3D} and 
again specified in the title. Note that the first eigenfunction in Figure \ref{fig:QCGLSpinningSoliton2DEigenfunctions} and in Figure 
\ref{fig:QCGLSpinningSoliton3DEigenfunctions} agrees with a slightly shifted
version of the rotational term $\mathbf{v_1}(x)=\left\langle Sx,\nabla \mathbf{v}_{\star}(x)\right\rangle$ 
which arises in the rotating wave equation \eqref{equ:NonlinearSteadyStateProblem2}. The eigenfunctions $3-10$ from Figure \ref{fig:QCGLSpinningSoliton2DEigenfunctions} 
and $5-14$ from Figure \ref{fig:QCGLSpinningSoliton3DEigenfunctions} belong to the eigenvalues in green boxes carrying a plus sign and satisfying $\Re\lambda>-\frac{1}{2}$. 
They are ordered with decaying real parts.
All eigenfunctions with eigenvalues satisfying  $\Re\lambda>-\frac{1}{2}$ seem to decay exponentially, as expected by Theorem \ref{thm:LinearizedOrnsteinUhlenbeckExponentialDecayInLp}.
The last three eigenfunctions in Figure \ref{fig:QCGLSpinningSoliton2DEigenfunctions} and the last eigenfunction in Figure \ref{fig:QCGLSpinningSoliton3DEigenfunctions} 
show those eigenfunctions whose  eigenvalues are marked by a green box but satisfy $\Re\lambda\leqslant-\frac{1}{2}$. In this case Theorem \ref{thm:LinearizedOrnsteinUhlenbeckExponentialDecayInLp} 
is not applicable. However, even these eigenfunctions seem to have exponential decay in space. Finally, we note that we found further isolated eigenvalues inside 
the zig-zag structure, see Figure \ref{fig:QCGLSpinningSolitonSpectrum}(b), the eigenfunctions of which seem to decay exponentially in space as well. 

%----------------------------------------------------------------------------------
% SUBSECTION 6.3: (Numerical rate of decay for spinning solitons and their associated eigenfunctions).
%----------------------------------------------------------------------------------
\subsection{Rate of exponential decay for spinning solitons and their eigenfunctions}
\label{subsec:6.3}
%----------------------------------------------------------------------------------
Let us consider the rates of exponential decay for spinning solitons and their associated eigenfunctions  in more detail. For this purpose 
we compare theoretical decay rates (short: TDR),  guaranteed by Theorem \ref{thm:NonlinearOrnsteinUhlenbeckSteadyState} and \ref{thm:LinearizedOrnsteinUhlenbeckExponentialDecayInLp}, 
with numerical decay rates (short: NDR) computed from our numerical results by linear regression.

\begin{figure}[ht]
  \centering
  \subfigure[]{\includegraphics[page=12,height=4.5cm] {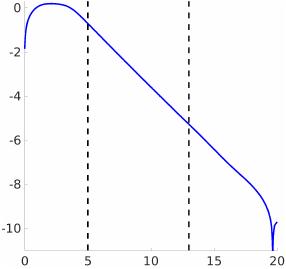}       \label{fig:QCGLSpinningSoliton2DProfileDecay}}
  \subfigure[]{\includegraphics[page=13,height=4.5cm] {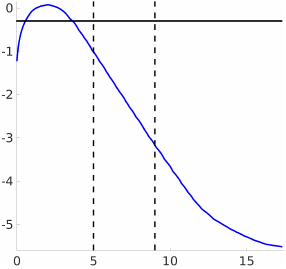}       \label{fig:QCGLSpinningSoliton3DProfileDecay}}
  \caption{Numerical exponential decay rate of the spinning soliton profiles for $d=2$ (a) and $d=3$ (b) arising in the QCGL \eqref{equ:ComplexQuinticCubicGinzburgLandauEquationRealVersion}. 
  The black line indicates the level $0.5$ use for the isosurfaces in Figure \ref{fig:QCGLSpinningSolitonProfile}(b).}
  \label{fig:QCGLSpinningSolitonDecay}
\end{figure} 

\textbf{Decay rates of spinning solitons:} The maximal rate of exponential decay for the profiles of the spinning solitons, which one obtains from 
Corollary \ref{cor:pointwise}, is given by, cf. \eqref{equ:QCGLBound1}, \eqref{equ:PGeneralBoundGinzburgLandau},
% \begin{align*}
%   0\leqslant\mu\leqslant\varepsilon\frac{\sqrt{\Re\alpha(-\Re\delta)}}{|\alpha|p}<\mu_{\max,p}^{\mathrm{pro}}<\frac{\sqrt{\Re\alpha(-\Re\delta)}}{|\alpha|p_{\min}}=\mu_{\max}^{\mathrm{pro}}
% \end{align*}
% with $p_{\min}$ from 
\begin{align}
  \label{equ:GeneralMaximalDecayRateProfile}
 0 \leqslant \mu \leqslant \frac{\varepsilon \nu}{p} < \frac{\nu}{p} =: \mu^{\mathrm{pro}}(p) < \frac{\nu}{\max\left\{p_{\min},\frac{d}{2}\right\}} =: \mu^{\mathrm{pro}}_{\max}. 
\end{align}
Taking the parameter values \eqref{equ:ParameterValuesQCGLNr1} into account, \eqref{equ:GeneralMaximalDecayRateProfile} implies the following upper bounds for the 
theoretical decay rates
\begin{align*}
  \mu^{\mathrm{pro}}(p)=\frac{1}{\sqrt{2}p}\approx\frac{0.7071}{p},\quad\mu_{\max}^{\mathrm{pro}}=\begin{cases}\frac{\sqrt{2}+1}{4}\approx 0.6036&\text{, d=2,}\\\frac{\sqrt{2}}{3}\approx 0.4714&\text{, d=3.}\end{cases}.
\end{align*}

We compare this with the numerical exponential decay rates for the profile: Figure \ref{fig:QCGLSpinningSolitonDecay} shows the absolute value of the spinning soliton profile 
along a  straight line in radial direction, for $d=2$ in (a) and $d=3$ in (b). To be more precise, Figure \ref{fig:QCGLSpinningSolitonDecay} (a)
shows the function
\begin{align}
  \label{equ:Line2D}
  [0,20]\rightarrow\R,\quad r\mapsto \log_{10}\left|\mathbf{w_{\star}}\left(r\cos\frac{\pi}{2},r\sin\frac{\pi}{2}\right)\right|
\end{align}
in case of $d=2$. Similarly, Figure \ref{fig:QCGLSpinningSolitonDecay} (b) shows the function
\begin{align}
  \label{equ:Line3D}
  [0,10\sqrt{3}]\rightarrow\R,\quad r\mapsto \log_{10}\left|\mathbf{w_{\star}}\left(\frac{r}{\sqrt{3}},\frac{r}{\sqrt{3}},\frac{r}{\sqrt{3}}\right)\right|
\end{align}
in case of $d=3$. The functions are almost linear at least in the regions enclosed by the black dashed lines, which are $[5,13]$ for $d=2$ and $[5,9]$ for $d=3$. 
In case $d=2$ the observed NDR is slightly below the TDR. This is attributed to the fact that the NDR is affected by the size of the bounded domain and by the 
choice of boundary conditions. Summarizing, this indicates that the heat kernel estimates from \cite{Otten2014,Otten2014a}, which form the origin of these decay 
rates, are quite accurate.

%Theorem \ref{thm:NonlinearOrnsteinUhlenbeckSteadyState}
%profile (2D): $0.5713$ (NDR), $0.3536$ (TDR)
%profile (3D): $0.5387$ (NDR), $0.3536$ (TDR)
%\eqref{equ:PGeneralBoundGinzburgLandau}
%\eqref{equ:QCGLBound1}
%Theorem \ref{thm:NonlinearOrnsteinUhlenbeckSteadyState}

\begin{figure}[ht]
  \centering
  \subfigure[]{\includegraphics[page=14,height=4.5cm] {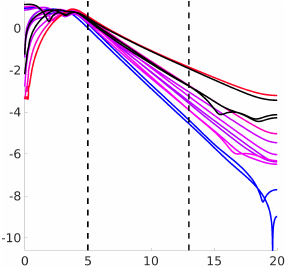} \label{fig:QCGLSpinningSoliton2DEigenfunctionsDecay}}
  \subfigure[]{\includegraphics[page=15,height=4.5cm] {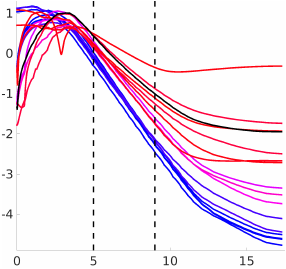} \label{fig:QCGLSpinningSoliton3DEigenfunctionsDecay}}
  \caption{Numerical rate of exponential decay of the eigenfunctions for $d=2$ (a) and $d=3$ (b) of \eqref{equ:ComplexQuinticCubicGinzburgLandauEquationEigenvalueProblemNumerically} linearized 
  at a spinning soliton.}
  \label{fig:QCGLSpinningSolitonDecayEigenfunctions}
\end{figure} 

\textbf{Decay rates of eigenfunctions:} The maximal rate of exponential decay for the eigenfunctions, obtained from Theorem \ref{thm:LinearizedOrnsteinUhlenbeckExponentialDecayInLp}, 
will now depend on $\lambda$, since 
\begin{align*}
  \Re\lambda\geqslant -(1-\varepsilon)\beta_{\infty}=-(1-\varepsilon)(-\Re\delta)\quad\Longleftrightarrow\quad \varepsilon \leqslant \frac{\Re\lambda-\Re\delta}{-\Re\delta}=: \varepsilon(\lambda).
\end{align*}
This gives us the bounds
% \begin{align*}
%   0\leqslant\mu_2\leqslant\varepsilon\frac{\sqrt{\Re\alpha(-\Re\delta)}}{|\alpha|p}\leqslant\mu_{\max,p}^{\mathrm{eig}}(\lambda)
%   <\frac{\Re\lambda-\Re\delta}{\sqrt{-\Re\delta}}\cdot\frac{\sqrt{\Re\alpha}}{|\alpha|p_{\min}}=\mu_{\max}^{\mathrm{eig}}(\lambda)
% \end{align*}
% with $p_{\min}$ from \eqref{equ:PGeneralBoundGinzburgLandau},
\begin{align}
  \label{equ:GeneralMaximalDecayRateEigenfunctions}
0 \leqslant \frac{\varepsilon \nu}{p} \leqslant \frac{\varepsilon(\lambda) \nu}{p} =: \mu^{\mathrm{eig}}(p,\lambda)
< \frac{\varepsilon(\lambda) \nu}{\max\left\{p_{\min},\frac{d}{2}\right\}} =: \mu^{\mathrm{eig}}_{\max}(\lambda).
  % \mu_{\max,p}^{\mathrm{eig}}(\lambda):=\frac{\Re\lambda-\Re\delta}{\sqrt{-\Re\delta}}\cdot\frac{\sqrt{\Re\alpha}}{|\alpha|p}\quad\text{and}\quad
  % \mu_{\max}^{\mathrm{eig}}(\lambda):=\frac{\Re\lambda-\Re\delta}{\sqrt{-\Re\delta}}\cdot\frac{\sqrt{\Re\alpha}(|\alpha|+\Re\alpha)}{2|\alpha|^2}
\end{align}
With parameter values \eqref{equ:ParameterValuesQCGLNr1} the bounds \eqref{equ:GeneralMaximalDecayRateEigenfunctions} lead to
\begin{align*}
  \mu^{\mathrm{eig}}(p,\lambda)=\frac{2\left(\Re\lambda+\frac{1}{2}\right)}{\sqrt{2}p},\quad\mu_{\max}^{\mathrm{eig}}(\lambda)=\frac{\sqrt{2}+1}{2}\left(\Re\lambda+\frac{1}{2}\right).
\end{align*}
This shows, that the decay rate is maximal for eigenvalues on the imaginary axis and decreases linearly to $0$ as $\Re\lambda$ approaches $-\frac{1}{2}$, cf. 
Figure \ref{fig:QCGLSpinningSolitonSpectrum}. Recall, that Theorem \ref{thm:LinearizedOrnsteinUhlenbeckExponentialDecayInLp} does not apply for $\Re\lambda\leqslant-\frac{1}{2}$.
For the isolated eigenvalues labeled by a green square in Figure \ref{fig:QCGLSpinningSolitonSpectrum}, the TDR's $\mu_{\max}^{\mathrm{eig}}(\lambda)$ of the associated eigenfunctions 
are given in the third columns of Table \ref{tab:DecayRatesSpinningSoliton}.

We compare with the numerical exponential decay rates for the eigenfunctions: Figure \ref{fig:QCGLSpinningSolitonDecayEigenfunctions} shows the absolute value 
of the eigenfunctions along the lines from \eqref{equ:Line2D} and \eqref{equ:Line3D} with $\mathbf{w}$ instead of $\mathbf{w_{\star}}$. The eigenfunctions are associated to the  
eigenvalues in green boxes in Figure \ref{fig:QCGLSpinningSolitonSpectrum}. The color of the graphs vary
with $\Re\lambda$ of the associated eigenvalue.
%  which gives us 
% information about their associated TDR, i.e. a blue graph indicates that $\Re\lambda$ is near $0$ which implies a high TDR.
Varying $\Re\lambda$ from $0$ to $-\frac{1}{2}$, the 
graphs change color from blue to red. A red graph indicates that $\Re\lambda$ is near $-\frac{1}{2}$ and that the TDR is small. Finally, a black graph indicates an eigenvalue
$\Re\lambda\leqslant-\frac{1}{2}$, in which case we do not have a TDR.
%The line color is blue for eigenvalues placed on the imaginary axis and the color becomes red as more the eigenvalue approaches the line $\delta_1+i\R=-\frac{1}{2}+i\R$.
All eigenfunctions are approximately linear in the regions  enclosed by the black dashed lines, which are again $[5,13]$ for $d=2$, and $[5,9]$ for $d=3$. Moreover, 
we observe that the decay rate of the eigenfunctions decreases when the eigenvalue moves to the left of the imaginary axis. We note that even those eigenfunctions the 
eigenvalues of which satisfy $\Re\lambda\leqslant-\frac{1}{2}$, have exponential decay. Once more, we used linear regression on $1000$ radially equispaced points to estimate 
the NDR. The numerical values are collected in the second columns of Table \ref{tab:DecayRatesSpinningSoliton}. Again the TDR's are surprisingly close to the NDR's with 
difference increasing towards $\Re\lambda=-\frac{1}{2}$.
%Again  the TDR's and NDR's differ only by a small factor between $1.6157$ and $3.0134$ (for the last eigenvalue it is $12.6440$) 
%for $d=2$, and between $1.5235$ and $34.2562$ (for the last eigenvalue it is $47.2429$) for $d=3$. The factor seems to be small for $\Re\lambda$ near $0$ and becomes larger when $\Re\lambda$ 
%approaches $-\frac{1}{2}$. But the overall impression remains that
%TDR's and NDR's are surprisingly close to each other.

\begin{table}[H]
\centering
\begin{minipage}{0.45\textwidth}
  \begin{tabular}{|r|c|c|}
    \hline
    eigenvalue & NDR & TDR \\
    \hline
    $5.29\cdot 10^{-15}$  & $0.5713$ & $0.6036$ \\ %& $0.3536$ \\
    $-0.00002\pm1.0270i$  & $0.5730$ & $0.6035$ \\ %& $0.3535$ \\
    $-0.17509\pm1.8183i$  & $0.5001$ & $0.3922$ \\ %& $0.2297$ \\
    $-0.17695\pm0.3340i$  & $0.4815$ & $0.3900$ \\ %& $0.2284$ \\
    $-0.19882\pm1.6205i$  & $0.4139$ & $0.3636$ \\ %& $0.2130$ \\
    $-0.21211\pm2.7050i$  & $0.4652$ & $0.3475$ \\ %& $0.2036$ \\
    $-0.22794\pm2.2695i$  & $0.5155$ & $0.3284$ \\ %& $0.1924$ \\
    $-0.26402\pm1.0624i$  & $0.5355$ & $0.2849$ \\ %& $0.1669$ \\
    $-0.31017\pm3.5224i$  & $0.4044$ & $0.2291$ \\ %& $0.1342$ \\
    $-0.46659\pm4.2742i$  & $0.2984$ & $0.0403$ \\ %& $0.0236$ \\
    $-0.54131\pm2.8166i$  & $0.2972$ & --- \\
    $-0.60226\pm0.6492i$  & $0.3982$ & --- \\
    $-0.67248\pm3.6064i$  & $0.3889$ & --- \\
    \hline
  \end{tabular}
  \vfill
\end{minipage}\hfill
\begin{minipage}{0.45\textwidth}
  \begin{tabular}{|r|c|c|}
  \hline
  eigenvalue & NDR & TDR\\
  \hline
  $8.999\cdot 10^{-15}$   & $0.5387$ & $0.4714$ \\ %& $0.3536$ \\
  $-5.6162\cdot 10^{-4}$  & $0.5478$ & $0.4714$ \\ %& $0.3536$ \\
  $0.00110\pm0.68827i$    & $0.5507$ & $0.4714$ \\ %& $0.3536$ \\
  $0.00248\pm0.6874i$     & $0.5398$ & $0.4714$ \\ %& $0.3536$ \\  
  $-0.06622\pm1.0112i$    & $0.4899$ & $0.4090$ \\ %& $0.3067$ \\
  $-0.07747\pm1.5274i$    & $0.5355$ & $0.3984$ \\ %& $0.2988$ \\ 
  $-0.22334\pm1.1593i$    & $0.4756$ & $0.2608$ \\ %& $0.1956$ \\
  $-0.26467\pm0.1193i$    & $0.4785$ & $0.2219$ \\ %& $0.1664$ \\
  $-0.30232\pm1.9457i$    & $0.4649$ & $0.1864$ \\ %& $0.1398$ \\
  $-0.43957\pm2.3248i$    & $0.3595$ & $0.0570$ \\ %& $0.0427$ \\
  $-0.44063\pm1.5128i$    & $0.3310$ & $0.0560$ \\ %& $0.0420$ \\
  $-0.47366\pm1.3552i$    & $0.4781$ & $0.0248$ \\ %& $0.0186$ \\
  $-0.48294\pm0.9163i$    & $0.4145$ & $0.0161$ \\ %& $0.0121$ \\
  $-0.48506\pm0.0991i$    & $0.2126$ & $0.0141$ \\ %& $0.0106$ \\
  $-0.49015\pm0.2535i$    & $0.3307$ & $0.0093$ \\ %& $0.0070$ \\
  $-0.55519\pm1.1222i$    & $0.3581$ & --- \\
  \hline
 \end{tabular}
\end{minipage}
\caption{Numerical (NDR) and theoretical (TDR) exponential decay rates of QCGL \eqref{equ:ComplexQuinticCubicGinzburgLandauEquation} for the eigenfunctions of 
 the linearization at a spinning soliton for $d=2$ (left) and $d=3$ (right).}
\label{tab:DecayRatesSpinningSoliton}
\end{table}

\def\cprime{$'$}


\begin{thebibliography}{10}

\bibitem{ComsolMultiphysics52}
\textsc{Comsol Multiphysics 5.2}, http://www.comsol.com, 2015.

\bibitem{Adams1975}
R.~A. Adams.
\newblock {\em {S}obolev {S}paces}, volume~65 of {\em Pure and applied
  mathematics ; 65}.
\newblock Acad. Press, New York [u.a.], 1975.

\bibitem{AfanasjevAkhmedievSoto-Crespo1996}
V.~V. Afanasjev, N.~Akhmediev, and J.~M. Soto-Crespo.
\newblock Three forms of localized solutions of the quintic complex
  ginzburg-landau equation.
\newblock {\em Phys. Rev. E}, 53:1931--1939, Feb 1996.

\bibitem{Alt2006}
H.~W. Alt.
\newblock {\em Lineare Funktionalanalysis}.
\newblock Springer-Verlag Berlin Heidelberg, Berlin, Heidelberg, 2006.

\bibitem{BeynLorenz2008}
W.-J. Beyn and J.~Lorenz.
\newblock Nonlinear stability of rotating patterns.
\newblock {\em Dyn. Partial Differ. Equ.}, 5(4):349--400, 2008.

\bibitem{BeynLorenz2014}
W.-J. Beyn and J.~Lorenz.
\newblock {R}otating {P}atterns on {F}inite {D}isks.
\newblock {\em unpublished}, 2014.

\bibitem{BeynOttenRottmannMatthes2013}
W.-J. Beyn, D.~Otten, and J.~Rottmann-Matthes.
\newblock Stability and {C}omputation of {D}ynamic {P}atterns in {PDE}s.
\newblock In {\em Current Challenges in Stability Issues for Numerical
  Differential Equations}, Lecture Notes in Mathematics, pages 89--172.
  Springer International Publishing, 2014.

\bibitem{BeynThuemmler2004}
W.-J. Beyn and V.~Th{\"u}mmler.
\newblock Freezing solutions of equivariant evolution equations.
\newblock {\em SIAM J. Appl. Dyn. Syst.}, 3(2):85--116 (electronic), 2004.

\bibitem{BeynThuemmler2007}
W.-J. Beyn and V.~Th{\"u}mmler.
\newblock Phase conditions, symmetries and {PDE} continuation.
\newblock In {\em Numerical continuation methods for dynamical systems},
  Underst. Complex Syst., pages 301--330. Springer, Dordrecht, 2007.

\bibitem{BeynThuemmler2009}
W.-J. Beyn and V.~Th{\"u}mmler.
\newblock Dynamics of patterns in nonlinear equivariant {PDE}s.
\newblock {\em GAMM-Mitt.}, 32(1):7--25, 2009.

\bibitem{CialdeaMazya2005}
A.~Cialdea and V.~Maz'ya.
\newblock Criterion for the {$L^p$}-dissipativity of second order differential
  operators with complex coefficients.
\newblock {\em J. Math. Pures Appl. (9)}, 84(8):1067--1100, 2005.

\bibitem{CialdeaMazya2009}
A.~Cialdea and V.~Maz'ya.
\newblock Criteria for the {$L^p$}-dissipativity of systems of second order
  differential equations.
\newblock {\em Ric. Mat.}, 55(2):233--265, 2006.

\bibitem{CrasovanMalomedMihalache2000}
L.-C. Crasovan, B.~A. Malomed, and D.~Mihalache.
\newblock Stable vortex solitons in the two-dimensional ginzburg-landau
  equation.
\newblock {\em Phys. Rev. E}, 63:016605, Dec 2000.

\bibitem{CrasovanMalomedMihalache2001}
L.-C. Crasovan, B.~A. Malomed, and D.~Mihalache.
\newblock {Spinning solitons in cubic-quintic nonlinear media}.
\newblock {\em Pramana-journal of Physics}, 57:1041--1059, 2001.

\bibitem{EngelNagel2000}
K.-J. Engel and R.~Nagel.
\newblock {\em One-parameter semigroups for linear evolution equations}, volume
  194 of {\em Graduate Texts in Mathematics}.
\newblock Springer-Verlag, New York, 2000.
\newblock With contributions by S. Brendle, M. Campiti, T. Hahn, G. Metafune,
  G. Nickel, D. Pallara, C. Perazzoli, A. Rhandi, S. Romanelli and R.
  Schnaubelt.

\bibitem{FiedlerScheel2003}
B.~Fiedler and A.~Scheel.
\newblock Spatio-temporal dynamics of reaction-diffusion patterns.
\newblock In {\em Trends in nonlinear analysis}, pages 23--152. Springer,
  Berlin, 2003.

\bibitem{Gustafson1968}
K.~Gustafson.
\newblock The angle of an operator and positive operator products.
\newblock {\em Bull. Amer. Math. Soc.}, 74:488--492, 1968.

\bibitem{Gustafson2012}
K.~Gustafson.
\newblock {\em Antieigenvalue analysis : with applications to numerical
  analysis, wavelets, statistics, quantum mechanics, finance and optimization}.
\newblock World Scientific, Hackensack, NJ [u.a.], 2012.

\bibitem{GinzburgLandau1950}
L.~D. Landau and V.~L. Ginzburg.
\newblock On the theory of superconductivity.
\newblock {\em Journal of Experimental and Theoretical Physics (USSR)},
  20:1064, 1950.

\bibitem{MetafunePallaraVespri2005}
G.~Metafune, D.~Pallara, and V.~Vespri.
\newblock {$L^p$}-estimates for a class of elliptic operators with unbounded
  coefficients in {$\mathbf R^N$}.
\newblock {\em Houston J. Math.}, 31(2):605--620 (electronic), 2005.

\bibitem{Mielke2002}
A.~Mielke.
\newblock The {G}inzburg-{L}andau equation in its role as a modulation
  equation.
\newblock In {\em Handbook of dynamical systems, {V}ol. 2}, pages 759--834.
  North-Holland, Amsterdam, 2002.

\bibitem{MihalacheMaziluCrasovanMalomedLederer2000}
D.~Mihalache, D.~Mazilu, L.-C. Crasovan, B.~A. Malomed, and F.~Lederer.
\newblock Three-dimensional spinning solitons in the cubic-quintic nonlinear
  medium.
\newblock {\em Phys. Rev. E}, 61:7142--7145, Jun 2000.

\bibitem{Moores1993}
J.~D. Moores.
\newblock On the ginzburg-landau laser mode-locking model with fifth-order
  saturable absorber term.
\newblock {\em Optics Communications}, 96(1–3):65--70, 1993.

\bibitem{Nikolskij1975}
S.~M. Nikol'skij.
\newblock {\em Approximation of {F}unctions of {S}everal {V}ariables and
  {I}mbedding {T}heorems}, volume 205 of {\em Die Grundlehren der
  mathematischen Wissenschaften ; 205}.
\newblock Springer, Berlin [u.a.], 1975.

\bibitem{OlverLozierBoisvertClark2010}
F.~W.~J. Olver, D.~W. Lozier, R.~F. Boisvert, and C.~W. Clark, editors.
\newblock {\em N{IST} handbook of mathematical functions}.
\newblock U.S. Department of Commerce National Institute of Standards and
  Technology, Washington, DC, 2010.
\newblock With 1 CD-ROM (Windows, Macintosh and UNIX).

\bibitem{Otten2014}
D.~Otten.
\newblock {\em Spatial decay and spectral properties of rotating waves in
  parabolic systems}.
\newblock PhD thesis, Bielefeld University, 2014,
  \normalfont{\url{www.math.uni-bielefeld.de/~dotten/files/diss/Diss_DennyOtte%
n.pdf}}.
\newblock Shaker Verlag, Aachen.

\bibitem{Otten2014a}
D.~Otten.
\newblock {E}xponentially weighted resolvent estimates for complex
  {O}rnstein-{U}hlenbeck systems.
\newblock {\em J. Evol. Equ.}, 15(4):753--799, 2015.

\bibitem{Otten2015a}
D.~Otten.
\newblock The {I}dentification {P}roblem for complex-valued
  {O}rnstein-{U}hlenbeck {O}perators in {$L^p(\mathbb R^d,\mathbb C^N)$}.
\newblock {\em Preprint, \normalfont{\url{http://arxiv.org/abs/1510.00827}}},
  2015 (submitted).

\bibitem{Otten2015b}
D.~Otten.
\newblock A new ${L}^p$-{A}ntieigenvalue {C}ondition for {O}rnstein-{U}hlenbeck
  {O}perators.
\newblock {\em Preprint, \normalfont{\url{http://arxiv.org/abs/1510.00864}}},
  2015 (submitted).

\bibitem{RosanovFedorovShatsev2006}
N.~Rosanov, S.~Fedorov, and A.~Shatsev.
\newblock Motion of clusters of weakly coupled two-dimensional cavity solitons.
\newblock {\em Journal of Experimental and Theoretical Physics}, 102:547--555,
  2006.

\bibitem{SandstedeScheel2000}
B.~Sandstede and A.~Scheel.
\newblock Absolute versus convective instability of spiral waves.
\newblock {\em Phys. Rev. E (3)}, 62(6, part A):7708--7714, 2000.

\bibitem{SandstedeScheel2001}
B.~Sandstede and A.~Scheel.
\newblock Superspiral structures of meandering and drifting spiral waves.
\newblock {\em Phys. Rev. Lett.}, 86:171--174, Jan 2001.

\bibitem{Soto-Crespo2009}
J.~Soto-Crespo, N.~Akhmediev, C.~Mej\'{i}a-Cort\'{e}s, and N.~Devine.
\newblock Dissipative ring solitons with vorticity.
\newblock {\em Opt. Express}, 17(6):4236--4250, Mar 2009.

\bibitem{AkhmedievAnkiewiczSoto-Crespo2000}
J.~M. Soto-Crespo, N.~Akhmediev, and A.~Ankiewicz.
\newblock Pulsating, creeping, and erupting solitons in dissipative systems.
\newblock {\em Phys. Rev. Lett.}, 85:2937--2940, Oct 2000.

\bibitem{Tao2009}
T.~Tao.
\newblock Real analysis.
\newblock {\em {L}ecture {N}otes,
  \normalfont{\url{https://terrytao.wordpress.com/2009/04/30/245c-notes-4-sobo%
lev-spaces/}}}, 2009.

\bibitem{ThualFauve1988}
O.~Thual and S.~Fauve.
\newblock Localized structures generated by subcritical instabilities.
\newblock {\em J. Phys. France}, 49(11):1829--1833, 1988.

\bibitem{Thuemmler2006}
V.~Th{\"u}mmler.
\newblock Numerical bifurcation analysis of relative equilibria with {F}emlab.
\newblock {\em in Proceedings of the COMSOL Users Conference (Comsol
  Anwenderkonferenz), Frankfurt, Femlab GmbH, Goettingen, Germany}, 2006.

\bibitem{TrilloTorruellas2010}
S.~Trillo and W.~Torruellas.
\newblock {\em Spatial Solitons}.
\newblock Springer Series in Optical Sciences. Springer, 2010.

\bibitem{VanSaarloosHohenberg1992}
W.~van Saarloos and P.~C. Hohenberg.
\newblock Fronts, pulses, sources and sinks in generalized complex
  {G}inzburg-{L}andau equations.
\newblock {\em Phys. D}, 56(4):303--367, 1992.

\bibitem{ZelikMielke2009}
S.~Zelik and A.~Mielke.
\newblock Multi-pulse evolution and space-time chaos in dissipative systems.
\newblock {\em Mem. Amer. Math. Soc.}, 198(925):vi+97, 2009.

\bibitem{Ziemer1989}
W.~P. Ziemer.
\newblock {\em Weakly differentiable functions}, volume 120 of {\em Graduate
  Texts in Mathematics}.
\newblock Springer-Verlag, New York, 1989.
\newblock Sobolev spaces and functions of bounded variation.

\end{thebibliography}
\end{document}